\documentclass{amsart}
\usepackage{amsmath}
\usepackage{amssymb}
\usepackage{amsthm}
\usepackage{pict2e}
\usepackage{graphicx,color}

\newtheorem{thm}{Theorem}[section]
\newtheorem{prop}[thm]{Proposition}

\newtheorem{lem}[thm]{Lemma}
\newtheorem{fact}{Fact}

\theoremstyle{definition}
\newtheorem{defin}[thm]{Definition}
\newtheorem{exa}[thm]{Example}
\newtheorem{notation}{Notation}

\theoremstyle{remark}
\newtheorem{rem}[thm]{Remark}

\newcommand{\nucircle}{
\begin{picture}(5,5) \put(3,3){\circle{10}} \put(0,1){$\nu$} \end{picture}\ 
}
\newcommand{\mucircle}{
\begin{picture}(5,5) \put(3,3){\circle{10}} \put(0,1){$\mu$} \end{picture}\ 
}
\newcommand{\epcircle}{\begin{picture}(5,5) \put(3,3){\circle{10}} \put(0,1){$\epsilon_i$} \end{picture}~}
\newcommand{\I}{\mathcal{I}}
\newcommand{\J}{\mathcal{J}}

\newcommand{\zpppu}{\begin{picture}(0,0)
\put(28,-22){\line(1,0){10}}
\put(10,-12){\line(1,0){10}}
\put(10,-12){\cb}
\put(20,-12){\cb}
\put(29,7){\line(1,0){10}}
\put(28,-22){\cb}
\put(38,-22){\cb}
\put(29,7){\cb}
\put(39,7){\cb}
\end{picture}
\DIII}

\newcommand{\zmppu}{\begin{picture}(0,0)
\put(28,-22){\line(1,0){10}}
\put(14,-17){\line(0,1){10}}
\put(29,7){\line(1,0){10}}
\put(28,-22){\cb}
\put(38,-22){\cb}
\put(14,-17){\cb}
\put(14,-7){\cb}
\put(29,7){\cb}
\put(39,7){\cb}
\put(30,-10){$-$}
\end{picture}
\DIII}

\newcommand{\zmppde}{\qquad \begin{picture}(0,0)
\put(28,-22){\line(1,0){10}}
\put(31,-32){$\widetilde{q_u}$}
\put(-20,-13){$(r : p)_u$}
\put(27,27){$(p:r)_u$}
\put(14,-17){\line(0,1){10}}
\put(29,7){\line(1,0){10}}
\put(28,-22){\cb}
\put(38,-22){\cb}
\put(14,-17){\cb}
\put(14,-7){\cb}
\put(29,7){\cb}
\put(39,7){\cb}
\put(30,-10){$-$}
\end{picture}
\DIII}
\newcommand{\zmppdeSnd}{\qquad \begin{picture}(0,0)
\put(28,-22){\line(1,0){10}}
\put(31,-32){$\widetilde{q_u}$}
\put(-20,-13){$r_u : p_u$}
\put(27,27){$p_u : r_u$}
\put(14,-17){\line(0,1){10}}
\put(29,7){\line(1,0){10}}
\put(28,-22){\cb}
\put(38,-22){\cb}
\put(14,-17){\cb}
\put(14,-7){\cb}
\put(29,7){\cb}
\put(39,7){\cb}
\put(30,-10){$-$}
\end{picture}
\DIII}

\newcommand{\zmppdSnd}{\begin{picture}(0,0)
\put(28,-22){\line(1,0){10}}
\put(31,-32){$\tilde{q}$}
\put(-7,-15){$r:p$}
\put(30,20){$p:r$}
\put(14,-17){\line(0,1){10}}
\put(29,7){\line(1,0){10}}
\put(28,-22){\cb}
\put(38,-22){\cb}
\put(14,-17){\cb}
\put(14,-7){\cb}
\put(29,7){\cb}
\put(39,7){\cb}
\put(30,-10){$-$}
\end{picture}
\DIII}

\newcommand{\zmmpu}{\begin{picture}(0,0)
\put(34,3){\line(0,1){10}}
\put(34,3){\cb}
\put(34,13){\cb}
\put(28,-22){\line(1,0){10}}
\put(14,-17){\line(0,1){10}}
\put(28,-22){\cb}
\put(38,-22){\cb}
\put(14,-17){\cb}
\put(14,-7){\cb}
\end{picture}
\DIII}

\newcommand{\zpmpe}{\begin{picture}(0,0)
\put(31,-32){$q_u$}
\put(15,5){$r_u$}
\put(45,10){$p_u$}
\put(30,-22){\line(1,0){10}}
\put(10,-12){\line(1,0){10}}
\put(34,3){\line(0,1){10}}
\put(30,-22){\cb}
\put(40,-22){\cb}
\put(10,-12){\cb}
\put(20,-12){\cb}
\put(34,3){\cb}
\put(34,13){\cb}
\end{picture}
\DIII
}

\newcommand{\zpmp}{\begin{picture}(0,0)
\put(33,-30){$q$}
\put(15,5){$r$}
\put(44,10){$p$}
\put(30,-22){\line(1,0){10}}
\put(10,-12){\line(1,0){10}}
\put(34,3){\line(0,1){10}}
\put(30,-22){\cb}
\put(40,-22){\cb}
\put(10,-12){\cb}
\put(20,-12){\cb}
\put(34,3){\cb}
\put(34,13){\cb}
\end{picture}
\DIII
}
\newcommand{\zeem}{\begin{picture}(0,0)
\put(34,-28){\line(0,1){10}}
\put(34,-28){\cb}
\put(34,-18){\cb}
\end{picture}
\DIII}

\newcommand{\zmpmSnd}{\begin{picture}(0,0)
\put(36,-34){$\tilde{q}$}
\put(-8,-15){$r:p$}
\put(30,20){$p:r$}
\put(29,7){\line(1,0){10}}
\put(29,7){\cb}
\put(39,7){\cb}
\put(14,-17){\line(0,1){10}}
\put(14,-17){\cb}
\put(14,-7){\cb}
\put(34,-28){\line(0,1){10}}
\put(34,-28){\cb}
\put(34,-18){\cb}
\end{picture}
\DIII}

\newcommand{\zpmmSnd}{\begin{picture}(0,0)
\put(3,-30){$q:p$}
\put(10,0){$\tilde{r}$}
\put(43,10){$p:q$}
\put(10,-12){\line(1,0){10}}
\put(10,-12){\cb}
\put(20,-12){\cb}
\put(34,3){\line(0,1){10}}
\put(34,3){\cb}
\put(34,13){\cb}
\put(34,-28){\line(0,1){10}}
\put(34,-28){\cb}
\put(34,-18){\cb}
\end{picture}
\DIII}

\newcommand{\pmm}{ 
\begin{minipage}{60pt}
        \begin{picture}(50,70)
            \cbezier(0,0)(0,3)(7,7)(12,12)
            \qbezier(19,19)(25,25)(31,31)
            \cbezier(38,38)(45,45)(50,50)(50,56)
            \cbezier(0,56)(0,50)(17.25,44.25)(34.5,34.5)
            \cbezier(34.5,34.5)(44,30)(50,20)(50,0)
            \cbezier(17,48)(21,53)(24,54)(24,56)
            \cbezier(12,41)(-1,25)(24,10)(24,0)
            {\color{black}
            {\put(15,50){\circle*{3}}
\put(15,40){\circle*{3}}
\put(15,40){\line(0,1){10}}
            }}
            {\color{black}{\put(8,15){\circle*{3}}
\put(18,15){\circle*{3}}
\put(9,15){\line(1,0){10}}}}
{\color{black}{\put(27,39){\circle*{3}}
\put(27,29){\circle*{3}}
\put(27,29){\line(0,1){10}}}}
\put(4,3){$q$}
\put(35,34){$p$}
\put(0,50){$r$}
        \end{picture}
    \end{minipage}}

\newcommand{\pmmbSnd}{ 
\begin{minipage}{60pt}
        \begin{picture}(50,70)
            \cbezier(0,0)(0,3)(7,7)(12,12)
            \qbezier(19,19)(25,25)(31,31)
            \cbezier(38,38)(45,45)(50,50)(50,56)
            \cbezier(0,56)(0,50)(17.25,44.25)(34.5,34.5)
            \cbezier(34.5,34.5)(44,30)(50,20)(50,0)
            \cbezier(17,48)(21,53)(24,54)(24,56)
            \cbezier(12,41)(-1,25)(24,10)(24,0)
            {\color{black}
            {\put(15,50){\circle*{3}}
\put(15,40){\circle*{3}}
\put(15,40){\line(0,1){10}}
            }}
            {\color{black}{\put(8,15){\circle*{3}}
\put(18,15){\circle*{3}}
\put(9,15){\line(1,0){10}}}}
{\color{black}{\put(27,39){\circle*{3}}
\put(27,29){\circle*{3}}
\put(27,29){\line(0,1){10}}}}
\put(-3,0){$q : p$}
\put(35,34){$p : q$}
\put(0,50){$\tilde{r}$}
        \end{picture}
    \end{minipage}}

    \newcommand{\mpmbSnd}{\begin{minipage}{60pt}
        \begin{picture}(50,70)
            \cbezier(0,0)(0,3)(7,7)(12,12)
            \qbezier(19,19)(25,25)(31,31)
            \cbezier(38,38)(45,45)(50,50)(50,56)
            \cbezier(0,56)(0,50)(17.25,44.25)(34.5,34.5)
            \cbezier(34.5,34.5)(44,30)(50,20)(50,0)
            \cbezier(17,48)(21,53)(24,54)(24,56)
            \cbezier(12,41)(-1,25)(24,10)(24,0)
{\color{black}{\put(16,21){\circle*{3}}
\put(16,10){\circle*{3}}
\put(16,11){\line(0,1){10}}}}
{\color{black}
            {\put(10,50){\circle*{3}}
\put(10,40){\circle*{3}}
\put(10,40){\line(0,1){10}}
            }}
            {\color{black}
            {\put(32.5,34.5){\circle*{3}}
\put(20.5,34.5){\circle*{3}}
\put(21.5,34.5){\line(1,0){10}}
            }}
            \put(15,14){$\tilde{q}$}
            \put(22,50){$p : r$}
            \put(-17,40){$r : p$}
        \end{picture}
    \end{minipage}}

    \newcommand{\ppmbSnd}{ 
\begin{minipage}{60pt}
        \begin{picture}(50,70)
            \cbezier(0,0)(0,3)(7,7)(12,12)
            \qbezier(19,19)(25,25)(31,31)
            \cbezier(38,38)(45,45)(50,50)(50,56)
            \cbezier(0,56)(0,50)(17.25,44.25)(34.5,34.5)
            \cbezier(34.5,34.5)(44,30)(50,20)(50,0)
            \cbezier(17,48)(21,53)(24,54)(24,56)
            \cbezier(12,41)(-1,25)(24,10)(24,0)
            {\color{black}
            {\put(15,50){\circle*{3}}
\put(15,40){\circle*{3}}
\put(15,40){\line(0,1){10}}
            }}
            {\color{black}{\put(8,15){\circle*{3}}
\put(18,15){\circle*{3}}
\put(9,15){\line(1,0){10}}}}
{\color{black}{\put(32,34){\circle*{3}}
\put(22,34){\circle*{3}}
\put(22,34){\line(1,0){10}}}}
\put(4,3){$\tilde{q}$}
\put(23,50){$p : r$}
\put(-15,41){$r : p$}
        \end{picture}
    \end{minipage}}
\newcommand{\ppp}{\begin{minipage}{60pt}
            \begin{picture}(50,70)
            \cbezier(0,0)(0,3)(7,7)(12,12)
            \qbezier(19,19)(25,25)(31,31)
            \cbezier(38,38)(45,45)(50,50)(50,56)
            \cbezier(0,56)(0,50)(17.25,44.25)(34.5,34.5)
            \cbezier(34.5,34.5)(44,30)(50,20)(50,0)
            \cbezier(17,48)(21,53)(24,54)(24,56)
            \cbezier(12,41)(-1,25)(24,10)(24,0)
{\color{black}
            {\put(8,44.5){\circle*{3}}
\put(20,44.5){\circle*{3}}
\put(9,44.5){\line(1,0){10}}
            }}
            {\color{black}
            {\put(37.5,34.5){\circle*{3}}
\put(25.5,34.5){\circle*{3}}
\put(26.5,34.5){\line(1,0){10}}
            }}
            {\color{black}{\put(4,15){\circle*{3}}
\put(14,15){\circle*{3}}
\put(5,15){\line(1,0){10}}}}
\put(4,0){$q$}
\put(38,33){$p$}
\put(5,50){$r$}
        \end{picture}
    \end{minipage}}
\newcommand{\pppUU}{\begin{minipage}{60pt}
            \begin{picture}(50,70)
            \cbezier(0,0)(0,3)(7,7)(12,12)
            \qbezier(19,19)(25,25)(31,31)
            \cbezier(38,38)(45,45)(50,50)(50,56)
            \cbezier(0,56)(0,50)(17.25,44.25)(34.5,34.5)
            \cbezier(34.5,34.5)(44,30)(50,20)(50,0)
            \cbezier(17,48)(21,53)(24,54)(24,56)
            \cbezier(12,41)(-1,25)(24,10)(24,0)
{\color{black}
            {\put(8,44.5){\circle*{3}}
\put(20,44.5){\circle*{3}}
\put(9,44.5){\line(1,0){10}}
            }}
            {\color{black}
            {\put(37.5,34.5){\circle*{3}}
\put(25.5,34.5){\circle*{3}}
\put(26.5,34.5){\line(1,0){10}}
            }}
            {\color{black}{\put(4,15){\circle*{3}}
\put(14,15){\circle*{3}}
\put(5,15){\line(1,0){10}}}}
\put(4,0){$q_u$}
\put(38,33){$p_u$}
\put(0,53){$r_u$}
        \end{picture}
    \end{minipage}}
\newcommand{\pppu}{\begin{minipage}{60pt}
            \begin{picture}(50,70)
            \cbezier(0,0)(0,3)(7,7)(12,12)
            \qbezier(19,19)(25,25)(31,31)
            \cbezier(38,38)(45,45)(50,50)(50,56)
            \cbezier(0,56)(0,50)(17.25,44.25)(34.5,34.5)
            \cbezier(34.5,34.5)(44,30)(50,20)(50,0)
            \cbezier(17,48)(21,53)(24,54)(24,56)
            \cbezier(12,41)(-1,25)(24,10)(24,0)
{\color{black}
            {\put(8,44.5){\circle*{3}}
\put(20,44.5){\circle*{3}}
\put(9,44.5){\line(1,0){10}}
            }}
            {\color{black}
            {\put(37.5,34.5){\circle*{3}}
\put(25.5,34.5){\circle*{3}}
\put(26.5,34.5){\line(1,0){10}}
            }}
            {\color{black}{\put(4,15){\circle*{3}}
\put(14,15){\circle*{3}}
\put(5,15){\line(1,0){10}}}}
        \end{picture}
    \end{minipage}}
\newcommand{\pmpb}{ 
    \begin{picture}(0,0)
\put(10,-33){$q$}
\put(41,2){$p$}
\put(11,17){$r$}
        \end{picture}
    \SpmpM}
\newcommand{\pmpbU}{ 
    \begin{picture}(0,0)
\put(10,-33){$q_u$}
\put(41,2){$p_u$}
\put(8,20){$r_u$}
        \end{picture}
    \SpmpM}
\newcommand{\pmpbc}{ 
    \begin{picture}(0,0)
\put(10,-33){$\nu$}
\put(41,2){$\mu$}
\put(11,17){$\lambda$}
        \end{picture}
    \SpmpM}

\newcommand{\pmpbdSnd}{ 
    \begin{picture}(0,0)
\put(3,-32){$q : p$}
\put(41,2){$p : q$}
\put(11,17){$\tilde{r}$}
        \end{picture}
    \SpmpM}
\newcommand{\pmpbde}{ 
    \begin{picture}(0,0)
\put(0,-40){$(q:p)_u$}
\put(50,15){$(p:q)_u$}
\put(8,20){$r_u$}
        \end{picture}
    \SpmpM}
\newcommand{\pmpbdeSnd}{ 
    \begin{picture}(0,0)
\put(0,-40){$q_u:p_u$}
\put(50,15){$p_u:q_u$}
\put(8,20){$r_u$}
        \end{picture}
    \SpmpM}

\newcommand{\pmpbfSnd}{ 
    \begin{picture}(0,0)
\put(-3,-42){$(q : p) : (p : q)$}
\put(49,16){~$(p : q) : (q : p)$}
\put(11,17){$\widetilde{\tilde{r}}$}
        \end{picture}
    \SpmpM}
\newcommand{\pmpbe}{ 
    \begin{picture}(0,0)
\put(0,-40){$m(+ : q)$}
\put(41,2){$\tilde{p}$}
\put(11,17){$\tilde{r}$}
        \end{picture}
    \SpmpM}
\newcommand{\pmp}{
    \begin{picture}(0,0)
\put(10,-33){$q$}
\put(41,2){$p$}
\put(11,17){$r$}
        \end{picture}
    \Spmp}
\newcommand{\mpp}{ 
    \begin{picture}(0,0)
            \put(20,-18){$q$}
            \put(50,15){$p$}
            \put(9,18){$r$}
        \end{picture}\Smpp}
\newcommand{\mppe}{ 
    \begin{picture}(0,0)
            \put(20,-18){$q_u$}
            \put(50,15){$p_u$}
            \put(7,20){$r_u$}
        \end{picture}\Smpp}

\newcommand{\mppbSnd}{ 
    \begin{picture}(0,0)
            \put(20,-18){$q : p$}
            \put(50,15){$p : q$}
            \put(9,18){$\tilde{r}$}
        \end{picture}\Smpp}
\newcommand{\mmp}{\begin{minipage}{60pt}
        \begin{picture}(50,70)
            \cbezier(0,0)(0,3)(7,7)(12,12)
            \qbezier(19,19)(25,25)(31,31)
            \cbezier(38,38)(45,45)(50,50)(50,56)
            \cbezier(0,56)(0,50)(17.25,44.25)(34.5,34.5)
            \cbezier(34.5,34.5)(44,30)(50,20)(50,0)
            \cbezier(17,48)(21,53)(24,54)(24,56)
            \cbezier(12,41)(-1,25)(24,10)(24,0)
            {\color{black}{\put(16,21){\circle*{3}}
\put(16,10){\circle*{3}}
\put(16,11){\line(0,1){10}}}}
{\color{black}
            {\put(4.5,44.5){\circle*{3}}
\put(16.5,44.5){\circle*{3}}
\put(5.5,44.5){\line(1,0){10}}
            }}
            {\color{black}{\put(27,39){\circle*{3}}
\put(27,29){\circle*{3}}
\put(27,29){\line(0,1){10}}}}
            \put(4,0){$q$}
            \put(35,33){$p$}
            \put(5,50){$r$}
        \end{picture}
    \end{minipage}}
\newcommand{\mmpbSnd}{\begin{minipage}{60pt}
        \begin{picture}(50,70)
            \cbezier(0,0)(0,3)(7,7)(12,12)
            \qbezier(19,19)(25,25)(31,31)
            \cbezier(38,38)(45,45)(50,50)(50,56)
            \cbezier(0,56)(0,50)(17.25,44.25)(34.5,34.5)
            \cbezier(34.5,34.5)(44,30)(50,20)(50,0)
            \cbezier(17,48)(21,53)(24,54)(24,56)
            \cbezier(12,41)(-1,25)(24,10)(24,0)
            {\color{black}{\put(16,21){\circle*{3}}
\put(16,10){\circle*{3}}
\put(16,11){\line(0,1){10}}}}
{\color{black}
            {\put(4.5,44.5){\circle*{3}}
\put(16.5,44.5){\circle*{3}}
\put(5.5,44.5){\line(1,0){10}}
            }}
            {\color{black}{\put(27,39){\circle*{3}}
\put(27,29){\circle*{3}}
\put(27,29){\line(0,1){10}}}}
            \put(-3,0){$q : p$}
            \put(35,33){$p : q$}
            \put(5,50){$\tilde{r}$}
        \end{picture}
    \end{minipage}}
\newcommand{\mmpu}{\begin{minipage}{60pt}
        \begin{picture}(50,70)
            \cbezier(0,0)(0,3)(7,7)(12,12)
            \qbezier(19,19)(25,25)(31,31)
            \cbezier(38,38)(45,45)(50,50)(50,56)
            \cbezier(0,56)(0,50)(17.25,44.25)(34.5,34.5)
            \cbezier(34.5,34.5)(44,30)(50,20)(50,0)
            \cbezier(17,48)(21,53)(24,54)(24,56)
            \cbezier(12,41)(-1,25)(24,10)(24,0)
            {\color{black}{\put(16,21){\circle*{3}}
\put(16,10){\circle*{3}}
\put(16,11){\line(0,1){10}}}}
{\color{black}
            {\put(4.5,44.5){\circle*{3}}
\put(16.5,44.5){\circle*{3}}
\put(5.5,44.5){\line(1,0){10}}
            }}
            {\color{black}{\put(27,39){\circle*{3}}
\put(27,29){\circle*{3}}
\put(27,29){\line(0,1){10}}}}
        \end{picture}
    \end{minipage}}
\newcommand{\eem}{\begin{minipage}{60pt}
        \begin{picture}(50,70)
            \cbezier(0,0)(0,3)(7,7)(12,12)
            \qbezier(19,19)(25,25)(31,31)
            \cbezier(38,38)(45,45)(50,50)(50,56)
            \cbezier(0,56)(0,50)(17.25,44.25)(34.5,34.5)
            \cbezier(34.5,34.5)(44,30)(50,20)(50,0)
            \cbezier(17,48)(21,53)(24,54)(24,56)
            \cbezier(12,41)(-1,25)(24,10)(24,0)
          {\color{black}{\put(14,49){\circle*{3}}
\put(14,39){\circle*{3}}
\put(14,39){\line(0,1){10}}}}
        \end{picture}
    \end{minipage}}
\newcommand{\twopp}{\begin{minipage}{30pt}
        \begin{picture}(30,30)
            \qbezier(0,0)(40,15)(0,30)
            \qbezier(11,21)(1,15)(11,9)
            \qbezier(18,24)(24,27)(30,30)
            \qbezier(18,5)(24,2.5)(30,0)
            {\color{black}
            {\put(8,22.5){\circle*{3}}
\put(20,22.5){\circle*{3}}
\put(9,22.5){\line(1,0){10}}
            }}
            {\color{black}{\put(10,12){\circle*{3}}
\put(10,3){\circle*{3}}
\put(10,4){\line(0,1){8}}}}
\put(10,27){$p$}
\put(7,-5){$q$}
        \end{picture}
    \end{minipage}}
\newcommand{\twoppu}{\begin{minipage}{30pt}
        \begin{picture}(30,30)
            \qbezier(0,0)(40,15)(0,30)
            \qbezier(11,21)(1,15)(11,9)
            \qbezier(18,24)(24,27)(30,30)
            \qbezier(18,5)(24,2.5)(30,0)
            {\color{black}
            {\put(8,22.5){\circle*{3}}
\put(20,22.5){\circle*{3}}
\put(9,22.5){\line(1,0){10}}
            }}
            {\color{black}{\put(10,12){\circle*{3}}
\put(10,3){\circle*{3}}
\put(10,4){\line(0,1){8}}}}
        \end{picture}
    \end{minipage}}
\newcommand{\twoppa}{\begin{minipage}{30pt}
        \begin{picture}(30,30)
            \qbezier(0,0)(40,15)(0,30)
            \qbezier(11,21)(1,15)(11,9)
            \qbezier(18,24)(24,27)(30,30)
            \qbezier(18,5)(24,2.5)(30,0)
            {\color{black}
            {\put(8,22.5){\circle*{3}}
\put(20,22.5){\circle*{3}}
\put(9,22.5){\line(1,0){10}}
            }}
            {\color{black}{\put(10,12){\circle*{3}}
\put(10,3){\circle*{3}}
\put(10,4){\line(0,1){8}}}}
        \end{picture}
    \end{minipage}}
\newcommand{\twoppb}{\begin{minipage}{30pt}
        \begin{picture}(30,30)
            \qbezier(0,0)(40,15)(0,30)
            \qbezier(11,21)(1,15)(11,9)
            \qbezier(18,24)(24,27)(30,30)
            \qbezier(18,5)(24,2.5)(30,0)
            {\color{black}
            {\put(8,22.5){\circle*{3}}
\put(20,22.5){\circle*{3}}
\put(9,22.5){\line(1,0){10}}
            }}
            {\color{black}{\put(10,12){\circle*{3}}
\put(10,3){\circle*{3}}
\put(10,4){\line(0,1){8}}}}
\put(8,29){$p_u$}
\put(7,-7){$q_u$}
        \end{picture}
    \end{minipage}}
\newcommand{\twoppc}{\begin{minipage}{30pt}
        \begin{picture}(30,30)
            \qbezier(0,0)(40,15)(0,30)
            \qbezier(11,21)(1,15)(11,9)
            \qbezier(18,24)(24,27)(30,30)
            \qbezier(18,5)(24,2.5)(30,0)
            {\color{black}
            {\put(8,22.5){\circle*{3}}
\put(20,22.5){\circle*{3}}
\put(9,22.5){\line(1,0){10}}
            }}
            {\color{black}{\put(10,12){\circle*{3}}
\put(10,3){\circle*{3}}
\put(10,4){\line(0,1){8}}}}
\put(5,29){$+$}
        \end{picture}
    \end{minipage}}
\newcommand{\twoppd}{\begin{minipage}{30pt}
        \begin{picture}(30,30)
            \qbezier(0,0)(40,15)(0,30)
            \qbezier(11,21)(1,15)(11,9)
            \qbezier(18,24)(24,27)(30,30)
            \qbezier(18,5)(24,2.5)(30,0)
            {\color{black}
            {\put(8,22.5){\circle*{3}}
\put(20,22.5){\circle*{3}}
\put(9,22.5){\line(1,0){10}}
            }}
            {\color{black}{\put(10,12){\circle*{3}}
\put(10,3){\circle*{3}}
\put(10,4){\line(0,1){8}}}}
\put(5,29){$-$}
        \end{picture}
    \end{minipage}}
\newcommand{\twoppe}{
\twoppa
\begin{picture}(0,0)
\put(-20,15){$-$}
\put(-20,-18){$-$}
\end{picture}
}
\newcommand{\twoppf}{
\twoppa
\begin{picture}(0,0)
\put(-20,15){$+$}
\put(-20,-18){$-$}
\end{picture}
}
\newcommand{\twoppg}{
\twoppa
\begin{picture}(0,0)
\put(-20,15){$-$}
\put(-20,-18){$+$}
\end{picture}
}
\newcommand{\twopph}{
\twoppa
\begin{picture}(0,0)
\put(-20,15){$+$}
\put(-20,-18){$+$}
\end{picture}
}
\newcommand{\verline}{ \begin{picture}(0,0)
\put(0,-10){\line(0,1){25}}
        \end{picture}}
\newcommand{\twopmb}{\sxb
        \begin{picture}(0,0)
\put(-20,18){$p$}
\put(-20,-15){$q$}
        \end{picture}}
\newcommand{\twopmbc}{ 
\sxb
\begin{picture}(0,0)
\put(-25,20){$p:q$}
\put(-25,-15){$q:p$}
        \end{picture}
}
\newcommand{\twopmbd}{ 
\sxb
\begin{picture}(0,0)
\put(-20,20){$\mu$}
\put(-20,-15){$\nu$}
        \end{picture}
}
\newcommand{\twopmbde}{ 
\sxb
\begin{picture}(0,0)
\put(-28,22){$q_u : p_u$}
\put(-28,-17){$p_u : q_u$}
        \end{picture}
}
\newcommand{\twopmbe}{ 
\sxb
\begin{picture}(0,0)
\put(-40,25){$(p:q):(q:p)$}
\put(-40,-20){$(q:p):(p:q)$}
        \end{picture}
}
\newcommand{\threepmbe}{ 
\sxb
\begin{picture}(0,0)
\put(-40,-18){$(p:q):(q:p)$}
\put(-40,25){$(q:p):(p:q)$}
\put(-6,0){\line(0,1){8}}
\put(-6,8){\cb}
\put(-6,0){\cb}
        \end{picture}
}
\newcommand{\twopmbf}{ 
\sxb
\begin{picture}(0,0)
\put(-35,22){$m(p : +)$}
\put(-20,-15){$\tilde{q}$}
        \end{picture}
}
\newcommand{\threepmbf}{ 
\sxb
\begin{picture}(0,0)
\put(-35,22){$m(q : +)$}
\put(-20,-15){$\tilde{p}$}
\put(-6,0){\line(0,1){8}}
\put(-6,8){\cb}
\put(-6,0){\cb}
        \end{picture}
}
\newcommand{\threepp}{\begin{minipage}{30pt}
        \begin{picture}(30,30)
                   \qbezier(0,0)(40,15)(0,30)
            \qbezier(11,21)(1,15)(11,9)
            \qbezier(18,24)(24,27)(30,30)
            \qbezier(18,5)(24,2.5)(30,0)
            {\color{black}
            {\put(8,22.5){\circle*{3}}
\put(20,22.5){\circle*{3}}
\put(9,22.5){\line(1,0){10}}
            }}
            {\color{black}{\put(10,12){\circle*{3}}
\put(10,3){\circle*{3}}
\put(10,4){\line(0,1){8}}}}
\put(10,27){$q$}
\put(7,-5){$p$}
 \put(20,10){\line(0,1){8}}
\put(20,18){\cb}
\put(20,10){\cb}
        \end{picture}
    \end{minipage}}
\newcommand{\twopmbg}{ 
\sxb
\begin{picture}(0,0)
\put(-35,22){$m(+ : +)$}
        \end{picture}
}
\newcommand{\twopmbga}{ 
\sxb
\begin{picture}(0,0)
\put(-35,22){$m(- : +)$}
\put(-20,-15){$-$}
        \end{picture}
}
\newcommand{\twopmbgb}{ 
\sxb
\begin{picture}(0,0)
\put(-35,22){$m(+ : +)$}
\put(-20,-15){$-$}
        \end{picture}
}
\newcommand{\twopmbgc}{ 
\sxb
\begin{picture}(0,0)
\put(-35,22){$m(- : +)$}
\put(-20,-15){$+$}
        \end{picture}
}
\newcommand{\twopmbgd}{ 
\sxb
\begin{picture}(0,0)
\put(-35,22){$m(+ : +)$}
\put(-20,-15){$+$}
        \end{picture}
}

\newcommand{\twopmbh}{ 
\sxb
\begin{picture}(0,0)
\put(-35,22){$m(- : +)$}
        \end{picture}
}

\newcommand{\twopmbi}{ 
\sxb
\begin{picture}(0,0)
\put(-20,17){$+$}
        \end{picture}
}
\newcommand{\twopmbj}{ 
\sxb
\begin{picture}(0,0)
\put(-20,17){$-$}
        \end{picture}
}
\newcommand{\twopmbk}{ 
\sxb
\begin{picture}(0,0)
\put(-20,17){$-$}
\put(-20,-15){$-$}
        \end{picture}
}
\newcommand{\twopmbl}{ 
\sxb
\begin{picture}(0,0)
\put(-20,17){$-$}
\put(-20,-15){$+$}
        \end{picture}
}
\newcommand{\twopmbm}{ 
\sxb
\begin{picture}(0,0)
\put(-20,17){$+$}
\put(-20,-15){$-$}
        \end{picture}
}
\newcommand{\twopmbn}{ 
\sxb
\begin{picture}(0,0)
\put(-20,17){$+$}
\put(-20,-15){$+$}
        \end{picture}
}
\newcommand{\twomp}{~\sxa
    \begin{picture}(0,0)
    \put(-7,0){$q$}
    \put(-32,0){$p$}
    \end{picture}}
    \newcommand{\twompe}{~~\sxa
    \begin{picture}(0,0)
    \put(-7,0){$q_u$}
    \put(-34,0){$p_u$}
    \end{picture}}
\newcommand{\twompb}{\qquad \sxa
    \begin{picture}(20,0)
    \put(-5,0){$q:p$}
    \put(-47,0){$p:q$}
    \end{picture}}
\newcommand{\riigen}{\twomp \otimes [xa] + 
\sxb
        \begin{picture}(0,0)
\put(-25,20){$p:q$}
\put(-25,-15){$q:p$}
        \end{picture} \otimes [xb]
}
\newcommand{\riigenb}{\sxa \!\!\! \otimes [xau] + 
\sxb \!\!\!
 \otimes [xbu]
}
\newcommand{\twomm}{\begin{minipage}{30pt}
        \begin{picture}(30,30)
            \qbezier(0,0)(40,15)(0,30)
            \qbezier(11,21)(1,15)(11,9)
            \qbezier(18,24)(24,27)(30,30)
            \qbezier(18,5)(24,2.5)(30,0)
  {\color{black}{\put(14,28){\circle*{3}}
\put(14,19){\circle*{3}}
\put(14,19){\line(0,1){8}}}}
            {\color{black}{\put(16,7){\circle*{3}}
\put(5,7){\circle*{3}}
\put(5,7){\line(1,0){10}}}}
\put(7,34){$p$}
\put(7,-3){$q$}
        \end{picture}
    \end{minipage}}
\newcommand{\twommu}{\begin{minipage}{30pt}
        \begin{picture}(30,30)
            \qbezier(0,0)(40,15)(0,30)
            \qbezier(11,21)(1,15)(11,9)
            \qbezier(18,24)(24,27)(30,30)
            \qbezier(18,5)(24,2.5)(30,0)
  {\color{black}{\put(14,28){\circle*{3}}
\put(14,19){\circle*{3}}
\put(14,19){\line(0,1){8}}}}
            {\color{black}{\put(16,7){\circle*{3}}
\put(5,7){\circle*{3}}
\put(5,7){\line(1,0){10}}}}
        \end{picture}
    \end{minipage}}
\newcommand{\twommb}{\begin{minipage}{30pt}
        \begin{picture}(30,30)
            \qbezier(0,0)(40,15)(0,30)
            \qbezier(11,21)(1,15)(11,9)
            \qbezier(18,24)(24,27)(30,30)
            \qbezier(18,5)(24,2.5)(30,0)
  {\color{black}{\put(14,28){\circle*{3}}
\put(14,19){\circle*{3}}
\put(14,19){\line(0,1){8}}}}
            {\color{black}{\put(16,7){\circle*{3}}
\put(5,7){\circle*{3}}
\put(5,7){\line(1,0){10}}}}
\put(2,34){$p:q$}
\put(2,-3){$q:p$}
        \end{picture}
    \end{minipage}}
\newcommand{\twommc}{\begin{minipage}{30pt}
        \begin{picture}(30,30)
            \qbezier(0,0)(40,15)(0,30)
            \qbezier(11,21)(1,15)(11,9)
            \qbezier(18,24)(24,27)(30,30)
            \qbezier(18,5)(24,2.5)(30,0)
  {\color{black}{\put(14,28){\circle*{3}}
\put(14,19){\circle*{3}}
\put(14,19){\line(0,1){8}}}}
            {\color{black}{\put(16,7){\circle*{3}}
\put(5,7){\circle*{3}}
\put(5,7){\line(1,0){10}}}}
\put(-3,34){$m(p : +)$}
\put(7,-3){$\tilde{q}$}
        \end{picture}
    \end{minipage}}

\newcommand{\isome}{\operatorname{isom}}
\newcommand{\rigencontrA}{
\begin{minipage}{30pt}
        \begin{picture}(30,30)
\qbezier(6.6,17)(0,24)(0,24)
\qbezier(0,4)(20,34)(25,17)
\qbezier(11,13)(20,4)(24.5,12)
\qbezier(24.5,12)(25.5,14.5)(25,17)
{\color{black}{\put(8,19.5){\circle*{3}}
\put(8,10.5){\circle*{3}}
\put(8,11.5){\line(0,1){8}}}}
\put(-3.5,12.5){\text{$p$}}
\put(14,12.5){\text{$-$}}
        \end{picture}
    \end{minipage} \!\! \otimes [x]}
\newcommand{\rigencontrB}{\begin{minipage}{30pt}
        \begin{picture}(30,30)
\qbezier(6.6,17)(0,24)(0,24)
\qbezier(0,4)(20,34)(25,17)
\qbezier(11,13)(20,4)(24.5,12)
\qbezier(24.5,12)(25.5,14.5)(25,17)
{\color{black}{\put(12.5,15.5){\circle*{3}}
\put(4.5,15.5){\circle*{3}}
\put(4.5,15.5){\line(1,0){8}}}}
\put(15.5,14){\text{$p$}}
        \end{picture}
    \end{minipage} \!\! \otimes [xa]}

\newcommand{\invrigen}{\begin{minipage}{30pt}
        \begin{picture}(30,30)
        \cbezier(0,3)(30,5)(30,25)(0,27) 
        \put(5,12.5){\text{$p$}}
        \end{picture}
    \end{minipage} \!\! }
    \newcommand{\invrigene}{\begin{minipage}{30pt}
        \begin{picture}(30,30)
        \cbezier(0,3)(30,5)(30,25)(0,27) 
        \put(5,12.5){\text{$p_u$}}
        \end{picture}
    \end{minipage} \!\! }
\newcommand{\rigen}{
    \begin{minipage}{30pt}
        \begin{picture}(30,30)
\qbezier(6.6,17)(0,24)(0,24)
\qbezier(0,4)(20,34)(25,17)
\qbezier(11,13)(20,4)(24.5,12)
\qbezier(24.5,12)(25.5,14.5)(25,17)
{\color{black}{\put(8,19.5){\circle*{3}}
\put(8,10.5){\circle*{3}}
\put(8,11.5){\line(0,1){8}}}}
\put(-3.5,12.5){\text{$p$}}
\put(14,12.5){\text{$+$}}
        \end{picture}
    \end{minipage}\!\! \otimes [x] 
- m(p:+) \begin{minipage}{30pt}
        \begin{picture}(30,30)
\qbezier(6.6,17)(0,24)(0,24)
\qbezier(0,4)(20,34)(25,17)
\qbezier(11,13)(20,4)(24.5,12)
\qbezier(24.5,12)(25.5,14.5)(25,17)
{\color{black}{\put(8,19.5){\circle*{3}}
\put(8,10.5){\circle*{3}}
\put(8,11.5){\line(0,1){8}}}}
\put(14,12.5){\text{$-$}}
        \end{picture}
    \end{minipage} \!\! \otimes [x]}
    
\newcommand{\rigenUe}{
    \begin{minipage}{30pt}
        \begin{picture}(30,30)
\qbezier(6.6,17)(0,24)(0,24)
\qbezier(0,4)(20,34)(25,17)
\qbezier(11,13)(20,4)(24.5,12)
\qbezier(24.5,12)(25.5,14.5)(25,17)
{\color{black}{\put(8,19.5){\circle*{3}}
\put(8,10.5){\circle*{3}}
\put(8,11.5){\line(0,1){8}}}}
\put(-4,14){\text{$p_u$}}
\put(14,12.5){\text{$+$}}
        \end{picture}
    \end{minipage}\!\! \otimes [xu] 
- m(p_u:+) \begin{minipage}{30pt}
        \begin{picture}(30,30)
\qbezier(6.6,17)(0,24)(0,24)
\qbezier(0,4)(20,34)(25,17)
\qbezier(11,13)(20,4)(24.5,12)
\qbezier(24.5,12)(25.5,14.5)(25,17)
{\color{black}{\put(8,19.5){\circle*{3}}
\put(8,10.5){\circle*{3}}
\put(8,11.5){\line(0,1){8}}}}
\put(14,12.5){\text{$-$}}
        \end{picture}
    \end{minipage} \!\! \otimes [xu]}

\newcommand{\bigcircle}{\begin{picture}(5,5) \put(3,3){\circle{9}} \end{picture}}

\newcommand{\bigcirclep}{\begin{picture}(5,5) \put(3,3){\circle{10}} \put(-0.3,1){$+$} \end{picture}}
\newcommand{\bigcirclem}{\begin{picture}(5,5) \put(3,3){\circle{10}} \put(-0.3,1){$-$} \end{picture}\,}
\newcommand{\bigcirclepp}{\begin{picture}(5,5) \put(3,3){\circle{10}} \put(0,1){$p$} \end{picture}}
\newcommand{\bigcirclerp}{\begin{picture}(5,5) \put(3,3){\circle{10}} \put(0,1){$r$} \end{picture}}
\newcommand{\bigcircleqp}{\begin{picture}(5,5) \put(3,3){\circle{10}} \put(0,1){$q$} \end{picture}}

\newcommand{\inc}{\operatorname{in}}

\newcommand{\SpmpM}{
\begin{minipage}{60pt}
        \begin{picture}(50,70)
            \cbezier(0,0)(0,3)(7,7)(12,12)
            \qbezier(19,19)(25,25)(31,31)
            \cbezier(38,38)(45,45)(50,50)(50,56)
            \cbezier(0,56)(0,50)(17.25,44.25)(34.5,34.5)
            \cbezier(34.5,34.5)(44,30)(50,20)(50,0)
            \cbezier(17,48)(21,53)(24,54)(24,56)
            \cbezier(12,41)(-1,25)(24,10)(24,0)
             \put(15,28){\text{$-$}}
            {\color{black}
            {\put(8,44.5){\circle*{3}}
\put(20,44.5){\circle*{3}}
\put(9,44.5){\line(1,0){10}}
            }}
            {\color{black}{\put(8,15){\circle*{3}}
\put(18,15){\circle*{3}}
\put(9,15){\line(1,0){10}}}}
{\color{black}{\put(27,39){\circle*{3}}
\put(27,29){\circle*{3}}
\put(27,29){\line(0,1){10}}}}
\end{picture}
\end{minipage}
}

\newcommand{\Spmp}{
\begin{minipage}{60pt}
        \begin{picture}(50,70)
            \cbezier(0,0)(0,3)(7,7)(12,12)
            \qbezier(19,19)(25,25)(31,31)
            \cbezier(38,38)(45,45)(50,50)(50,56)
            \cbezier(0,56)(0,50)(17.25,44.25)(34.5,34.5)
            \cbezier(34.5,34.5)(44,30)(50,20)(50,0)
            \cbezier(17,48)(21,53)(24,54)(24,56)
            \cbezier(12,41)(-1,25)(24,10)(24,0)
            \put(15,28){\text{$+$}}
            {\color{black}
            {\put(8,44.5){\circle*{3}}
\put(20,44.5){\circle*{3}}
\put(9,44.5){\line(1,0){10}}
            }}
            {\color{black}{\put(8,15){\circle*{3}}
\put(18,15){\circle*{3}}
\put(9,15){\line(1,0){10}}}}
{\color{black}{\put(27,39){\circle*{3}}
\put(27,29){\circle*{3}}
\put(27,29){\line(0,1){10}}}}
\end{picture}
\end{minipage}
}

\newcommand{\Smpp}{\begin{minipage}{60pt}\begin{picture}(50,70)
            \cbezier(0,0)(0,3)(7,7)(12,12)
            \qbezier(19,19)(25,25)(31,31)
            \cbezier(38,38)(45,45)(50,50)(50,56)
            \cbezier(0,56)(0,50)(17.25,44.25)(34.5,34.5)
            \cbezier(34.5,34.5)(44,30)(50,20)(50,0)
            \cbezier(17,48)(21,53)(24,54)(24,56)
            \cbezier(12,41)(-1,25)(24,10)(24,0)
            {\color{black}{\put(16,21){\circle*{3}}
\put(16,10){\circle*{3}}
\put(16,11){\line(0,1){10}}}}
{\color{black}
            {\put(4.5,44.5){\circle*{3}}
\put(16.5,44.5){\circle*{3}}
\put(5.5,44.5){\line(1,0){10}}
            }}
            {\color{black}
            {\put(32.5,34.5){\circle*{3}}
\put(20.5,34.5){\circle*{3}}
\put(21.5,34.5){\line(1,0){10}}
            }}
            \end{picture}
            \end{minipage}
}

\newcommand{\sxu}{\begin{minipage}{30pt}
        \begin{picture}(30,40)
            \qbezier(0,5)(40,20)(0,35)
            \qbezier(11,26)(1,20)(11,14)
            \qbezier(18,29)(24,32)(30,35)
            \qbezier(18,10)(24,7.5)(30,5)
            {\color{black}
            {\put(8,27.5){\circle*{3}}
\put(20,27.5){\circle*{3}}
\put(9,27.5){\line(1,0){10}}
            }}
            {\color{black}{\put(16,12){\circle*{3}}
\put(5,12){\circle*{3}}
\put(5,12){\line(1,0){10}}}}
\put(6,17){\text{$+$}}
        \end{picture}
    \end{minipage}
}

\newcommand{\sx}{\begin{minipage}{30pt}
        \begin{picture}(30,40)
            \qbezier(0,5)(40,20)(0,35)
            \qbezier(11,26)(1,20)(11,14)
            \qbezier(18,29)(24,32)(30,35)
            \qbezier(18,10)(24,7.5)(30,5)
            {\color{black}
            {\put(8,27.5){\circle*{3}}
\put(20,27.5){\circle*{3}}
\put(9,27.5){\line(1,0){10}}
            }}
            {\color{black}{\put(16,12){\circle*{3}}
\put(5,12){\circle*{3}}
\put(5,12){\line(1,0){10}}}}
\put(6,17){\text{$+$}}
\put(7,35){$p$}
\put(7,0){$q$}
        \end{picture}
    \end{minipage}
}

\newcommand{\sxc}{\begin{minipage}{30pt}
        \begin{picture}(30,40)
            \qbezier(0,5)(40,20)(0,35)
            \qbezier(11,26)(1,20)(11,14)
            \qbezier(18,29)(24,32)(30,35)
            \qbezier(18,10)(24,7.5)(30,5)
            {\color{black}
            {\put(8,27.5){\circle*{3}}
\put(20,27.5){\circle*{3}}
\put(9,27.5){\line(1,0){10}}
            }}
            {\color{black}{\put(16,12){\circle*{3}}
\put(5,12){\circle*{3}}
\put(5,12){\line(1,0){10}}}}
\put(6,17){\text{$+$}}
\put(7,35){$p_u$}
\put(7,0){$q_u$}
        \end{picture}
    \end{minipage}
}

\newcommand{\sxb}{\begin{minipage}{30pt}
        \begin{picture}(30,37)
            \qbezier(0,5)(40,20)(0,35)
            \qbezier(11,26)(1,20)(11,14)
            \qbezier(18,29)(24,32)(30,35)
            \qbezier(18,10)(24,7.5)(30,5)
            {\color{black}
            {\put(8,27.5){\circle*{3}}
\put(20,27.5){\circle*{3}}
\put(9,27.5){\line(1,0){10}}
            }}
            {\color{black}{\put(16,12){\circle*{3}}
\put(5,12){\circle*{3}}
\put(5,12){\line(1,0){10}}}}
\put(6,17){\text{$-$}}
\end{picture}
\end{minipage}
}

\newcommand{\sxa}{\begin{minipage}{30pt}
        \begin{picture}(40,40)
            \qbezier(0,5)(40,20)(0,35)
            \qbezier(11,26)(1,20)(11,14)
            \qbezier(18,29)(24,32)(30,35)
            \qbezier(18,10)(24,7.5)(30,5)
  {\color{black}{\put(14,33){\circle*{3}}
\put(14,24){\circle*{3}}
\put(14,24){\line(0,1){8}}}}
            {\color{black}{\put(10,17){\circle*{3}}
\put(10,8){\circle*{3}}
\put(10,9){\line(0,1){8}}}}
        \end{picture} 
        \end{minipage}
}

\newcommand{\feq}{\begin{minipage}{50pt}
\begin{picture}(50,50)
\put(0,10){\line(1,1){30}}
\qbezier(30,10)(30,10)(20,20)
\qbezier(10,30)(0,40)(0,40)
\linethickness{2pt}
\put(15,15){\circle*{5}}
\put(15,35){\circle*{5}}
\put(15,15){\line(0,1){20}}
\put(22,23){$q$}
\put(3,23){$p$}
\put(13,5){$a$}
\end{picture}
\end{minipage} \otimes [x] \qquad
\longmapsto \qquad\quad\quad
\begin{minipage}{50pt}
\begin{picture}(50,50)
\put(0,10){\line(1,1){30}}
\qbezier(30,10)(30,10)(20,20)
\qbezier(10,30)(0,40)(0,40)
\linethickness{2pt}
\put(5,25){\circle*{5}}
\put(25,25){\circle*{5}}
\put(5,25){\line(1,0){20}}
\put(5,40){$p:q$}
\put(5,5){$q:p$}
\put(30,23){$a$}
\end{picture}
\end{minipage}\ \  \otimes [xa]
}

\newcommand{\cb}{\circle*{3}}

\newcommand{\dIII}{\begin{minipage}{60pt}
        \begin{picture}(50,70)
            \cbezier(0,0)(0,3)(7,7)(12,12)
            \qbezier(19,19)(25,25)(31,31)
            \cbezier(38,38)(45,45)(50,50)(50,56)
            \cbezier(0,56)(0,50)(17.25,44.25)(34.5,34.5)
            \cbezier(34.5,34.5)(44,30)(50,20)(50,0)
            \cbezier(17,48)(21,53)(24,54)(24,56)
            \cbezier(12,41)(-1,25)(24,10)(24,0)
            \end{picture}
            \end{minipage}}
\newcommand{\dIIIa}{
\begin{minipage}{60pt}
        \begin{picture}(50,70)
            \cbezier(0,0)(0,3)(7,7)(12,12)
            \qbezier(12,12)(15.5,15.5)(19,19)
            \qbezier(19,19)(25,25)(31,31)
            \cbezier(38,38)(45,45)(50,50)(50,56)
            \cbezier(0,56)(0,50)(17.25,44.25)(34.5,34.5)
            \cbezier(34.5,34.5)(44,30)(50,20)(50,0)
            \cbezier(17,48)(21,53)(24,54)(24,56)
            \qbezier(12,41)(-1,20)(10,15)
            \qbezier(15,10)(24,0)(24,0)
            \end{picture}
            \end{minipage}
            }
\newcommand{\DIII}{\begin{minipage}{50pt}
        \begin{picture}(50,70)
        \cbezier(37.5,42.5)(43,47)(48,52)(50,56)
            \qbezier(31,36)(25,30)(19,24)
            \cbezier(12,17)(5,10)(0,5)(0,0)
            \cbezier(50,0)(50,5)(32.75,10.75)(15.5,20.5)
            \cbezier(15.5,20.5)(6,25)(0,50)(0,56)
            \cbezier(31.5,8.5)(30,7)(24,4)(24,0)
            \cbezier(37,13)(61,30)(24,40)(24,56)
        \end{picture}
    \end{minipage}}
\newcommand{\DIIIa}{\begin{minipage}{50pt}
        \begin{picture}(50,70)
        \cbezier(37.5,42.5)(43,47)(48,52)(50,56)
         \qbezier(31,36)(34,39)(37.5,42.5)
            \qbezier(31,36)(25,30)(19,24)
            \cbezier(12,17)(5,10)(0,5)(0,0)
            \cbezier(50,0)(50,5)(32.75,10.75)(15.5,20.5)
            \cbezier(15.5,20.5)(6,25)(0,50)(0,56)
            \cbezier(31.5,8.5)(30,7)(24,4)(24,0)
            \qbezier(37,13)(61,30)(40,42)
            \qbezier(36,46)(30,51)(24,56)
        \end{picture}
    \end{minipage}
}
\newcommand{\ic}{\operatorname{in}}
\newcommand{\id}{\operatorname{id}}

\newcommand{\second}
{\begin{minipage}{30pt}
        \begin{picture}(30,30)
            \qbezier(0,0)(40,15)(0,30)
            \qbezier(11,21)(1,15)(11,9)
            \qbezier(18,24)(24,27)(30,30)
            \qbezier(18,5)(24,2.5)(30,0)
        \end{picture}
    \end{minipage}
}

\newcommand{\secondD}{\begin{minipage}{30pt}
        \begin{picture}(30,30)
            \qbezier(0,0)(20,15)(0,30)
            \qbezier(30,0)(10,15)(30,30)
        \end{picture}
    \end{minipage}
}

\newcommand{\fI}{\begin{minipage}{90pt}
         \begin{picture}(90,50)
         \put(20,10){\line(1,0){1}}
        \put(17.5,10.2){\line(-4,1){1}}
        \put(15,10.9){\line(-3,1){1}}
        \put(12.5,12){\line(-2,1){1}}
        \put(10.6,13.2){\line(-1,1){.9}}
        \put(8.9,14.9){\line(-3,4){1}}
        \put(7.3,17){\line(-1,2){.8}}
        \put(6,20){\line(-1,3){.5}}
        \put(5.3,22.7){\line(-1,6){.2}}
        \put(5,25){\line(0,1){1}}
        \put(20,40){\line(-1,0){1}}
        \put(17.5,39.8){\line(-4,-1){1}}
        \put(15,39.1){\line(-3,-1){1}}
        \put(12.5,38){\line(-2,-1){1}}
        \put(10.6,36.8){\line(-1,-1){.9}}
        \put(8.9,35.1){\line(-3,-4){1}}
        \put(7.3,33){\line(-1,-2){.8}}
        \put(6,30){\line(-1,-3){.5}}
        \put(5.3,27.3){\line(-1,-6){.2}}
        \put(20, 23){$+$}
                \put(20,25){\oval(30,30)[r]}
                \put(70,25){\oval(30,30)[l]}
        \put(62,23){$+$}
                 \put(70,10){\line(1,0){1}}
        \put(72.5,10.2){\line(4,1){1}}
        \put(75,10.9){\line(3,1){1}}
        \put(77.5,12){\line(2,1){1}}
        \put(79.4,13.2){\line(1,1){.9}}
        \put(81.1,14.9){\line(3,4){1}}
        \put(82.7,17){\line(1,2){.8}}
        \put(84,20){\line(1,3){.5}}
        \put(84.7,22.7){\line(1,6){.2}}
        \put(85,25){\line(0,1){1}}
        \put(70,40){\line(1,0){1}}
        \put(72.5,39.8){\line(4,-1){1}}
        \put(75,39.1){\line(3,-1){1}}
        \put(77.5,38){\line(2,-1){1}}
        \put(79.4,36.8){\line(1,-1){.9}}
        \put(81.1,35.1){\line(3,-4){1}}
        \put(82.7,33){\line(1,-2){.8}}
        \put(84,30){\line(1,-3){.5}}
        \put(84.7,27.3){\line(1,-6){.2}}
        \end{picture}
        \end{minipage}
&\begin{minipage}{90pt}
\begin{picture}(90,50)
\put(10,24){\vector(1,0){50}}
\end{picture}
        \end{minipage}
\begin{minipage}{70pt}
\begin{picture}(70,50)
\put(25,22){$+$}
\put(15,10){\line(1,0){1}}
        \put(12.5,10.2){\line(-4,1){1}}
        \put(10,10.9){\line(-3,1){1}}
        \put(7.5,12){\line(-2,1){1}}
        \put(5.6,13.2){\line(-1,1){.9}}
        \put(3.9,14.9){\line(-3,4){1}}
        \put(2.3,17){\line(-1,2){.8}}
        \put(1,20){\line(-1,3){.5}}
        \put(0.3,22.7){\line(-1,6){.2}}
        \put(0,25){\line(0,1){1}}
        \put(15,40){\line(-1,0){1}}
        \put(12.5,39.8){\line(-4,-1){1}}
        \put(10,39.1){\line(-3,-1){1}}
        \put(7.5,38){\line(-2,-1){1}}
        \put(5.6,36.8){\line(-1,-1){.9}}
        \put(3.9,35.1){\line(-3,-4){1}}
        \put(2.3,33){\line(-1,-2){.8}}
        \put(1,30){\line(-1,-3){.5}}
        \put(0.3,27.3){\line(-1,-6){.2}}
\qbezier(15,40)(30,25)(45,40)
\qbezier(15,10)(30,25)(45,10)
 \put(45,10){\line(1,0){1}}
        \put(47.5,10.2){\line(4,1){1}}
        \put(50,10.9){\line(3,1){1}}
        \put(52.5,12){\line(2,1){1}}
        \put(54.4,13.2){\line(1,1){.9}}
        \put(56.1,14.9){\line(3,4){1}}
        \put(57.7,17){\line(1,2){.8}}
        \put(59,20){\line(1,3){.5}}
        \put(59.7,22.7){\line(1,6){.2}}
        \put(60,25){\line(0,1){1}}
        \put(45,40){\line(1,0){1}}
        \put(47.5,39.8){\line(4,-1){1}}
        \put(50,39.1){\line(3,-1){1}}
        \put(52.5,38){\line(2,-1){1}}
        \put(54.4,36.8){\line(1,-1){.9}}
        \put(56.1,35.1){\line(3,-4){1}}
        \put(57.7,33){\line(1,-2){.8}}
        \put(59,30){\line(1,-3){.5}}
        \put(59.7,27.3){\line(1,-6){.2}}
        \end{picture}
    \end{minipage}
    \text{$+$}\quad
    \begin{minipage}{80pt}
\begin{picture}(80,50)
\put(15,10){\line(1,0){1}}
\put(25,22){$-$}
        \put(12.5,10.2){\line(-4,1){1}}
        \put(10,10.9){\line(-3,1){1}}
        \put(7.5,12){\line(-2,1){1}}
        \put(5.6,13.2){\line(-1,1){.9}}
        \put(3.9,14.9){\line(-3,4){1}}
        \put(2.3,17){\line(-1,2){.8}}
        \put(1,20){\line(-1,3){.5}}
        \put(0.3,22.7){\line(-1,6){.2}}
        \put(0,25){\line(0,1){1}}
        \put(15,40){\line(-1,0){1}}
        \put(12.5,39.8){\line(-4,-1){1}}
        \put(10,39.1){\line(-3,-1){1}}
        \put(7.5,38){\line(-2,-1){1}}
        \put(5.6,36.8){\line(-1,-1){.9}}
        \put(3.9,35.1){\line(-3,-4){1}}
        \put(2.3,33){\line(-1,-2){.8}}
        \put(1,30){\line(-1,-3){.5}}
        \put(0.3,27.3){\line(-1,-6){.2}}
\qbezier(15,40)(30,25)(45,40)
\qbezier(15,10)(30,25)(45,10)
 \put(45,10){\line(1,0){1}}
        \put(47.5,10.2){\line(4,1){1}}
        \put(50,10.9){\line(3,1){1}}
        \put(52.5,12){\line(2,1){1}}
        \put(54.4,13.2){\line(1,1){.9}}
        \put(56.1,14.9){\line(3,4){1}}
        \put(57.7,17){\line(1,2){.8}}
        \put(59,20){\line(1,3){.5}}
        \put(59.7,22.7){\line(1,6){.2}}
        \put(60,25){\line(0,1){1}}
        \put(45,40){\line(1,0){1}}
        \put(47.5,39.8){\line(4,-1){1}}
        \put(50,39.1){\line(3,-1){1}}
        \put(52.5,38){\line(2,-1){1}}
        \put(54.4,36.8){\line(1,-1){.9}}
        \put(56.1,35.1){\line(3,-4){1}}
        \put(57.7,33){\line(1,-2){.8}}
        \put(59,30){\line(1,-3){.5}}
        \put(59.7,27.3){\line(1,-6){.2}}
        \end{picture}
    \end{minipage}}
\newcommand{\fII}{\begin{minipage}{90pt}
         \begin{picture}(90,50)
         \put(20,10){\line(1,0){1}}
        \put(17.5,10.2){\line(-4,1){1}}
        \put(15,10.9){\line(-3,1){1}}
        \put(12.5,12){\line(-2,1){1}}
        \put(10.6,13.2){\line(-1,1){.9}}
        \put(8.9,14.9){\line(-3,4){1}}
        \put(7.3,17){\line(-1,2){.8}}
        \put(6,20){\line(-1,3){.5}}
        \put(5.3,22.7){\line(-1,6){.2}}
        \put(5,25){\line(0,1){1}}
        \put(20,40){\line(-1,0){1}}
        \put(17.5,39.8){\line(-4,-1){1}}
        \put(15,39.1){\line(-3,-1){1}}
        \put(12.5,38){\line(-2,-1){1}}
        \put(10.6,36.8){\line(-1,-1){.9}}
        \put(8.9,35.1){\line(-3,-4){1}}
        \put(7.3,33){\line(-1,-2){.8}}
        \put(6,30){\line(-1,-3){.5}}
        \put(5.3,27.3){\line(-1,-6){.2}}
        \put(20, 23){$+$}
                \put(20,25){\oval(30,30)[r]}
                \put(70,25){\oval(30,30)[l]}
        \put(62,23){$-$}
                 \put(70,10){\line(1,0){1}}
        \put(72.5,10.2){\line(4,1){1}}
        \put(75,10.9){\line(3,1){1}}
        \put(77.5,12){\line(2,1){1}}
        \put(79.4,13.2){\line(1,1){.9}}
        \put(81.1,14.9){\line(3,4){1}}
        \put(82.7,17){\line(1,2){.8}}
        \put(84,20){\line(1,3){.5}}
        \put(84.7,22.7){\line(1,6){.2}}
        \put(85,25){\line(0,1){1}}
        \put(70,40){\line(1,0){1}}
        \put(72.5,39.8){\line(4,-1){1}}
        \put(75,39.1){\line(3,-1){1}}
        \put(77.5,38){\line(2,-1){1}}
        \put(79.4,36.8){\line(1,-1){.9}}
        \put(81.1,35.1){\line(3,-4){1}}
        \put(82.7,33){\line(1,-2){.8}}
        \put(84,30){\line(1,-3){.5}}
        \put(84.7,27.3){\line(1,-6){.2}}
        \end{picture}
        \end{minipage}
&\begin{minipage}{90pt}
\begin{picture}(90,50)
\put(10,24){\vector(1,0){50}}
\end{picture}
        \end{minipage}
    \begin{minipage}{70pt}
\begin{picture}(70,50)
\put(25,22){$+$}
\put(15,10){\line(1,0){1}}
        \put(12.5,10.2){\line(-4,1){1}}
        \put(10,10.9){\line(-3,1){1}}
        \put(7.5,12){\line(-2,1){1}}
        \put(5.6,13.2){\line(-1,1){.9}}
        \put(3.9,14.9){\line(-3,4){1}}
        \put(2.3,17){\line(-1,2){.8}}
        \put(1,20){\line(-1,3){.5}}
        \put(0.3,22.7){\line(-1,6){.2}}
        \put(0,25){\line(0,1){1}}
        \put(15,40){\line(-1,0){1}}
        \put(12.5,39.8){\line(-4,-1){1}}
        \put(10,39.1){\line(-3,-1){1}}
        \put(7.5,38){\line(-2,-1){1}}
        \put(5.6,36.8){\line(-1,-1){.9}}
        \put(3.9,35.1){\line(-3,-4){1}}
        \put(2.3,33){\line(-1,-2){.8}}
        \put(1,30){\line(-1,-3){.5}}
        \put(0.3,27.3){\line(-1,-6){.2}}
\qbezier(15,40)(30,25)(45,40)
\qbezier(15,10)(30,25)(45,10)
 \put(45,10){\line(1,0){1}}
        \put(47.5,10.2){\line(4,1){1}}
        \put(50,10.9){\line(3,1){1}}
        \put(52.5,12){\line(2,1){1}}
        \put(54.4,13.2){\line(1,1){.9}}
        \put(56.1,14.9){\line(3,4){1}}
        \put(57.7,17){\line(1,2){.8}}
        \put(59,20){\line(1,3){.5}}
        \put(59.7,22.7){\line(1,6){.2}}
        \put(60,25){\line(0,1){1}}
        \put(45,40){\line(1,0){1}}
        \put(47.5,39.8){\line(4,-1){1}}
        \put(50,39.1){\line(3,-1){1}}
        \put(52.5,38){\line(2,-1){1}}
        \put(54.4,36.8){\line(1,-1){.9}}
        \put(56.1,35.1){\line(3,-4){1}}
        \put(57.7,33){\line(1,-2){.8}}
        \put(59,30){\line(1,-3){.5}}
        \put(59.7,27.3){\line(1,-6){.2}}
        \end{picture}
    \end{minipage}}
\newcommand{\fIII}{\begin{minipage}{90pt}
         \begin{picture}(90,50)
         \put(20,10){\line(1,0){1}}
        \put(17.5,10.2){\line(-4,1){1}}
        \put(15,10.9){\line(-3,1){1}}
        \put(12.5,12){\line(-2,1){1}}
        \put(10.6,13.2){\line(-1,1){.9}}
        \put(8.9,14.9){\line(-3,4){1}}
        \put(7.3,17){\line(-1,2){.8}}
        \put(6,20){\line(-1,3){.5}}
        \put(5.3,22.7){\line(-1,6){.2}}
        \put(5,25){\line(0,1){1}}
        \put(20,40){\line(-1,0){1}}
        \put(17.5,39.8){\line(-4,-1){1}}
        \put(15,39.1){\line(-3,-1){1}}
        \put(12.5,38){\line(-2,-1){1}}
        \put(10.6,36.8){\line(-1,-1){.9}}
        \put(8.9,35.1){\line(-3,-4){1}}
        \put(7.3,33){\line(-1,-2){.8}}
        \put(6,30){\line(-1,-3){.5}}
        \put(5.3,27.3){\line(-1,-6){.2}}
        \put(20, 23){$-$}
                \put(20,25){\oval(30,30)[r]}
                \put(70,25){\oval(30,30)[l]}
        \put(62,23){$+$}
                 \put(70,10){\line(1,0){1}}
        \put(72.5,10.2){\line(4,1){1}}
        \put(75,10.9){\line(3,1){1}}
        \put(77.5,12){\line(2,1){1}}
        \put(79.4,13.2){\line(1,1){.9}}
        \put(81.1,14.9){\line(3,4){1}}
        \put(82.7,17){\line(1,2){.8}}
        \put(84,20){\line(1,3){.5}}
        \put(84.7,22.7){\line(1,6){.2}}
        \put(85,25){\line(0,1){1}}
        \put(70,40){\line(1,0){1}}
        \put(72.5,39.8){\line(4,-1){1}}
        \put(75,39.1){\line(3,-1){1}}
        \put(77.5,38){\line(2,-1){1}}
        \put(79.4,36.8){\line(1,-1){.9}}
        \put(81.1,35.1){\line(3,-4){1}}
        \put(82.7,33){\line(1,-2){.8}}
        \put(84,30){\line(1,-3){.5}}
        \put(84.7,27.3){\line(1,-6){.2}}
        \end{picture}
        \end{minipage}
&\begin{minipage}{90pt}
\begin{picture}(90,50)
\put(10,24){\vector(1,0){50}}
\end{picture}
        \end{minipage}
    \begin{minipage}{70pt}
\begin{picture}(70,50)
\put(25,22){$+$}
\put(15,10){\line(1,0){1}}
        \put(12.5,10.2){\line(-4,1){1}}
        \put(10,10.9){\line(-3,1){1}}
        \put(7.5,12){\line(-2,1){1}}
        \put(5.6,13.2){\line(-1,1){.9}}
        \put(3.9,14.9){\line(-3,4){1}}
        \put(2.3,17){\line(-1,2){.8}}
        \put(1,20){\line(-1,3){.5}}
        \put(0.3,22.7){\line(-1,6){.2}}
        \put(0,25){\line(0,1){1}}
        \put(15,40){\line(-1,0){1}}
        \put(12.5,39.8){\line(-4,-1){1}}
        \put(10,39.1){\line(-3,-1){1}}
        \put(7.5,38){\line(-2,-1){1}}
        \put(5.6,36.8){\line(-1,-1){.9}}
        \put(3.9,35.1){\line(-3,-4){1}}
        \put(2.3,33){\line(-1,-2){.8}}
        \put(1,30){\line(-1,-3){.5}}
        \put(0.3,27.3){\line(-1,-6){.2}}
\qbezier(15,40)(30,25)(45,40)
\qbezier(15,10)(30,25)(45,10)
 \put(45,10){\line(1,0){1}}
        \put(47.5,10.2){\line(4,1){1}}
        \put(50,10.9){\line(3,1){1}}
        \put(52.5,12){\line(2,1){1}}
        \put(54.4,13.2){\line(1,1){.9}}
        \put(56.1,14.9){\line(3,4){1}}
        \put(57.7,17){\line(1,2){.8}}
        \put(59,20){\line(1,3){.5}}
        \put(59.7,22.7){\line(1,6){.2}}
        \put(60,25){\line(0,1){1}}
        \put(45,40){\line(1,0){1}}
        \put(47.5,39.8){\line(4,-1){1}}
        \put(50,39.1){\line(3,-1){1}}
        \put(52.5,38){\line(2,-1){1}}
        \put(54.4,36.8){\line(1,-1){.9}}
        \put(56.1,35.1){\line(3,-4){1}}
        \put(57.7,33){\line(1,-2){.8}}
        \put(59,30){\line(1,-3){.5}}
        \put(59.7,27.3){\line(1,-6){.2}}
        \end{picture}
    \end{minipage}}
\newcommand{\fIV}{\begin{minipage}{90pt}
         \begin{picture}(90,50)
         \put(20,10){\line(1,0){1}}
        \put(17.5,10.2){\line(-4,1){1}}
        \put(15,10.9){\line(-3,1){1}}
        \put(12.5,12){\line(-2,1){1}}
        \put(10.6,13.2){\line(-1,1){.9}}
        \put(8.9,14.9){\line(-3,4){1}}
        \put(7.3,17){\line(-1,2){.8}}
        \put(6,20){\line(-1,3){.5}}
        \put(5.3,22.7){\line(-1,6){.2}}
        \put(5,25){\line(0,1){1}}
        \put(20,40){\line(-1,0){1}}
        \put(17.5,39.8){\line(-4,-1){1}}
        \put(15,39.1){\line(-3,-1){1}}
        \put(12.5,38){\line(-2,-1){1}}
        \put(10.6,36.8){\line(-1,-1){.9}}
        \put(8.9,35.1){\line(-3,-4){1}}
        \put(7.3,33){\line(-1,-2){.8}}
        \put(6,30){\line(-1,-3){.5}}
        \put(5.3,27.3){\line(-1,-6){.2}}
        \put(20, 23){$-$}
                \put(20,25){\oval(30,30)[r]}
                \put(70,25){\oval(30,30)[l]}
        \put(62,23){$-$}
                 \put(70,10){\line(1,0){1}}
        \put(72.5,10.2){\line(4,1){1}}
        \put(75,10.9){\line(3,1){1}}
        \put(77.5,12){\line(2,1){1}}
        \put(79.4,13.2){\line(1,1){.9}}
        \put(81.1,14.9){\line(3,4){1}}
        \put(82.7,17){\line(1,2){.8}}
        \put(84,20){\line(1,3){.5}}
        \put(84.7,22.7){\line(1,6){.2}}
        \put(85,25){\line(0,1){1}}
        \put(70,40){\line(1,0){1}}
        \put(72.5,39.8){\line(4,-1){1}}
        \put(75,39.1){\line(3,-1){1}}
        \put(77.5,38){\line(2,-1){1}}
        \put(79.4,36.8){\line(1,-1){.9}}
        \put(81.1,35.1){\line(3,-4){1}}
        \put(82.7,33){\line(1,-2){.8}}
        \put(84,30){\line(1,-3){.5}}
        \put(84.7,27.3){\line(1,-6){.2}}
        \end{picture}
        \end{minipage}
&\begin{minipage}{90pt}
\begin{picture}(90,50)
\put(10,24){\vector(1,0){50}}
\end{picture}
        \end{minipage}
    \begin{minipage}{70pt}
\begin{picture}(70,50)
\put(25,22){$-$}
\put(15,10){\line(1,0){1}}
        \put(12.5,10.2){\line(-4,1){1}}
        \put(10,10.9){\line(-3,1){1}}
        \put(7.5,12){\line(-2,1){1}}
        \put(5.6,13.2){\line(-1,1){.9}}
        \put(3.9,14.9){\line(-3,4){1}}
        \put(2.3,17){\line(-1,2){.8}}
        \put(1,20){\line(-1,3){.5}}
        \put(0.3,22.7){\line(-1,6){.2}}
        \put(0,25){\line(0,1){1}}
        \put(15,40){\line(-1,0){1}}
        \put(12.5,39.8){\line(-4,-1){1}}
        \put(10,39.1){\line(-3,-1){1}}
        \put(7.5,38){\line(-2,-1){1}}
        \put(5.6,36.8){\line(-1,-1){.9}}
        \put(3.9,35.1){\line(-3,-4){1}}
        \put(2.3,33){\line(-1,-2){.8}}
        \put(1,30){\line(-1,-3){.5}}
        \put(0.3,27.3){\line(-1,-6){.2}}
\qbezier(15,40)(30,25)(45,40)
\qbezier(15,10)(30,25)(45,10)
 \put(45,10){\line(1,0){1}}
        \put(47.5,10.2){\line(4,1){1}}
        \put(50,10.9){\line(3,1){1}}
        \put(52.5,12){\line(2,1){1}}
        \put(54.4,13.2){\line(1,1){.9}}
        \put(56.1,14.9){\line(3,4){1}}
        \put(57.7,17){\line(1,2){.8}}
        \put(59,20){\line(1,3){.5}}
        \put(59.7,22.7){\line(1,6){.2}}
        \put(60,25){\line(0,1){1}}
        \put(45,40){\line(1,0){1}}
        \put(47.5,39.8){\line(4,-1){1}}
        \put(50,39.1){\line(3,-1){1}}
        \put(52.5,38){\line(2,-1){1}}
        \put(54.4,36.8){\line(1,-1){.9}}
        \put(56.1,35.1){\line(3,-4){1}}
        \put(57.7,33){\line(1,-2){.8}}
        \put(59,30){\line(1,-3){.5}}
        \put(59.7,27.3){\line(1,-6){.2}}
        \end{picture}
    \end{minipage}}
\newcommand{\fVIIL}{\begin{minipage}{70pt}
  \begin{picture}(90,60)
  \put(5,25){$p$}
            \qbezier(10,40)(25,25)(10,15)
\qbezier(40,15)(25,25)(40,40)
\put(40,25){$q$}
\end{picture}
    \end{minipage}}
\newcommand{\fVIIR}{
\begin{minipage}{80pt}
\begin{picture}(80,60)
        \put(0,24){\vector(1,0){50}}
        \end{picture}
    \end{minipage}\quad
    \begin{minipage}{70pt}
  \begin{picture}(80,60)
  \put(10,42){$p:q$}
        \qbezier(5,40)(20,25)(35,40)
\qbezier(5,15)(20,30)(35,15)
\put(10,10){$q:p$}
\end{picture}
    \end{minipage}
     \Big( = \qquad
    \begin{minipage}{65pt}
    \begin{picture}(0,60)
    \put(10,42){$q:p$}
    \qbezier(5,40)(20,25)(35,40)
    \qbezier(5,15)(20,30)(35,15)
    \put(10,10){$p:q$}
    \end{picture}
    \end{minipage} \Big)
    }
\newcommand{\fVIL}{\begin{minipage}{70pt}
        \begin{picture}(50,60)
        \put(20,25){$-$}
        \put(25,0){\line(-1,0){1}}
        \put(22.6,0.2){\line(-4,1){1}}
        \put(20,0.9){\line(-3,1){1}}
        \put(17.5,2){\line(-2,1){1}}
        \put(15.6,3.2){\line(-1,1){.9}}
        \put(13.9,4.9){\line(-3,4){1}}
        \put(12.3,7){\line(-1,2){.8}}
        \put(11,10){\line(-1,3){.5}}
        \put(10.3,12.7){\line(-1,6){.2}}
        \put(10,15){\line(0,1){1}}
        \put(25,0){\line(1,0){1}}
        \put(27.5,0.2){\line(4,1){1}}
        \put(30,0.9){\line(3,1){1}}
        \put(32.5,2){\line(2,1){1}}
        \put(34.4,3.2){\line(1,1){.9}}
        \put(36.1,4.9){\line(3,4){1}}
        \put(37.7,7){\line(1,2){.8}}
        \put(39,10){\line(1,3){.5}}
        \put(39.7,12.7){\line(1,6){.2}}
        \put(40,15){\line(0,1){1}}
    \qbezier(10,40)(25,25)(10,15)
\qbezier(40,15)(25,25)(40,40)
        \put(25,55){\line(-1,0){1}}
        \put(22.6,54.8){\line(-4,-1){1}}
        \put(20,54.1){\line(-3,-1){1}}
        \put(17.5,53){\line(-2,-1){1}}
        \put(15.6,51.8){\line(-1,-1){.9}}
        \put(13.9,50.1){\line(-3,-4){1}}
        \put(12.3,48){\line(-1,-2){.8}}
        \put(11,45){\line(-1,-3){.5}}
        \put(10.3,42.3){\line(-1,-6){.2}}
        \put(10,40){\line(0,-1){1}}
        \put(25,55){\line(1,0){1}}
        \put(27.5,54.8){\line(4,-1){1}}
        \put(30,54.1){\line(3,-1){1}}
        \put(32.5,53){\line(2,-1){1}}
        \put(34.4,51.8){\line(1,-1){.9}}
        \put(36.1,50.1){\line(3,-4){1}}
        \put(37.7,48){\line(1,-2){.8}}
        \put(39,45){\line(1,-3){.5}}
        \put(39.7,42.3){\line(1,-6){.2}}
        \put(40,40){\line(0,-1){1}}
        \end{picture}
    \end{minipage}}
\newcommand{\fVIvector}{\begin{minipage}{40pt}
        \begin{picture}(0,60)
\put(0,24){\vector(1,0){25}}
\end{picture}
    \end{minipage}}
\newcommand{\fVIRa}{\begin{minipage}{50pt}
        \begin{picture}(50,60)
        \put(15,10){$-$}
        \put(15,40){$+$}
\put(20,0){\line(-1,0){1}}
        \put(17.6,0.2){\line(-4,1){1}}
        \put(15,0.9){\line(-3,1){1}}
        \put(12.5,2){\line(-2,1){1}}
        \put(10.6,3.2){\line(-1,1){.9}}
        \put(8.9,4.9){\line(-3,4){1}}
        \put(7.3,7){\line(-1,2){.8}}
        \put(6,10){\line(-1,3){.5}}
        \put(5.3,12.7){\line(-1,6){.2}}
        \put(5,15){\line(0,1){1}}
        \put(20,0){\line(1,0){1}}
        \put(22.5,0.2){\line(4,1){1}}
        \put(25,0.9){\line(3,1){1}}
        \put(27.5,2){\line(2,1){1}}
        \put(29.4,3.2){\line(1,1){.9}}
        \put(31.1,4.9){\line(3,4){1}}
        \put(32.7,7){\line(1,2){.8}}
        \put(34,10){\line(1,3){.5}}
        \put(34.7,12.7){\line(1,6){.2}}
        \put(35,15){\line(0,1){1}}
\qbezier(5,40)(20,20)(35,40)
\qbezier(5,15)(20,35)(35,15)
\put(20,55){\line(-1,0){1}}
        \put(17.6,54.8){\line(-4,-1){1}}
        \put(15,54.1){\line(-3,-1){1}}
        \put(12.5,53){\line(-2,-1){1}}
        \put(10.6,51.8){\line(-1,-1){.9}}
        \put(8.9,50.1){\line(-3,-4){1}}
        \put(7.3,48){\line(-1,-2){.8}}
        \put(6,45){\line(-1,-3){.5}}
        \put(5.3,42.3){\line(-1,-6){.2}}
        \put(5,40){\line(0,-1){1}}
        \put(20,55){\line(1,0){1}}
        \put(22.5,54.8){\line(4,-1){1}}
        \put(25,54.1){\line(3,-1){1}}
        \put(27.5,53){\line(2,-1){1}}
        \put(29.4,51.8){\line(1,-1){.9}}
        \put(31.1,50.1){\line(3,-4){1}}
        \put(32.7,48){\line(1,-2){.8}}
        \put(34,45){\line(1,-3){.5}}
        \put(34.7,42.3){\line(1,-6){.2}}
        \put(35,40){\line(0,-1){1}}
\end{picture}
    \end{minipage}} 
\newcommand{\fVIRb}{\begin{minipage}{50pt}
        \begin{picture}(50,60)
        \put(15,10){$+$}
        \put(15,40){$-$}
\put(20,0){\line(-1,0){1}}
        \put(17.6,0.2){\line(-4,1){1}}
        \put(15,0.9){\line(-3,1){1}}
        \put(12.5,2){\line(-2,1){1}}
        \put(10.6,3.2){\line(-1,1){.9}}
        \put(8.9,4.9){\line(-3,4){1}}
        \put(7.3,7){\line(-1,2){.8}}
        \put(6,10){\line(-1,3){.5}}
        \put(5.3,12.7){\line(-1,6){.2}}
        \put(5,15){\line(0,1){1}}
        \put(20,0){\line(1,0){1}}
        \put(22.5,0.2){\line(4,1){1}}
        \put(25,0.9){\line(3,1){1}}
        \put(27.5,2){\line(2,1){1}}
        \put(29.4,3.2){\line(1,1){.9}}
        \put(31.1,4.9){\line(3,4){1}}
        \put(32.7,7){\line(1,2){.8}}
        \put(34,10){\line(1,3){.5}}
        \put(34.7,12.7){\line(1,6){.2}}
        \put(35,15){\line(0,1){1}}
\qbezier(5,40)(20,20)(35,40)
\qbezier(5,15)(20,35)(35,15)
\put(20,55){\line(-1,0){1}}
        \put(17.6,54.8){\line(-4,-1){1}}
        \put(15,54.1){\line(-3,-1){1}}
        \put(12.5,53){\line(-2,-1){1}}
        \put(10.6,51.8){\line(-1,-1){.9}}
        \put(8.9,50.1){\line(-3,-4){1}}
        \put(7.3,48){\line(-1,-2){.8}}
        \put(6,45){\line(-1,-3){.5}}
        \put(5.3,42.3){\line(-1,-6){.2}}
        \put(5,40){\line(0,-1){1}}
        \put(20,55){\line(1,0){1}}
        \put(22.5,54.8){\line(4,-1){1}}
        \put(25,54.1){\line(3,-1){1}}
        \put(27.5,53){\line(2,-1){1}}
        \put(29.4,51.8){\line(1,-1){.9}}
        \put(31.1,50.1){\line(3,-4){1}}
        \put(32.7,48){\line(1,-2){.8}}
        \put(34,45){\line(1,-3){.5}}
        \put(34.7,42.3){\line(1,-6){.2}}
        \put(35,40){\line(0,-1){1}}
\end{picture}
    \end{minipage}}
\newcommand{\fVIRc}{\begin{minipage}{50pt}
        \begin{picture}(50,60)
        \put(-3,27){$s$}
        \put(15,10){$-$}
        \put(15,40){$-$}
\put(20,0){\line(-1,0){1}}
        \put(17.6,0.2){\line(-4,1){1}}
        \put(15,0.9){\line(-3,1){1}}
        \put(12.5,2){\line(-2,1){1}}
        \put(10.6,3.2){\line(-1,1){.9}}
        \put(8.9,4.9){\line(-3,4){1}}
        \put(7.3,7){\line(-1,2){.8}}
        \put(6,10){\line(-1,3){.5}}
        \put(5.3,12.7){\line(-1,6){.2}}
        \put(5,15){\line(0,1){1}}
        \put(20,0){\line(1,0){1}}
        \put(22.5,0.2){\line(4,1){1}}
        \put(25,0.9){\line(3,1){1}}
        \put(27.5,2){\line(2,1){1}}
        \put(29.4,3.2){\line(1,1){.9}}
        \put(31.1,4.9){\line(3,4){1}}
        \put(32.7,7){\line(1,2){.8}}
        \put(34,10){\line(1,3){.5}}
        \put(34.7,12.7){\line(1,6){.2}}
        \put(35,15){\line(0,1){1}}
\qbezier(5,40)(20,20)(35,40)
\qbezier(5,15)(20,35)(35,15)
\put(20,55){\line(-1,0){1}}
        \put(17.6,54.8){\line(-4,-1){1}}
        \put(15,54.1){\line(-3,-1){1}}
        \put(12.5,53){\line(-2,-1){1}}
        \put(10.6,51.8){\line(-1,-1){.9}}
        \put(8.9,50.1){\line(-3,-4){1}}
        \put(7.3,48){\line(-1,-2){.8}}
        \put(6,45){\line(-1,-3){.5}}
        \put(5.3,42.3){\line(-1,-6){.2}}
        \put(5,40){\line(0,-1){1}}
        \put(20,55){\line(1,0){1}}
        \put(22.5,54.8){\line(4,-1){1}}
        \put(25,54.1){\line(3,-1){1}}
        \put(27.5,53){\line(2,-1){1}}
        \put(29.4,51.8){\line(1,-1){.9}}
        \put(31.1,50.1){\line(3,-4){1}}
        \put(32.7,48){\line(1,-2){.8}}
        \put(34,45){\line(1,-3){.5}}
        \put(34.7,42.3){\line(1,-6){.2}}
        \put(35,40){\line(0,-1){1}}
\end{picture}
    \end{minipage}}
\newcommand{\fVL}{\begin{minipage}{70pt}
        \begin{picture}(50,60)
        \put(20,25){$+$}
        \put(25,0){\line(-1,0){1}}
        \put(22.6,0.2){\line(-4,1){1}}
        \put(20,0.9){\line(-3,1){1}}
        \put(17.5,2){\line(-2,1){1}}
        \put(15.6,3.2){\line(-1,1){.9}}
        \put(13.9,4.9){\line(-3,4){1}}
        \put(12.3,7){\line(-1,2){.8}}
        \put(11,10){\line(-1,3){.5}}
        \put(10.3,12.7){\line(-1,6){.2}}
        \put(10,15){\line(0,1){1}}
        \put(25,0){\line(1,0){1}}
        \put(27.5,0.2){\line(4,1){1}}
        \put(30,0.9){\line(3,1){1}}
        \put(32.5,2){\line(2,1){1}}
        \put(34.4,3.2){\line(1,1){.9}}
        \put(36.1,4.9){\line(3,4){1}}
        \put(37.7,7){\line(1,2){.8}}
        \put(39,10){\line(1,3){.5}}
        \put(39.7,12.7){\line(1,6){.2}}
        \put(40,15){\line(0,1){1}}
    \qbezier(10,40)(25,25)(10,15)
\qbezier(40,15)(25,25)(40,40)
        \put(25,55){\line(-1,0){1}}
        \put(22.6,54.8){\line(-4,-1){1}}
        \put(20,54.1){\line(-3,-1){1}}
        \put(17.5,53){\line(-2,-1){1}}
        \put(15.6,51.8){\line(-1,-1){.9}}
        \put(13.9,50.1){\line(-3,-4){1}}
        \put(12.3,48){\line(-1,-2){.8}}
        \put(11,45){\line(-1,-3){.5}}
        \put(10.3,42.3){\line(-1,-6){.2}}
        \put(10,40){\line(0,-1){1}}
        \put(25,55){\line(1,0){1}}
        \put(27.5,54.8){\line(4,-1){1}}
        \put(30,54.1){\line(3,-1){1}}
        \put(32.5,53){\line(2,-1){1}}
        \put(34.4,51.8){\line(1,-1){.9}}
        \put(36.1,50.1){\line(3,-4){1}}
        \put(37.7,48){\line(1,-2){.8}}
        \put(39,45){\line(1,-3){.5}}
        \put(39.7,42.3){\line(1,-6){.2}}
        \put(40,40){\line(0,-1){1}}
        \end{picture}
    \end{minipage}}
\newcommand{\fVvector}{\begin{minipage}{40pt}
        \begin{picture}(0,60)
\put(0,24){\vector(1,0){25}}
\end{picture}
    \end{minipage}}
\newcommand{\fVRa}{\begin{minipage}{50pt}
        \begin{picture}(50,60)
        \put(15,10){$+$}
        \put(15,40){$+$}
\put(20,0){\line(-1,0){1}}
        \put(17.6,0.2){\line(-4,1){1}}
        \put(15,0.9){\line(-3,1){1}}
        \put(12.5,2){\line(-2,1){1}}
        \put(10.6,3.2){\line(-1,1){.9}}
        \put(8.9,4.9){\line(-3,4){1}}
        \put(7.3,7){\line(-1,2){.8}}
        \put(6,10){\line(-1,3){.5}}
        \put(5.3,12.7){\line(-1,6){.2}}
        \put(5,15){\line(0,1){1}}
        \put(20,0){\line(1,0){1}}
        \put(22.5,0.2){\line(4,1){1}}
        \put(25,0.9){\line(3,1){1}}
        \put(27.5,2){\line(2,1){1}}
        \put(29.4,3.2){\line(1,1){.9}}
        \put(31.1,4.9){\line(3,4){1}}
        \put(32.7,7){\line(1,2){.8}}
        \put(34,10){\line(1,3){.5}}
        \put(34.7,12.7){\line(1,6){.2}}
        \put(35,15){\line(0,1){1}}
\qbezier(5,40)(20,20)(35,40)
\qbezier(5,15)(20,35)(35,15)
\put(20,55){\line(-1,0){1}}
        \put(17.6,54.8){\line(-4,-1){1}}
        \put(15,54.1){\line(-3,-1){1}}
        \put(12.5,53){\line(-2,-1){1}}
        \put(10.6,51.8){\line(-1,-1){.9}}
        \put(8.9,50.1){\line(-3,-4){1}}
        \put(7.3,48){\line(-1,-2){.8}}
        \put(6,45){\line(-1,-3){.5}}
        \put(5.3,42.3){\line(-1,-6){.2}}
        \put(5,40){\line(0,-1){1}}
        \put(20,55){\line(1,0){1}}
        \put(22.5,54.8){\line(4,-1){1}}
        \put(25,54.1){\line(3,-1){1}}
        \put(27.5,53){\line(2,-1){1}}
        \put(29.4,51.8){\line(1,-1){.9}}
        \put(31.1,50.1){\line(3,-4){1}}
        \put(32.7,48){\line(1,-2){.8}}
        \put(34,45){\line(1,-3){.5}}
        \put(34.7,42.3){\line(1,-6){.2}}
        \put(35,40){\line(0,-1){1}}
\end{picture}
    \end{minipage}}
\newcommand{\fVRb}{\begin{minipage}{50pt}
        \begin{picture}(50,60)
        \put(-3,27){$t$}
        \put(15,10){$-$}
        \put(15,40){$-$}
\put(20,0){\line(-1,0){1}}
        \put(17.6,0.2){\line(-4,1){1}}
        \put(15,0.9){\line(-3,1){1}}
        \put(12.5,2){\line(-2,1){1}}
        \put(10.6,3.2){\line(-1,1){.9}}
        \put(8.9,4.9){\line(-3,4){1}}
        \put(7.3,7){\line(-1,2){.8}}
        \put(6,10){\line(-1,3){.5}}
        \put(5.3,12.7){\line(-1,6){.2}}
        \put(5,15){\line(0,1){1}}
        \put(20,0){\line(1,0){1}}
        \put(22.5,0.2){\line(4,1){1}}
        \put(25,0.9){\line(3,1){1}}
        \put(27.5,2){\line(2,1){1}}
        \put(29.4,3.2){\line(1,1){.9}}
        \put(31.1,4.9){\line(3,4){1}}
        \put(32.7,7){\line(1,2){.8}}
        \put(34,10){\line(1,3){.5}}
        \put(34.7,12.7){\line(1,6){.2}}
        \put(35,15){\line(0,1){1}}
\qbezier(5,40)(20,20)(35,40)
\qbezier(5,15)(20,35)(35,15)
\put(20,55){\line(-1,0){1}}
        \put(17.6,54.8){\line(-4,-1){1}}
        \put(15,54.1){\line(-3,-1){1}}
        \put(12.5,53){\line(-2,-1){1}}
        \put(10.6,51.8){\line(-1,-1){.9}}
        \put(8.9,50.1){\line(-3,-4){1}}
        \put(7.3,48){\line(-1,-2){.8}}
        \put(6,45){\line(-1,-3){.5}}
        \put(5.3,42.3){\line(-1,-6){.2}}
        \put(5,40){\line(0,-1){1}}
        \put(20,55){\line(1,0){1}}
        \put(22.5,54.8){\line(4,-1){1}}
        \put(25,54.1){\line(3,-1){1}}
        \put(27.5,53){\line(2,-1){1}}
        \put(29.4,51.8){\line(1,-1){.9}}
        \put(31.1,50.1){\line(3,-4){1}}
        \put(32.7,48){\line(1,-2){.8}}
        \put(34,45){\line(1,-3){.5}}
        \put(34.7,42.3){\line(1,-6){.2}}
        \put(35,40){\line(0,-1){1}}
\end{picture}
    \end{minipage}}

\begin{document}
\title[On Khovanov complexes]{On Khovanov complexes}
\author[N. Ito]{Noboru Ito}

\begin{abstract}
In this paper, we discuss a proof of the isotopy invariance of a  parametrized Khovanov link homology including  categorifications of the Jones polynomial and the Kauffman bracket polynomial though it is a known fact.  In order to present a proof easy-to-follow, we give an  explicit description of retractions and chain homotopies between complexes to induce the invariance under isotopy of links.     
\end{abstract}

\maketitle

\section{Introduction}\label{ito_5}
Khovanov \cite{khovanov1} introduced 
a collection of groups, each of which is labelled by paired integers
 for a link diagram $D$ of an oriented link $L$.    These groups are homology groups of chain complexes depending on $D$.    
Their Euler characteristics are the coefficients of a version of the Jones polynomial $V_L (q)$ of $L$.   
Concretely, for a diagram $D$ of $L$, for an index $j$, there exist chain groups   
\[
\cdots \stackrel{d^{i-2}}{\to} C^{i-1, j}(D) \stackrel{d^{i-1}}{\to} C^{i, j}(D) \stackrel{d^{i}}{\to} C^{i+1, j}(D) \stackrel{d^{i+1}}{\to} \cdots
\]
that induce homology groups $\{ H^{i, j}(D), d^i \}_{i, j \in \mathbb{Z}}$, and  
\[
V_L (q) = \sum_{i, j} (-1)^i  q^j {\rm{rank}}\, H^{i, j} (D).  
\]
The homology groups $\{ H^{i, j}(D), d^i \}_{i, j \in \mathbb{Z}}$ do not depend on choices of $D$, i.e. the homology groups are invariant under isotopy of $L$.   Proving this  invariance, we check $H^{i, j}(D)$ $\cong$ $H^{i, j}(D')$ of three types, each of which corresponds to a local replacement, called a \emph{Reidemeister move}, $D$ $\leftrightarrow$ $D'$:  

$
\begin{minipage}{30pt}
        \begin{picture}(30,30)
\qbezier(6.6,18)(0,25)(0,25)
\qbezier(0,5)(20,35)(25,18)
\qbezier(11,14)(20,5)(24.5,13)
\qbezier(24.5,13)(25.5,15.5)(25,18)
\put(5,-10){$D$}
\put(55,-10){$D'$}
        \end{picture}
    \end{minipage}$
     $\longleftrightarrow$
     $\begin{minipage}{30pt}
        \begin{picture}(30,30)
        \cbezier(0,3)(30,5)(30,25)(0,27) 
        \put(45,-10){$D$}
\put(100,-10){$D'$}
        \end{picture}
    \end{minipage}$,  
$\begin{minipage}{30pt}
        \begin{picture}(30,30)
            \qbezier(0,0)(40,15)(0,30)
            \qbezier(11,21)(1,15)(11,9)
            \qbezier(18,24)(24,27)(30,30)
            \qbezier(18,5)(24,2.5)(30,0)
        \put(110,-30){$D$}
\put(195,-30){$D'$}
        \end{picture}
    \end{minipage}$
$\longleftrightarrow$ $\secondD$, 
$\dIII$ 
$\longleftrightarrow$ 
$\DIII$.

\begin{picture}(100,10)
\end{picture}

Khovanov \cite{khovanovF}  also introduced a parametrized Frobenius algebra that 
induces not only the above mentioned homology but also  some other homologies defined by Bar-Natan \cite{bar-natan} and by Lee \cite{lee}, respectively.  

In this paper, we focus on a proof of the invariance of this parametrized chain complex.  
In particular, we present explicit chain homotopies to obtain  the invariance of this parametrized Khovanov homology.  
Although this invariance under Reidemeister moves was already proved in several ways \cite{khovanovF, naot, wehrli}, each proof is either based on non-elementary methods, such as TQFTs, representation theory, or leaves many details to readers.   
Therefore, in this paper, using (enhanced) Kauffman states, we  redefine a universal parametrized Khovanov homology of the Jones polynomial and that of the Kauffman bracket polynomial.   We give explicit chain homotopies between complexes corresponding to  Reidemeister moves  (Proposition~\ref{thm1}, Theorem~\ref{thm2},  and Theorem~\ref{thm3}).  
   
Parametrized Khovanov homology,  consisting of enhanced states,  provides advantages.  It not only is good for discussion on a $\mathbb{Z}$-homology  \cite{itobicomplex} but also gives an observation of the invariance 
on chain levels.   
As we mentioned above, it induces  an elementary proof of the invariance, by ``linear algebra", which includes a fact broadly known in the field.  

The plan is as follows.  Sec.~\ref{sec2} is a short review of the definitions of two Khovanov homologies.  One is a Khovanov homology of the Jones polynomial of oriented unframed links.  The other is a Khovanov homology of the Kauffman bracket, also called the bracket polynomial, of unoriented framed links.  
In Sec.~\ref{sec3}, we recall the ``parametrized" Khovanov homologies for both the Jones polynomial and Kauffman brackets.    
In Sec.~\ref{reide}, we give an elementary proof of the invariance under Reidemeister moves, which ensure the invariance of both parametrized homologies.  

\section{Preliminaries}\label{sec2}
\subsection{Links and framed links}
A knot is a circle that is smoothly embedded into $\mathbb{R}^3$, and a framed knot is an annulus that is smoothly embedded into $\mathbb{R}^{3}$.   
A link is $S^{1} \sqcup S^{1} \sqcup \dots \sqcup S^{1}$ that is smoothly embedded into $\mathbb{R}^3$, and a ($l$-component) framed link is $(S^{1} \times I_1) \sqcup (S^{1} \times I_2) \sqcup \dots \sqcup (S^1 \times I_l)$ that is smoothly embedded into $\mathbb{R}^3$, where each $I_i$ is $[-\epsilon, \epsilon]$ for a sufficiently small real positive number $\epsilon$.  Two links (framed links, resp.) $L_0$ and $L_1$ are isotopic (framed isotopic, resp.) if there exists an isotopy $h_t : \mathbb{R}^3$ $\to$ $\mathbb{R}^3, t \in [0, 1]$ such that $h_0={\operatorname{id}}$ and $h_1(L_0)=L_1$.  
Standard definitions of link diagrams, crossings, and positive, and negative crossings apply.    

\subsection{The Jones polynomial and the Kauffman bracket}
The Jones polynomial $V_L(q)$ is a polynomial in $\mathbb{Z}[q, q^{-1}]$, which is a link invariant for every  isotopy class $L$ of oriented unframed links.    
For a crossing of an oriented link diagram of $L$, let $L_+$, $L_-$, and $L_0$ be links defined by replacing a sufficiently small disk of a crossing with one of three disks, each of which corresponds to a figure labeled by $L_+$, $L_-$, or  $L_0$ as in Fig.~\ref{skein_triple}, respectively,  where  
the exteriors of the three disks in Fig.~\ref{skein_triple} are the same.  
\begin{figure}[b]
\includegraphics[width=5cm]{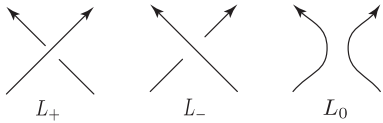}
\caption{Figures in three small disks for  $L_+$, $L_-$, and $L_0$.  
 }\label{skein_triple}
\end{figure}
Then the Jones polynomial $V_L$ is defined
by 
\begin{align*}
& V_{\text{unknot}}(q)=q+q^{-1}, \\
& q^{-2}V_{L_+}(q)-q^{2}V_{L_-}(q)=(q^{-1} - q) V_{L_0}(q).  
\end{align*}
 
In the same way as the above, for a given unoriented link diagram $D$,     
using Fig.~\ref{KB_skein}, we define three link diagrams corresponding to figures $D_\times$, $D_0$, and $D_\infty$.  
\begin{figure}[tbp]
\includegraphics[width=5cm]{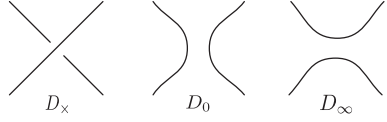}
\caption{
Figures in three small disks for $D_\times$, $D_0$, and $D_\infty$.  
 }\label{KB_skein}
\end{figure}
Then, for a link diagram $D$, 
the Kauffman bracket $\langle D \rangle \in \mathbb{Z}[A, A^{-1}]$ is defined by 
\begin{align}
&\langle~\bigcircle~\rangle = -A^2 - A^{-2}, \label{Kau1}\\
&\langle D \sqcup \ \bigcircle~\rangle = (-A^2 - A^{-2}) \langle D \rangle, \label{Kau2}\\
&\langle D_\times \rangle = A \langle D_0 \rangle + A^{-1} \langle D_\infty \rangle.  \label{Kau3}
\end{align}

Let $n_+$ ($n_-$, resp.) be the number of positive (negative, resp.) crossings, and let $w(D)=n_+ - n_-$.  It is known that for a link diagram $D$, 
\begin{equation}\label{Kau_Jones}
V_L(-A^{-2}) = (-A)^{-3w(D)} \langle D \rangle.  
\end{equation}  

Using (\ref{Kau3}), the Kauffman bracket of $D$ is expressed as a linear sum of Kauffman brackets of an arrangement of circles on a plane.  
Then, each arrangement of circles on the plane is called a {\it{state}}.  

Here, we reinterpret a {\it{state}} as a configuration of sufficiently small edges, each of which is on a crossing, called a {\it{marker}}.  We denote a state by $s$.  Then $D_0$ and $D_\infty$ are expressed by a marker on $D_{\times}$ (Fig.~\ref{marker_smooth}), i.e. $D_0$ ($D_{\infty}$, resp.)  is obtained from $D_\times$ by smoothing of a crossing along a marker as in Fig.~\ref{marker_smooth}.  We say that a marker corresponding to $D_0$ ($D_\infty$, resp.) is positive (negative, resp.).  For a state $s$, the number of positive (negative, resp.) markers is denoted by $\sigma_+(s)$ ($\sigma_-(s)$, resp.).  
\begin{figure}[b]
\includegraphics[width=5cm]{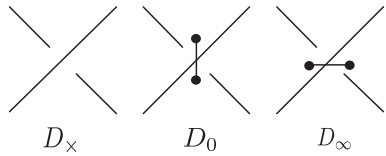}
\caption{Markers on crossings.}\label{marker_smooth}
\end{figure}
Let $\sigma(s)=\sigma_+(s)-\sigma_-(s)$ and let $|s|$ be the number of circles in $s$.  
By definition,  
\begin{equation*}\label{Kau_sum}
\langle D \rangle = \sum_s A^{\sigma(s)} (-A^{-2}-A^2)^{|s|}.    
\end{equation*}
In \cite{viro}, Viro introduced a refinement of a state by attaching signs to circles in $s$.  An {\it{enhanced state}} $S$ is a state $s$ together with a choice of $+/-$-sign for each circle.    
Each circle in an enhanced state is called a \emph{state circle}.  
Let $\tau_+(S)$ ($\tau_-(S)$, resp.) be the number of circles labeled $+$ ($-$, resp.) in $S$.  Let $\tau(S)=\tau_-(S)-\tau_+(S)$.    
Let $\sigma(S)$ $=$ $\sigma(s)$ if $S$ is an enhanced state for a state $s$.  Then, noting that $\tau(S) \equiv |s|$ (mod $2$), we have
\begin{align}\label{Kau_s}
\langle D \rangle &= \sum_s A^{\sigma(s)}(-A^2 - A^{-2})^{|s|} \nonumber \\
&=\sum_S (-1)^{\tau(S)} A^{\sigma(S)-2\tau(S)}.   
\end{align} 
For an enhanced state $S$ of an oriented diagram $D$, let $i(S)$ $=$ $(w(D)-\sigma(S))/2$ and let $j(S)$ $=$ $w(D)+i(S)+\tau(S)$ $=$ $(3w(D)-\sigma(D)+2\tau(S))/2$.  By (\ref{Kau_Jones}) and (\ref{Kau_s}), we have
\begin{equation*}\label{EulerJ}
\begin{split}
V_L(q)|_{q=-A^{-2}}&=  (-A)^{-3w(D)} \langle D \rangle \\
&=\sum_S (-1)^{-3w(D)+\tau(D)} A^{-3w(D)+\sigma(S)-2\tau(S)}\\
&=\sum_S (-1)^{w(D)+\tau(D)} (A^{-2})^{w(D) + (w(D)-\sigma(S))/2 + \tau(S)}\\
&=\sum_S (-1)^{(w(D)-\sigma(S))/2} (-A^{-2})^{w(D) + (w(D)-\sigma(S))/2 + \tau(S)}\\
&=\sum_S (-1)^{(w(D)-\sigma(S))/2} q^{w(D)+(w(D)-\sigma(S))/2 + \tau(S)}\\
&=\sum_S (-1)^{i(S)} q^{j(S)}. 
\end{split} 
\end{equation*}
Thus, each coefficient of  $V_L(q)$ is the Euler characteristic.  
\subsection{Khovanov homologies for the Jones polynomial and the Kauffman bracket}\label{review_cat}
Using \cite{viro}, we give a review of the definitions of Khovanov homologies.  
\begin{defin}[oriented enhanced states and $C^{i, j}(D)$]\label{def_cpx1}
For an enhanced state, \emph{an orientation of a state} is an ordering of the negative markers up to even permutation, where orientations that differ by odd permutations are considered opposite.  For two enhanced states, we define a relation such that one enhanced state equals the other enhanced state multiplied by $-1$ ($1$, resp.) if they are the same enhanced states but with opposite (the same, resp.) orientations.  Enhanced states with orientations are called {\it{oriented enhanced states}}.  Then, for a link diagram $D$, let $C^{i, j}(D)$ be the free abelian group generated by oriented enhanced states with $i(S)=i$ and $j(S)=j$.  
\end{defin}
We introduce a notation  \cite[Page 1215]{jacobsson} of elements of $C^{i, j}(D)$ as follows.  
\begin{notation}[a notation of elements in $C^{i, j}(D)$]\label{def_cpx2}
Let $\mathbb{L}$ be the set of crossings with negative markers of an enhanced state $S$ of a link diagram $D$, $C^{i, j}_{\mathbb{L}}(D)$ the free abelian group generated by the enhanced states having $\mathbb{L}$ with $i(S)=i$ and $j(S)=j$, and $E(\mathbb{L})$ the free abelian group generated by bijections from $\{1, 2, \dots, \sharp \mathbb{L}\}$ to $\mathbb{L}$, where $\sharp \mathbb{L}$ is the cardinality of $\mathbb{L}$.  
For $f, g \in E(\mathbb{L})$, $p(f, g)$ is $-1$ ($1$, resp.) if $f^{-1} g$ is an odd (even, resp.) permutation.  Then, let $F(\mathbb{L})=E(\mathbb{L})/(p(f, g)f-g)$.  By definition, $C^{i, j}(D)$ is identified with $\bigoplus_{\mathbb{L}, \sharp \mathbb{L}=n(i)} C^{i, j}_{ \mathbb{L}}(D) \otimes F(\mathbb{L})$, where $n(i)$ denotes $\sharp \mathbb{L}$ when we fix $i$.  
\end{notation}
Throughout this paper, we freely use the identification.    
\begin{exa}
By  Notation~\ref{def_cpx2}, for a link diagram as in the first line of Fig.~\ref{ex1}, the oriented enhanced states considered are represented as in the third line.    
\begin{figure}
\includegraphics[width=8cm]{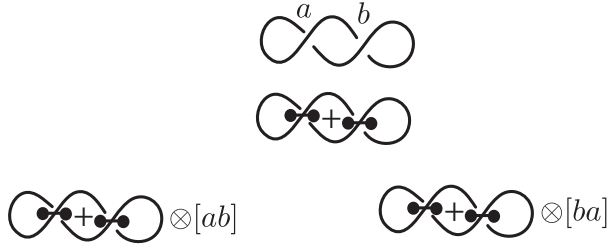}
\caption{A knot diagram with two crossings $a$ and $b$ (1st line), an enhanced state $S$ of the knot diagram (2nd line), and  
the same enhanced states with opposite orientations 
  (3rd line).}\label{ex1}
\end{figure}
The notation ``$ab$" or ``$ba$" gives an order of the crossings with negative markers.  The square bracket denotes an equivalence class including an order.    
For this example, we have $[ab]$ $=$ $-[ba]$.  
\end{exa}
\begin{defin}[coboundary operator]\label{coboundaryKh}
Let $x$ be a sequence of crossings with negative markers and $S$ and $T$ enhanced states.  Let $a$ be a crossing with a positive marker of $S$.   For $S \otimes [x]$ and $T \otimes [xa]$, if a pair $S$, $T$ appears as the figures in Fig.~\ref{boundary}, $(S : T)$ is $1$; otherwise $(S : T)$ is $0$ as in the last line of Fig.~\ref{boundary}.  The number $(S : T)$ is called the \emph{incidence number} of the pair of $S$, $T$.  
Then, a coboundary operator $d^i : C^{i, j}(D) \to C^{i+1, j}(D)$ is defined by
\begin{equation*}
d^i (S \otimes [x])= \sum_{a} (S: T)\, T\otimes [xa].  
\end{equation*}  
\begin{figure}
\includegraphics[width=6cm]{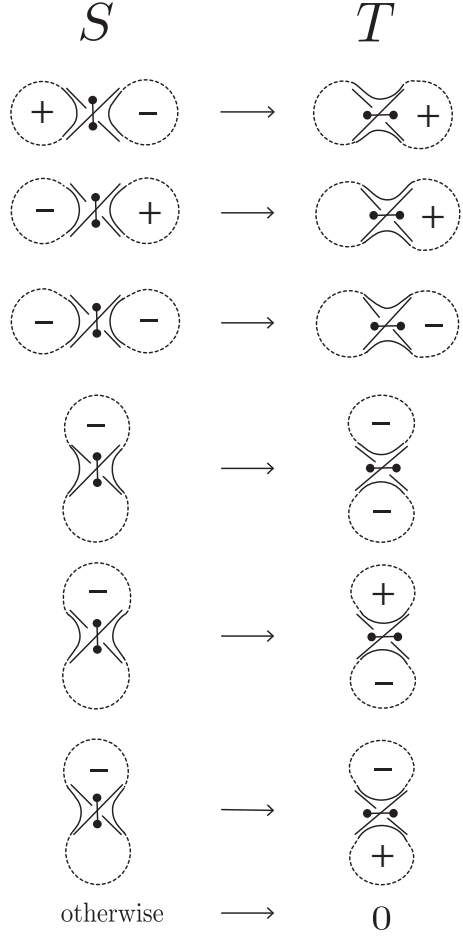}
\caption{The definition of incidence numbers for enhanced states $S$ and $T$.  The dotted arcs indicate how arcs of $S$ or $T$ are connected. 
}\label{boundary}
\end{figure}
\end{defin}
Here, recall (\ref{Kau_s}):  
\begin{align*}
\langle D \rangle &=\sum_S (-1)^{\tau(S)} A^{\sigma(S)-2\tau(S)}. \label{Kau_sb}  
\end{align*}  
Let $\I(S)$ $=$ $\tau(S)$ and let $\J(S)$ $=$ $\sigma(S)-2\tau(S)$.  
For a link diagram $D$, let $C_{\I, \J}(D)$ be the free abelian group generated by oriented enhanced states with $\I(S)=\I$ and $\J(S)=\J$.  
\begin{defin}[chain group $C_{\I, \J}(D)$]\label{cpx_dfnK}
Let $\mathbb{L}$ be the set of crossings with negative markers of an enhanced state $S$ of a link diagram $D$, and let $C_{ \I, \J, \mathbb{L}}(D)$ be the free abelian group generated by enhanced states having $\mathbb{L}$ with $\I(S)=\I$ and $\J(S)=\J$.  Then, let $C_{\I, \J}(D)=\bigoplus_{\mathbb{L}, \sharp \mathbb{L}=n(\J + 2 \I)} C_{ \I, \J, \mathbb{L}}(D) \otimes F(\mathbb{L})$ where $n(\J+ 2\I)$ denotes $\sharp \mathbb{L}$ when we fix $\J + 2\I$ (note that $\J(S) + 2\I(S) = \sigma (S)$).  
\end{defin}

\begin{defin}[boundary operator]
Let $x$ be a sequence of crossings with negative markers and $S$ and $T$ enhanced states.  Let $a$ be a crossing with a positive marker of $S$.   For $S \otimes [x]$ and $T \otimes [xa]$, if a pair $S$, $T$ appears as in Fig.~\ref{boundary}, $(S : T)$ is $1$; otherwise $(S : T)$ is $0$ as in the last line of Fig.~\ref{boundary}.  The number $(S : T)$ is called the \emph{incidence number} of the pair of $S$, $T$.  
Then, a boundary operator $\partial_{\I} : C_{\I, \J}(D) \to C_{\I-1, \J}$ is defined by
\begin{equation*}
\partial_{\I} (S \otimes [x])=\sum_{a} (S : T)\, T \otimes [xa].  
\end{equation*}
\end{defin}

Thanks to Khovanov \cite{khovanov1}, we have Fact~\ref{factKhovanov}.  
\begin{fact}[Khovanov]\label{factKhovanov}
The homology group $H^{i, j}(D)$ $(=H^i(C^{*, j}(D)))$ of a diagram $D$ of an oriented link $L$ is a link invariant, thus, it can be denoted by $H^{i, j}(L)$, such that
\begin{equation*}
V_L(q)=\sum_{i, j} (-1)^{i}q^{j} {\rm{rank}}\, H^{i, j}(L).  
\end{equation*}
\end{fact}
\begin{fact}[Khovanov, another grading of Viro \cite{viro}]
The homology group $H_{\I, \J}(D)$  $(=H_\I(C_{*, \J}(D)))$ for a diagram $D$ of an unoriented framed link $L$ is a framed link invariant.  Thus, it can be denoted by $H_{\I, \J}(L)$ and
\begin{equation*}
\langle D \rangle=\sum_{\I, \J} (-1)^{\I} A^{\J} {\rm{rank}}\, H_{\I, \J}(L).  
\end{equation*}
\end{fact}
\section{Definition of a parametrized Khovanov homology}\label{sec3}
\subsection{A parametrized Khovanov homology of the Jones polynomial}
\begin{defin}[a parametrized Khovanov homology of the Jones polynomial]
\label{def_st}
Let $D$ be an oriented link diagram and $C^{i, j} (D)$ a $\mathbb{Z}$-module of Notation~\ref{def_cpx2}.  
Let $\mathcal{C}^{i}(D)$ by $\mathcal{C}^{i}(D)$ $:=$ $\bigoplus_{j}C^{i, j}(D)$.  
Let $x$ be a sequence of crossings with negative markers, and $S$ and $T$ enhanced states.  Then, let $a$ be a crossing with a positive marker of $S$   
and let $(S : T) \in \mathbb{Z}[s, t]$.     
For $S \otimes [x]$ and $T \otimes [xa]$,  if $S$ appears in a left-hand side of an arrow of Fig.~\ref{frobenius}, $(S : T)\, T$ is defined by the right-hand side of this arrow of Fig.~\ref{frobenius}. 
Then,  a linear map $\delta^i_{s, t} : \mathcal{C}^{i}$ $\to$ $\mathcal{C}^{i+1}$ is defined by 
\begin{equation*}
\delta^i_{s, t} (S \otimes [x]) := \sum_{a} (S : T)\, T \otimes [xa].  
\end{equation*}
\end{defin}
Thanks to Khovanov \cite{khovanovF}, it is known that $\{ \mathcal{C}^i(D),  \delta^i_{s, t} \}_{i \in \mathbb{Z}}$ becomes a chain complex.  In particular, $\delta^{i+1}_{s, t} \circ \delta^{i}_{s, t}  = 0$.  
\begin{figure}
\begin{align}
\fI \label{f1} \begin{picture}(0,0) \put(-179,0){$s$} \put(-87,0){$t$}  \end{picture}\\
\fII \label{f2}    \\
\fIII \label{f3}    \\
\fIV \label{f4}    \\ 
\fVL \label{f5}  
&\fVvector
\fVRa
\quad\!
    \text{$+$}\quad\quad\! 
\fVRb    \\
\fVIL \label{f6}     
&\fVIvector
\fVIRa
\quad\!
    \text{$+$}\quad\quad\!
\fVIRb     
    \text{$-$}\quad\quad\!
\fVIRc     
    \\
\fVIIL & \fVIIR 
\nonumber
\end{align}
\caption{A generalized Frobenius calculus of signed circles.  }\label{frobenius}
\end{figure}
\begin{rem}
The case $s$ $=$ $t$ $=$ $0$
 ($s$ $=$ $0$ and $t$ $=$ $1$ with the coefficient $\mathbb{Q}$, resp.)
corresponds to the ordinary  Khovanov homology (Lee homology,~resp.).  
\end{rem}
\subsection{A parametrized Khovanov homology of the Kauffman bracket}
In order to define a parametrized Khovanov homology of the Kauffman bracket, for technical reasons \footnote{The Frobenius calculus of Fig.~\ref{frobenius} implies that the degree $\tau(S)$ does not always decrease by $1$.}, we need to redefine a grading of an ``unparametrized" Khovanov homology in the same manner as in the preprint version \cite{viro_a} of \cite{viro}.    

Recalling (\ref{Kau_s}): 
\begin{align*}
\langle D \rangle &=\sum_S (-1)^{\tau(S)} A^{\sigma(S)-2\tau(S)},    
\end{align*}
we switch the variable $A$ to $(-1)^{\frac{1}{2}} A$.  Then, 
\begin{align*}
\langle D \rangle &=\sum_S (-1)^{\frac{\sigma(S)}{2}} A^{\sigma(S)-2\tau(S)}.   
\end{align*}
Let $I(S)$ $=$ $\frac{\sigma(S)}{2}$ and let $J(S)$ $=$ $\sigma(S)-2\tau(S)$.  
For a link diagram $D$, let $C_{I, J}(D)$ be the free abelian group generated by oriented enhanced states with $I(S)=I$ and $J(S)=J$.  An orientation of an enhanced state is exhibited by $F(\mathbb{L})$ (Definition~\ref{cpx_dfnK}).    
\begin{defin}[chain group $C_{I, J}(D)$]\label{cpx_dfnK2}
Let $\mathbb{L}$ be the set of crossings with negative markers of an enhanced state $S$ of a link diagram $D$, and let $C_{I, J, \mathbb{L}}(D)$ be the free abelian group generated by enhanced states having $\mathbb{L}$ with $I(S)=I$ and $J(S)=J$.  Then, let $C_{I, J}(D)=\bigoplus_{\mathbb{L}, \sharp \mathbb{L}=n(I)} C_{I, J, \mathbb{L}}(D) \otimes F(\mathbb{L})$ where $n(I)$ denotes $\sharp \mathbb{L}$ when we fix $I$.  
\end{defin}
If one orients the link diagram $D$, the chain groups $\mathcal{C}^{i, j}(D)$ appear, and
\begin{equation}\label{id_KJ}
C_{I, J}(D) = C^{\frac{w(D)}{2}-\I, \frac{3w(D)-\J}{2}} (D).  
\end{equation}
Under this identification, the coboundary operator $\{d^{i}\}_{i \in \mathbb{Z}}$ (Definition~\ref{coboundaryKh}) turns into a boundary operator $\{\partial_I\}_{I \in \mathbb{Z}}$.  This $\{\partial_I\}_{I \in \mathbb{Z}}$   corresponds to an ``unparametrized" case, which will be extended to Definition~\ref{generalKau}.  
\begin{defin}[a parametrized Khovanov homology of the Kauffman bracket]\label{generalKau}
Let $D$ be a link diagram and $C_{I, J}(D)$ is a $\mathbb{Z}$-module of Definition~\ref{cpx_dfnK2}.  
Let $\mathcal{C}_{I}(D)$ $:=$ $\bigoplus_{J} C_{I, J}(D)$.  
Let $x$ be a sequence of crossings with negative markers and $S$ and $T$ enhanced states.  Then, let $a$ be a crossing with a positive marker of $S$ and let $(S : T)$ be an integer.   
For $S \otimes [x]$ and $T \otimes [xa]$,  if $S$ appears in a left-hand side of an arrow of Fig.~\ref{frobenius}, $(S : T)\, T$ is defined by the right-hand side of this arrow of Fig.~\ref{frobenius}.  Then,  
a linear map $\delta_{I, s, t} : \mathcal{C}_{I}$ $\to$ $\mathcal{C}_{I-1}$ is defined by 
\begin{equation*}
\delta_{I, s, t} (S \otimes [x]) := \sum_{a} (S : T)\, T \otimes [xa].  
\end{equation*}
\end{defin}
Using the identification (\ref{id_KJ}), for a link diagram $D$, the complex $\{ \mathcal{C}^i (D), \delta^{i}_{s, t} \}_{i \in \mathbb{Z}}$ turns into the complex  $\{ \mathcal{C}_{I} (D), \delta_{I, s, t} \}_{I \in \mathbb{Z}}$.  In particular, 
$\delta_{I-1, s, t} \circ \delta_{I, s, t}$ $=$ $0$.   

\section{Invariance parametrized Khovanov homology
 }\label{reide} 

By constructions in Sec.~\ref{sec3}, below, we will show the invariance of homology groups of $\{ \mathcal{C}^i (D), \delta^{i}_{s, t} \}_{i \in \mathbb{Z}}$ since the proof for $\{ \mathcal{C}_{I} (D), \delta_{I, s, t} \}_{I \in \mathbb{Z}}$ is parallel to that of $\{ \mathcal{C}^i (D), \delta^{i}_{s, t} \}_{i \in \mathbb{Z}}$ except for replacing $i$ with $I$ by (\ref{id_KJ}).   
 
Before starting proofs, we prepare notations
and a definition.
\begin{notation}[$p : q$, $q : p$, and $m(p : q)$
\cite{jacobsson}
]\label{n1}
Let $p$ and $q$ be signs ($+$ or $-$) on circles in an enhanced state.   The coboundary map $\delta_{s, t}$ (Definition~\ref{def_st}) is denoted  by
\begin{equation}\label{dif}
\feq
\end{equation}
with symbols $p:q$ and $q:p$ 
using correspondences as in  Fig.~\ref{frobenius}.   
When it is convenient to specify that $p : q$ is of multiplication  (Fig.~\ref{frobenius}, (\ref{f1})--(\ref{f4})), we shall denote $p:q$ by $m(p : q)$.  
\end{notation}
Note also 
that   
$
\begin{minipage}{50pt}
\begin{picture}(50,50)
\put(0,10){\line(1,1){30}}
\qbezier(30,10)(30,10)(20,20)
\qbezier(10,30)(0,40)(0,40)
\linethickness{2pt}
\put(5,25){\circle*{5}}
\put(25,25){\circle*{5}}
\put(5,25){\line(1,0){20}}
\put(5,40){$p:q$}
\put(5,5){$q:p$}
\put(30,23){$a$}
\end{picture}
\end{minipage}  \otimes [xa]
 = 
\begin{minipage}{50pt}
\begin{picture}(50,50)
\put(0,10){\line(1,1){30}}
\qbezier(30,10)(30,10)(20,20)
\qbezier(10,30)(0,40)(0,40)
\linethickness{2pt}
\put(5,25){\circle*{5}}
\put(25,25){\circle*{5}}
\put(5,25){\line(1,0){20}}
\put(5,40){$q:p$}
\put(5,5){$p:q$}
\put(30,23){$a$}
\end{picture}
\end{minipage}  \otimes [xa]$, i.e. $
\begin{minipage}{40pt}
    \begin{picture}(0,60)
    \put(10,42){$p:q$}
    \qbezier(5,40)(20,25)(35,40)
    \qbezier(5,15)(20,30)(35,15)
    \put(10,10){$q:p$}
    \end{picture}
    \end{minipage}
    = 
    \begin{minipage}{40pt}
    \begin{picture}(0,60)
    \put(10,42){$q:p$}
    \qbezier(5,40)(20,25)(35,40)
    \qbezier(5,15)(20,30)(35,15)
    \put(10,10){$p:q$}
    \end{picture}
    \end{minipage}$ as in  Fig.~\ref{frobenius}.  
Throughout this paper, we freely use this property of $(p:q, q:p)$.  
\begin{notation}[$\mathcal{C}(S_1, S_2, \dots, S_n)$]\label{n2}
Let $D$ be a link diagram and $\mathcal{C}^i(D)$ a chain group as in Definition~\ref{def_st}.   By definition, $\mathcal{C}^i(D)$ is a $\mathbb{Z}$-module.  When we do not need to specify the index $i$, we denote it by $\mathcal{C}(D)$ simply.  
Then, we denote by $\mathcal{C}(S_1, S_2, \dots, S_n
)$ a sub-module spanned by $S_1,  S_2, \dots, S_n \in \mathcal{C}(D)$.   
\end{notation}
\begin{exa}
Let $p$ be a sign of a state circle.  
By Fig.~\ref{frobenius}, it is easy to see that $m(p: -)$ $=$ $p$.  It is also easy to see that after splitting a circle with the sign $p$, if the sign of a circle is $+$, the other is $p$.  
\end{exa}
\begin{defin}[Reidemeister moves]\label{reide_def}
Let $D$ be a link diagram ($\subset \mathbb{R}^2$).  
The left-twisted first Reidemeister move is a replacement of a sufficiently small disk $U_1$ $=$
$\begin{minipage}{30pt}
        \begin{picture}(30,30)
\qbezier(6.6,18)(0,25)(0,25)
\qbezier(0,5)(20,35)(25,18)
\qbezier(11,14)(20,5)(24.5,13)
\qbezier(24.5,13)(25.5,15.5)(25,18)
        \end{picture}
    \end{minipage}$
     with another one $U'_1$ $=$
     $\begin{minipage}{30pt}
        \begin{picture}(30,30)
        \cbezier(0,3)(30,5)(30,25)(0,27) 
        \end{picture}
    \end{minipage}$  
    and its inverse, where  $\mathbb{R}^2 \setminus U_1$ $=$ $\mathbb{R}^2 \setminus U'_1$.  
The second Reidemeister move is a replacement of a sufficiently small disk $U_2$ $=$ 
$\begin{minipage}{30pt}
        \begin{picture}(30,30)
            \qbezier(0,0)(40,15)(0,30)
            \qbezier(11,21)(1,15)(11,9)
            \qbezier(18,24)(24,27)(30,30)
            \qbezier(18,5)(24,2.5)(30,0)
        \end{picture}
    \end{minipage}$
with another one $U'_2$ $=$ $\secondD$  
and its inverse, where $\mathbb{R}^2 \setminus U_2$ $=$ $\mathbb{R}^2 \setminus U'_2$.
The third Reidemeister move is a replacement of a sufficiently small disk $U_3$ $=$ 
$\dIII$ with another one $U'_3$ $=$ $\DIII$ 
and its inverse, where $\mathbb{R}^2 \setminus U_3$ $=$ $\mathbb{R}^2 \setminus U'_3$.   
For an oriented link diagram $D$, each equality $\mathbb{R}^2 \setminus U_i$ $=$ $\mathbb{R}^2 \setminus U'_i$ ($i=1, 2, 3$) induces the orientation of $D'$.  
\end{defin}
One may worry about other moves.  Note that the right-twisted first Reidemeister move is generated by the left-twisted first and the second Reidemeister moves.   Note also that the local move between $\dIIIa$ and $\DIIIa$ is realized by a sequence of the second Reidemeister moves and the third Reidemeister move.   Therefore, in Sec.~\ref{1st}--Sec.~\ref{3rd}, we will show the invariances of the left-twisted first, the second, and the third Reidemeister moves as in  Definition~\ref{reide_def}.   

\subsection{The invariance under the left-twisted first Reidemeister move}\label{1st}
Let $D$ and $D'$ be two link diagrams which are related by a  left-twisted first Reidemeister move at a crossing ``$a$", and $D$ is represented by \quad$\begin{minipage}{30pt}
        \begin{picture}(30,30)
        \put(-2,13.5){$a$}
\qbezier(6.6,18)(0,25)(0,25)
\qbezier(0,5)(20,35)(25,18)
\qbezier(11,14)(20,5)(24.5,13)
\qbezier(24.5,13)(25.5,15.5)(25,18)
        \end{picture}
    \end{minipage}$ and $D'$ is represented by $\begin{minipage}{30pt}
        \begin{picture}(30,30)
        \cbezier(0,3)(30,5)(30,25)(0,27) 
        \end{picture}
    \end{minipage}$.   Let $x$ be a sequence of crossings with negative markers.  
Let $\mathcal{C}$ be a $\mathbb{Z}$-module generated by  enhanced states of type
\begin{equation}\label{rIgen}
    \rigen
     \quad (p = +, -),  
    \end{equation}
where the second term is fixed by the first term in (\ref{rIgen}).  
Let $\mathcal{C}_{\rm{contr}}$ be a $\mathbb{Z}$-module generated by enhanced states of types
\[\rigencontrA
 \quad (p = +, -), \rigencontrB \quad (p = +, -).
\]
\begin{lem}
\begin{equation}\label{decomp1}
\mathcal{C}(D) = \mathcal{C} \oplus \mathcal{C}_{\rm{contr}},   
\end{equation}
where each of $\mathcal{C}$ and $\mathcal{C}_{\rm{contr}}$ is a subcomplex of $\mathcal{C}(D)$.  
\end{lem}
Before starting the proof, we prepare a notation.  
\begin{notation}\label{noteU}
Let $p_u$ be the sign induced from a sign $p$ by changing a marker on a crossing $u$ when we apply  $\delta_{s, t}$ (\ref{dif}).  
By definition, $p_u=p$ if a new marker on $u$ cannot affect a state circle with $p$.  
For example, if $p$ $=$ $m(p : +)$, we denote $p_u$ by $(m(p : +))_u$.  
\end{notation}
\begin{proof}
Let $p_u$ be a sign as in Notation~\ref{noteU}.  
\begin{align*}\label{r1}
&\delta_{s, t} \left(~~\rigen \right) \nonumber \\
&= \sum_{u} ~ \left(~ \begin{minipage}{30pt}
        \begin{picture}(30,30)
\qbezier(6.6,17)(0,24)(0,24)
\qbezier(0,4)(20,34)(25,17)
\qbezier(11,13)(20,4)(24.5,12)
\qbezier(24.5,12)(25.5,14.5)(25,17)
{\color{black}{\put(8,19.5){\circle*{3}}
\put(8,10.5){\circle*{3}}
\put(8,11.5){\line(0,1){8}}}}
\put(-5,13.5){\text{$p_u$}}
\put(14,12.5){\text{$+$}}
        \end{picture}
    \end{minipage}\!\! \otimes [xu] 
- (m(p:+))_u \begin{minipage}{30pt}
        \begin{picture}(30,30)
\qbezier(6.6,17)(0,24)(0,24)
\qbezier(0,4)(20,34)(25,17)
\qbezier(11,13)(20,4)(24.5,12)
\qbezier(24.5,12)(25.5,14.5)(25,17)
{\color{black}{\put(8,19.5){\circle*{3}}
\put(8,10.5){\circle*{3}}
\put(8,11.5){\line(0,1){8}}}}
\put(14,12.5){\text{$-$}}
        \end{picture}
    \end{minipage} \!\! \otimes [xu] \right)
\\
&= \sum_{u} ~ \left(~ \begin{minipage}{30pt}
        \begin{picture}(30,30)
\qbezier(6.6,17)(0,24)(0,24)
\qbezier(0,4)(20,34)(25,17)
\qbezier(11,13)(20,4)(24.5,12)
\qbezier(24.5,12)(25.5,14.5)(25,17)
{\color{black}{\put(8,19.5){\circle*{3}}
\put(8,10.5){\circle*{3}}
\put(8,11.5){\line(0,1){8}}}}
\put(-5,13.5){\text{$p_u$}}
\put(14,12.5){\text{$+$}}
        \end{picture}
    \end{minipage}\!\! \otimes [xu] 
- m(p_u:+) \begin{minipage}{30pt}
        \begin{picture}(30,30)
\qbezier(6.6,17)(0,24)(0,24)
\qbezier(0,4)(20,34)(25,17)
\qbezier(11,13)(20,4)(24.5,12)
\qbezier(24.5,12)(25.5,14.5)(25,17)
{\color{black}{\put(8,19.5){\circle*{3}}
\put(8,10.5){\circle*{3}}
\put(8,11.5){\line(0,1){8}}}}
\put(14,12.5){\text{$-$}}
        \end{picture}
    \end{minipage} \!\! \otimes [xu] \right) \in \mathcal{C},
\end{align*} 
\begin{align*}
&\delta_{s, t}
\left(~
\begin{minipage}{30pt}
        \begin{picture}(30,30)
\qbezier(6.6,17)(0,24)(0,24)
\qbezier(0,4)(20,34)(25,17)
\qbezier(11,13)(20,4)(24.5,12)
\qbezier(24.5,12)(25.5,14.5)(25,17)
{\color{black}{\put(8,19.5){\circle*{3}}
\put(8,10.5){\circle*{3}}
\put(8,11.5){\line(0,1){8}}}}
\put(-3.5,12.5){\text{$p$}}
\put(14,12.5){\text{$-$}}
        \end{picture}
    \end{minipage} \!\! \otimes [x] \right) =  
    \begin{minipage}{30pt}
        \begin{picture}(30,30)
\qbezier(6.6,17)(0,24)(0,24)
\qbezier(0,4)(20,34)(25,17)
\qbezier(11,13)(20,4)(24.5,12)
\qbezier(24.5,12)(25.5,14.5)(25,17)
{\color{black}{\put(12.5,15.5){\circle*{3}}
\put(4.5,15.5){\circle*{3}}
\put(4.5,15.5){\line(1,0){8}}}}
\put(15.5,14){\text{$p$}}
        \end{picture}
    \end{minipage} \!\! \otimes [xa]
    + \! \sum_u ~~ \begin{minipage}{30pt}
        \begin{picture}(30,30)
\qbezier(6.6,17)(0,24)(0,24)
\qbezier(0,4)(20,34)(25,17)
\qbezier(11,13)(20,4)(24.5,12)
\qbezier(24.5,12)(25.5,14.5)(25,17)
{\color{black}{\put(8,19.5){\circle*{3}}
\put(8,10.5){\circle*{3}}
\put(8,11.5){\line(0,1){8}}}}
\put(-4.8,13){\text{$p_u$}}
\put(14,12.5){\text{$-$}}
        \end{picture}
    \end{minipage} \!\! \otimes [xu] \in \mathcal{C}_{\rm{contr}}
\end{align*}
(the last equality holds since $\delta_{s, t}$ is the coboundary operator, e.g.~\cite[5.3.A]{viro}), 
and
\begin{align*}
&\delta_{s, t}
\left( \begin{minipage}{30pt}
        \begin{picture}(30,30)
\qbezier(6.6,17)(0,24)(0,24)
\qbezier(0,4)(20,34)(25,17)
\qbezier(11,13)(20,4)(24.5,12)
\qbezier(24.5,12)(25.5,14.5)(25,17)
{\color{black}{\put(12.5,15.5){\circle*{3}}
\put(4.5,15.5){\circle*{3}}
\put(4.5,15.5){\line(1,0){8}}}}
\put(15.5,14){\text{$p$}}
        \end{picture}
    \end{minipage} \!\! \otimes [xa]
 \right) = \sum_u ~ \begin{minipage}{30pt}
        \begin{picture}(30,30)
\qbezier(6.6,17)(0,24)(0,24)
\qbezier(0,4)(20,34)(25,17)
\qbezier(11,13)(20,4)(24.5,12)
\qbezier(24.5,12)(25.5,14.5)(25,17)
{\color{black}{\put(12.5,15.5){\circle*{3}}
\put(4.5,15.5){\circle*{3}}
\put(4.5,15.5){\line(1,0){8}}}}
\put(15.0,15){\text{$p_u$}}
        \end{picture}
    \end{minipage} \!\! \otimes [xau] \in \mathcal{C}_{\rm{contr}}.  
\end{align*}
\end{proof}
We consider the composition 
\begin{equation*}
\mathcal{C}(D) = \mathcal{C} \oplus \mathcal{C}_{\rm{contr}} \stackrel{\rho_{1}}{\to} \mathcal{C} \stackrel{\isome_1}{\to} \mathcal{C}(D'), 
\end{equation*} 
which consists of the projection $\rho_{1} :$ 
 $\mathcal{C} \oplus \mathcal{C}_{\rm{contr}}$ 
$\to$ $\mathcal{C}$ and 
\begin{align*}
 \isome_1 : \mathcal{C} \left(~
    \rigen \right) 
    \to 
    \mathcal{C}\left(~
    \begin{minipage}{30pt}
        \begin{picture}(30,30)
        \cbezier(0,3)(30,5)(30,25)(0,27) 
        \end{picture}
    \end{minipage} \!\! \otimes [x]
    \right); \\ 
\begin{minipage}{30pt}
        \begin{picture}(30,30)
\qbezier(6.6,17)(0,24)(0,24)
\qbezier(0,4)(20,34)(25,17)
\qbezier(11,13)(20,4)(24.5,12)
\qbezier(24.5,12)(25.5,14.5)(25,17)
{\color{black}{\put(8,19.5){\circle*{3}}
\put(8,10.5){\circle*{3}}
\put(8,11.5){\line(0,1){8}}}}
\put(-3.5,12.5){\text{$p$}}
\put(14,12.5){\text{$+$}}
        \end{picture}
    \end{minipage}\!\! \otimes [x] 
- m(p:+) \begin{minipage}{30pt}
        \begin{picture}(30,30)
\qbezier(6.6,17)(0,24)(0,24)
\qbezier(0,4)(20,34)(25,17)
\qbezier(11,13)(20,4)(24.5,12)
\qbezier(24.5,12)(25.5,14.5)(25,17)
{\color{black}{\put(8,19.5){\circle*{3}}
\put(8,10.5){\circle*{3}}
\put(8,11.5){\line(0,1){8}}}}
\put(14,12.5){\text{$-$}}
        \end{picture}
    \end{minipage} \!\! \otimes [x] \mapsto~\ 
    \begin{minipage}{30pt}
        \begin{picture}(30,30)
        \cbezier(0,3)(30,5)(30,25)(0,27) 
        \put(5,12.5){\text{$p$}}
        \end{picture}
    \end{minipage} \!\! \otimes [x].
    \end{align*}        
Since $\rho_1$ is a projection, 
it is a chain map.   We also have Lemma~\ref{lemmaR1}.  
\begin{lem}\label{lemmaR1}
The linear map $\isome_1$ is a bijection and a chain map.  
\end{lem}
\begin{proof}
Recall that the second term $m(p:+) \begin{minipage}{30pt}
        \begin{picture}(30,30)
\qbezier(6.6,17)(0,24)(0,24)
\qbezier(0,4)(20,34)(25,17)
\qbezier(11,13)(20,4)(24.5,12)
\qbezier(24.5,12)(25.5,14.5)(25,17)
{\color{black}{\put(8,19.5){\circle*{3}}
\put(8,10.5){\circle*{3}}
\put(8,11.5){\line(0,1){8}}}}
\put(14,12.5){\text{$-$}}
        \end{picture}
    \end{minipage} \!\! \otimes [x]$ is fixed by the first term $\begin{minipage}{30pt}
        \begin{picture}(35,30)
\qbezier(6.6,17)(0,24)(0,24)
\qbezier(0,4)(20,34)(25,17)
\qbezier(11,13)(20,4)(24.5,12)
\qbezier(24.5,12)(25.5,14.5)(25,17)
{\color{black}{\put(8,19.5){\circle*{3}}
\put(8,10.5){\circle*{3}}
\put(8,11.5){\line(0,1){8}}}}
\put(-1.5,12.5){\text{$p$}}
\put(14,12.5){\text{$+$}}
        \end{picture}
    \end{minipage}\!\! \otimes [x]$.  Then, it implies that $\begin{minipage}{30pt}
        \begin{picture}(35,30)
\qbezier(6.6,17)(0,24)(0,24)
\qbezier(0,4)(20,34)(25,17)
\qbezier(11,13)(20,4)(24.5,12)
\qbezier(24.5,12)(25.5,14.5)(25,17)
{\color{black}{\put(8,19.5){\circle*{3}}
\put(8,10.5){\circle*{3}}
\put(8,11.5){\line(0,1){8}}}}
\put(-1.5,12.5){\text{$p$}}
\put(14,12.5){\text{$+$}}
        \end{picture}
    \end{minipage}\!\! \otimes [x]$ fixes $\begin{minipage}{30pt}
        \begin{picture}(30,30)
        \cbezier(0,3)(30,5)(30,25)(0,27) 
        \put(5,12.5){\text{$p$}}
        \end{picture}
    \end{minipage} \!\! \otimes [x]$ for each sign $p$ by $\isome_1$.  
    Conversely, $\begin{minipage}{30pt}
        \begin{picture}(30,30)
        \cbezier(0,3)(30,5)(30,25)(0,27) 
        \put(5,12.5){\text{$p$}}
        \end{picture}
    \end{minipage} \!\! \otimes [x]$ also fixes $\begin{minipage}{30pt}
        \begin{picture}(35,30)
\qbezier(6.6,17)(0,24)(0,24)
\qbezier(0,4)(20,34)(25,17)
\qbezier(11,13)(20,4)(24.5,12)
\qbezier(24.5,12)(25.5,14.5)(25,17)
{\color{black}{\put(8,19.5){\circle*{3}}
\put(8,10.5){\circle*{3}}
\put(8,11.5){\line(0,1){8}}}}
\put(-1.5,12.5){\text{$p$}}
\put(14,12.5){\text{$+$}}
        \end{picture}
    \end{minipage}\!\! \otimes [x]$, which implies that $\isome_1$ is a bijection.  
    
Next, letting $p_u$ be as in Notation~\ref{noteU}, 
\begin{align*}
\delta_{s, t} \circ \isome_1 \left ( \rigen \right) 
&= \delta_{s, t} \left( \invrigen \otimes [x] \right)\\
&= \sum_u \invrigene \otimes [xu],   
\end{align*}
and 
\begin{align*}
&\isome_1 \circ \delta_{s, t} \left (~ \rigen \right)\\ 
&= \isome_1 \left( \sum_u~~~ \rigenUe \right) \\
&= \sum_u \invrigene \otimes [xu].  
\end{align*}
\end{proof}
 \begin{lem}\label{ret1}
 $\rho_1$ is represented by 
\begin{equation*}\label{1st-ret}
\begin{split}
&\begin{minipage}{30pt}
        \begin{picture}(30,30)
        \qbezier(6.6,17)(0,24)(0,24)
\qbezier(0,4)(20,34)(25,17)
\qbezier(11,13)(20,4)(24.5,12)
\qbezier(24.5,12)(25.5,14.5)(25,17)
{\color{black}{\put(8,19.5){\circle*{3}}
\put(8,10.5){\circle*{3}}
\put(8,11.5){\line(0,1){8}}}}
\put(-3.5,12.5){\text{$p$}}
\put(14,12.5){\text{$+$}}
        \end{picture}
    \end{minipage} \!\! \otimes [x] \mapsto \ 
\begin{minipage}{30pt}
        \begin{picture}(30,30)
\qbezier(6.6,17)(0,24)(0,24)
\qbezier(0,4)(20,34)(25,17)
\qbezier(11,13)(20,4)(24.5,12)
\qbezier(24.5,12)(25.5,14.5)(25,17)
{\color{black}{\put(8,19.5){\circle*{3}}
\put(8,10.5){\circle*{3}}
\put(8,11.5){\line(0,1){8}}}}
\put(-3.5,12.5){\text{$p$}}
\put(14,12.5){\text{$+$}}
        \end{picture}
    \end{minipage}\!\! \otimes [x] 
- m(p:+) \begin{minipage}{30pt}
        \begin{picture}(30,30)
\qbezier(6.6,17)(0,24)(0,24)
\qbezier(0,4)(20,34)(25,17)
\qbezier(11,13)(20,4)(24.5,12)
\qbezier(24.5,12)(25.5,14.5)(25,17)
{\color{black}{\put(8,19.5){\circle*{3}}
\put(8,10.5){\circle*{3}}
\put(8,11.5){\line(0,1){8}}}}
\put(14,12.5){\text{$-$}}
        \end{picture}
    \end{minipage} \!\! \otimes [x], \\
&    \begin{minipage}{30pt}
        \begin{picture}(30,30)
\qbezier(6.6,17)(0,24)(0,24)
\qbezier(0,4)(20,34)(25,17)
\qbezier(11,13)(20,4)(24.5,12)
\qbezier(24.5,12)(25.5,14.5)(25,17)
{\color{black}{\put(8,19.5){\circle*{3}}
\put(8,10.5){\circle*{3}}
\put(8,11.5){\line(0,1){8}}}}
\put(-3.5,12.5){\text{$p$}}
\put(14,12.5){\text{$-$}}
        \end{picture}
    \end{minipage} \!\! \otimes [x] \mapsto \ 0, {\text{and}} ~\   
    \begin{minipage}{30pt}
        \begin{picture}(30,30)
\qbezier(6.6,17)(0,24)(0,24)
\qbezier(0,4)(20,34)(25,17)
\qbezier(11,13)(20,4)(24.5,12)
\qbezier(24.5,12)(25.5,14.5)(25,17)
{\color{black}{\put(12.5,15.5){\circle*{3}}
\put(4.5,15.5){\circle*{3}}
\put(4.5,15.5){\line(1,0){8}}}}
\put(16,14){\text{$p$}}
        \end{picture}
    \end{minipage} \!\! \otimes [xa] \mapsto \ 0.  
\end{split}
\end{equation*}
\end{lem}
\begin{prop}\label{thm1}
Let ``$\ic$" be an inclusion map and ``$\id$" be an identity map.   Let $\delta_{s, t}$ be as in Definition~\ref{def_st}.  Then,   
a homotopy $h_1$ connecting $\operatorname{in}$ $\circ$ $\rho_{1}$ to the identity, i.e. a map $h_1 : \mathcal{C}\left(~
    \begin{minipage}{30pt}
        \begin{picture}(30,30)
\qbezier(6.6,18)(0,25)(0,25)
\qbezier(0,5)(20,35)(25,18)
\qbezier(11,14)(20,5)(24.5,13)
\qbezier(24.5,13)(25.5,15.5)(25,18)
        \end{picture}
    \end{minipage}\!\!
\right)$ $\to$ 
$\mathcal{C}\left(~
    \begin{minipage}{30pt}
        \begin{picture}(30,30)
\qbezier(6.6,18)(0,25)(0,25)
\qbezier(0,5)(20,35)(25,18)
\qbezier(11,14)(20,5)(24.5,13)
\qbezier(24.5,13)(25.5,15.5)(25,18)
        \end{picture}
    \end{minipage}\!\!
\right)$ such that $\delta_{s, t}$ $\circ$ $h_{1}$ $+$  $h_{1}$ $\circ$ $\delta_{s, t}$ $=$ $\operatorname{id} - \operatorname{in} \circ \rho_{1}$, is obtained by the formulas   
\begin{equation*}\label{first-hom}
\begin{split}
\begin{minipage}{30pt}
        \begin{picture}(30,30)
\qbezier(6.6,17)(0,24)(0,24)
\qbezier(0,4)(20,34)(25,17)
\qbezier(11,13)(20,4)(24.5,12)
\qbezier(24.5,12)(25.5,14.5)(25,17)
{\color{black}{\put(12.5,15.5){\circle*{3}}
\put(4.5,15.5){\circle*{3}}
\put(4.5,15.5){\line(1,0){8}}}}
\put(16,14){\text{$p$}}
        \end{picture}
    \end{minipage} \!\! \otimes [xa] &\mapsto~ 
    \begin{minipage}{30pt}
        \begin{picture}(30,30)
\qbezier(6.6,17)(0,24)(0,24)
\qbezier(0,4)(20,34)(25,17)
\qbezier(11,13)(20,4)(24.5,12)
\qbezier(24.5,12)(25.5,14.5)(25,17)
{\color{black}{\put(8,19.5){\circle*{3}}
\put(8,10.5){\circle*{3}}
\put(8,11.5){\line(0,1){8}}}}
\put(-3.5,12.5){\text{$p$}}
\put(14,12.5){\text{$-$}}
        \end{picture}
    \end{minipage} \!\! \otimes [x],~
    {\text{otherwise}} \mapsto 0.  
\end{split}
\end{equation*}
\end{prop}
\begin{proof}
Based on this section as above, what to prove is that
 $\delta_{s, t}$ $\circ$ $h_{1}$ $+$  $h_{1}$ $\circ$ $\delta_{s, t}$ $=$ $\operatorname{id} - \operatorname{in} \circ \rho_{1}$.   Recall Lemma~\ref{ret1}.  

\begin{equation*}
\label{first-begin}
\begin{split}
\left(h_{1} \circ \delta_{s, t} + \delta_{s, t} \circ h_{1} \right)  \left(
~\begin{minipage}{30pt}
        \begin{picture}(30,30)
\qbezier(6.6,17)(0,24)(0,24)
\qbezier(0,4)(20,34)(25,17)
\qbezier(11,13)(20,4)(24.5,12)
\qbezier(24.5,12)(25.5,14.5)(25,17)
{\color{black}{\put(8,19.5){\circle*{3}}
\put(8,10.5){\circle*{3}}
\put(8,11.5){\line(0,1){8}}}}
\put(-3.5,12.5){\text{$p$}}
\put(14,13){\text{$q$}}
        \end{picture}
    \end{minipage} \!\! \otimes [x]\right) 
    &= h_{1}\left( m(p:q) \begin{minipage}{30pt}
        \begin{picture}(30,30)
\qbezier(6.6,17)(0,24)(0,24)
\qbezier(0,4)(20,34)(25,17)
\qbezier(11,13)(20,4)(24.5,12)
\qbezier(24.5,12)(25.5,14.5)(25,17)
{\color{black}{\put(12.5,15.5){\circle*{3}}
\put(4.5,15.5){\circle*{3}}
\put(4.5,15.5){\line(1,0){8}}}}
        \end{picture}
    \end{minipage} \!\! \otimes [xa]
      \right) \\ &= m(p:q) \begin{minipage}{30pt}
        \begin{picture}(30,30)
\qbezier(6.6,17)(0,24)(0,24)
\qbezier(0,4)(20,34)(25,17)
\qbezier(11,13)(20,4)(24.5,12)
\qbezier(24.5,12)(25.5,14.5)(25,17)
{\color{black}{\put(8,19.5){\circle*{3}}
\put(8,10.5){\circle*{3}}
\put(8,11.5){\line(0,1){8}}}}
\put(14,12.5){\text{$-$}}
        \end{picture}
    \end{minipage} \!\! \otimes [x]\\ 
    &= \left( \operatorname{id} - \operatorname{in} \circ \rho_{1} \right)\left(~
    \begin{minipage}{30pt}
        \begin{picture}(30,30)
\qbezier(6.6,17)(0,24)(0,24)
\qbezier(0,4)(20,34)(25,17)
\qbezier(11,13)(20,4)(24.5,12)
\qbezier(24.5,12)(25.5,14.5)(25,17)
{\color{black}{\put(8,19.5){\circle*{3}}
\put(8,10.5){\circle*{3}}
\put(8,11.5){\line(0,1){8}}}}
\put(-3.5,12.5){\text{$p$}}
\put(14,13){\text{$q$}}
        \end{picture}
    \end{minipage} \!\! \otimes [x]
    \right).  
    \end{split}
    \end{equation*}
        \begin{equation*}\label{1st-last}
        \begin{split}
\left(h_{1} \circ \delta_{s, t} + \delta_{s, t} \circ h_{1} \right)\left(\begin{minipage}{30pt}
        \begin{picture}(30,30)
\qbezier(6.6,17)(0,24)(0,24)
\qbezier(0,4)(20,34)(25,17)
\qbezier(11,13)(20,4)(24.5,12)
\qbezier(24.5,12)(25.5,14.5)(25,17)
{\color{black}{\put(12.5,15.5){\circle*{3}}
\put(4.5,15.5){\circle*{3}}
\put(4.5,15.5){\line(1,0){8}}}}
\put(16,14){\text{$p$}}
        \end{picture}
    \end{minipage} \!\! \otimes [xa]
\right) 
    &= \begin{minipage}{30pt}
        \begin{picture}(30,30)
\qbezier(6.6,17)(0,24)(0,24)
\qbezier(0,4)(20,34)(25,17)
\qbezier(11,13)(20,4)(24.5,12)
\qbezier(24.5,12)(25.5,14.5)(25,17)
{\color{black}{\put(12.5,15.5){\circle*{3}}
\put(4.5,15.5){\circle*{3}}
\put(4.5,15.5){\line(1,0){8}}}}
\put(16,14){\text{$p$}}
        \end{picture}
    \end{minipage} \!\! \otimes [xa]\\
    &= \left( \operatorname{id} - \operatorname{in} \circ \rho_{1} \right)\left(
    \begin{minipage}{30pt}
        \begin{picture}(30,30)
\qbezier(6.6,17)(0,24)(0,24)
\qbezier(0,4)(20,34)(25,17)
\qbezier(11,13)(20,4)(24.5,12)
\qbezier(24.5,12)(25.5,14.5)(25,17)
{\color{black}{\put(12.5,15.5){\circle*{3}}
\put(4.5,15.5){\circle*{3}}
\put(4.5,15.5){\line(1,0){8}}}}
\put(16,14){\text{$p$}}
        \end{picture}
    \end{minipage} \!\! \otimes [xa]
\right).  
\end{split}
\end{equation*}
\end{proof}
\begin{rem}
One may wish a little bit more information of two chain homotopies $h_1$ and $\widetilde{h}_1$ we mention here.    
\begin{enumerate}
\item  The composition $\mathcal{C}(D) \stackrel{\rho_{1}}{\to} \mathcal{C}  \stackrel{\isome_1}{\to} \mathcal{C}(D')  \stackrel{\isome^{-1}_1 }{\to} \mathcal{C} \stackrel{\inc}{\to} \mathcal{C}(D)$ corresponds to  $\delta_{s, t}$ $\circ$ ${h}_{1}$ $+$  ${h}_{1}$ $\circ$ $\delta_{s, t}$ $=$ $\operatorname{id} - \operatorname{in} \circ \rho_{1}$, which we have checked as above.    
\item The composition $
\mathcal{C}(D')  \stackrel{\isome^{-1}_1 }{\to} \mathcal{C} \stackrel{\inc}{\to} \mathcal{C}(D) \stackrel{\rho_{1}}{\to} \mathcal{C} \stackrel{\isome_1}{\to} \mathcal{C}(D')$ corresponds to  $\delta_{s, t}$ $\circ$ $\widetilde{h}_{1}$ $+$  $\widetilde{h}_{1}$ $\circ$ $\delta_{s, t}$ $=$ $\operatorname{id} - \isome_1 \circ \rho_1 \circ \inc \circ \isome^{-1}_1$.  However, the composition $\isome_1 \circ \rho_1 \circ \inc \circ \isome^{-1}_1$ is equal to $\operatorname{id}_{\mathcal{C}(D')}$ (it implies that we may set $\widetilde{h}_1=0$).  
\end{enumerate}
\end{rem}
\begin{rem}
The explicit formula of the homotopy map $h_{1}$ for the special case ($s$ $=$ $t$ $=$ $0$) obtained from the original Khovanov homology is given by Viro \cite[Section 5.5]{viro}.  
\end{rem}

\subsection{The invariance under the second Reidemeister move}\label{2nd}
\begin{defin}\label{dfn2}
Let $D$ and $D'$ be link diagrams related by a  second Reidemeister move.    Let $a$ and $b$ be crossings, and $D$ and $D'$ are represented by
$\begin{minipage}{30pt}
        \begin{picture}(30,30)
            \qbezier(0,0)(40,15)(0,30)
            \qbezier(11,21)(1,15)(11,9)
            \qbezier(18,24)(24,27)(30,30)
            \qbezier(18,5)(24,2.5)(30,0)
            \put(2,21){$a$}
            \put(2,5){$b$}
        \end{picture}
    \end{minipage}
$ and 
$\secondD$, respectively.       
Let $x$ be a sequence of crossings with negative markers, and let $p$ and $q$ be signs.  
Below, in order to describe formulas simply, we often omit  markers or signs when no confusion is likely to arise; any signs of state circles are allowed if they are not specified.    
Generators of $\mathcal{C}(D)$ (Definition~\ref{def_st}) are selected as follows.  
\begin{equation}\label{five}
\twoppu \otimes [x], \sxa \otimes [xa], \sxu \otimes [xb], \sxb \otimes [xb],~{\text{and}}~ \twommu \otimes [xab].  
\end{equation}
These five types are labeled by markers and signs as follows (these notation is introduced by Jacobsson \cite{jacobsson}): 
\begin{align*}
&``++" ~{\textrm{for}}~ \twoppu \otimes [x], ``-+" ~{\textrm{for}}~ \sxa \otimes [xa], ``+-, +" ~{\textrm{for}}~ \sxu \otimes [xb],\\ 
&``+-, -" ~{\textrm{for}}~ \sxb \otimes [xb], ~{\textrm{and}}~``--" ~{\textrm{for}}~ \twommu \otimes [xab]. 
\end{align*} 
For convenience, a linear map $\cdot |_{+-, -}$ is defined by   
\begin{align*}
\twopmb \otimes [xb] \mapsto \twopmb \otimes [xb],
~{\textrm{otherwise}}~\mapsto 0.  
\end{align*}
Let $\mathcal{C}$ be a $\mathbb{Z}$-module generated by  enhanced states of type
\begin{equation}\label{rIIgen}
\riigen \qquad \quad (\forall p, q \in \{ +, - \}), 
\end{equation}
where the second term is fixed by the first term in  (\ref{rIIgen}).
Let $\mathcal{C}_{\rm{contr}}$ be a $\mathbb{Z}$-module generated by enhanced states of types
\[
\twopp \otimes [x], \twopmb \otimes [xb],  \twomm \otimes [xab] \quad (\forall p, q \in \{ +, - \}).
\]
\end{defin}
\begin{lem}\label{lemmaR2}
\begin{equation*}\label{decomp2}
\mathcal{C}(D) = \mathcal{C} \oplus \mathcal{C}_{\rm{contr}},   
\end{equation*}
where each of $\mathcal{C}$ and $\mathcal{C}_{\rm{contr}}$ is a subcomplex of $\mathcal{C}(D)$.  
\end{lem}
We will prove Lemma~\ref{lemmaR2} after we prepare Notation~\ref{notation4}.  
\begin{notation}\label{notation4}
Recall that each circle in an enhanced state is called a state circle.    
Then, let $\bigcirclepp~$ be a state circle with a sign $p$ and let $\left( \begin{matrix} \bigcirclepp \\ \bigcircleqp \end{matrix}\, \right)$ be a pair of two state circles with signs $p$ and $q$.   When no confusion is likely to arise, 
$\begin{picture}(5,5) \put(3,3){\circle{10}} \put(0,1){$\lambda$} \end{picture}~$, 
$\begin{picture}(5,5) \put(3,3){\circle{10}} \put(0,1){$\mu$} \end{picture}~$, or ~$\begin{picture}(5,5) \put(3,3){\circle{10}} \put(0,1){$\nu$} \end{picture}~$
 denotes  
a linear sum $\sum_i a_i~  \epcircle$ ($a_i \in \mathbb{Z}[s, t]$).   We also say that 
$\begin{picture}(5,5) \put(3,3){\circle{10}} \put(0,1){$\mu$} \end{picture}~$ $=$ $\begin{picture}(5,5) \put(3,3){\circle{10}} \put(0,1){$\nu$} \end{picture}~$ (~$\begin{picture}(5,5) \put(3,3){\circle{10}} \put(0,1){$\mu$} \end{picture}~$ $\neq$ $\begin{picture}(5,5) \put(3,3){\circle{10}} \put(0,1){$\nu$} \end{picture}~$,~resp.) if $\mu$ and $\nu$ belongs (do not belong, resp.) to the same component.
\end{notation}
\begin{proof}
  

First,     
\begin{align}\label{r2}
\delta_{s, t} \left(~~\riigen \right) 
\!\! = \sum_{u} ~\left( \riigenb \right)
\in {\mathcal{C}},
\end{align}
here we omit signs $p_u$, $q_u$, $(p:q)_u$, etc. 
as in Notation~\ref{noteU} since it is straightforward to prove that a pair $(p:q)_u$, $(q:p)_u$ is replaced by another pair $p_u : q_u$, $q_u : p_u$.  
Below, we omit to mention the similar remark when no confusion is likely to arise.      
(\ref{r2}) implies that $\mathcal{C}$ is a subcomplex of $\mathcal{C}(D)$.    

Second, \begin{align}\label{eq_pp}
\delta_{s, t}
\left(\twopp \otimes [x] \right) = \twompb \otimes [xa] +  \sx \otimes [xb] + \twopmbd \otimes [xb] \\
+ \sum_u \twoppu \otimes [xu], \nonumber 
\end{align}
where $\nucircle$ $=$ $t\,\, \bigcirclem~$, $\bigcirclep\, - s \,\,\bigcirclem~$ using Notation~\ref{notation4} and  
if $\begin{picture}(5,5) \put(3,3){\circle{10}} \put(0,1){$\mu$} \end{picture}~$ $\neq$ $\begin{picture}(5,5) \put(3,3){\circle{10}} \put(0,1){$\nu$} \end{picture}~$, 
$\left( \begin{matrix} \begin{picture}(5,5) \put(3,3){\circle{10}} \put(0,1){$\mu$} \end{picture}~ \\ \begin{picture}(5,5) \put(3,3){\circle{10}} \put(0,1){$\nu$} \end{picture}~ \end{matrix} \right)$ $=$ $\left( \begin{matrix} \bigcirclepp \\ \nucircle \end{matrix} \right)$
By noting that the left-hand side and the third and fourth terms of the right-hand side of (\ref{eq_pp}) are in $\mathcal{C}_{\rm{contr}}$, 
\begin{equation}\label{riieq}
\twompb \otimes [xa] +  \sx \otimes [xb] = 0    \quad {\rm{on}}~\mathcal{C}.
\end{equation}
Hence, 
$\sx \otimes [xb] \in \mathcal{C}.$   
This together with the fact that $\mathcal{C}(D)$ is generated by five types of (\ref{five}) implies that $\mathcal{C}(D)$ is a direct sum of two $\mathbb{Z}$-modules $\mathcal{C} \oplus \mathcal{C}_{\rm{contr}}$.

Third, noting 
 (\ref{eq_pp}), 
    modulo the subcomplex $\mathcal{C}$, we have
\begin{align*}
\delta_{s, t}
\left(\twopp \otimes [x] \right) 
 \in \mathcal{C}_{\rm{contr}}.    
\end{align*}
We also have 
\begin{align*}
\delta_{s, t}
\left(\twopmb \otimes [xb]
\right) = \twomm \otimes [xba] + \sum_u \sxb \otimes [xbu] \in \mathcal{C}_{\rm{contr}},
\end{align*}
and
\begin{align*}
\delta_{s, t}
\left( \twomm \otimes [xab] 
\right) = \sum_u \twommu \otimes [xabu] \in \mathcal{C}_{\rm{contr}},
\end{align*}
Therefore, $\mathcal{C}_{\rm{contr}}$ is also a subcomplex of $\mathcal{C}(D)$.  


In conclusion, the decomposition of $\mathcal{C}(D)$ implies the direct sum of two subcomplexes $ \mathcal{C} \oplus \mathcal{C}_{\rm{contr}}$, which implies the statement.  
\end{proof}
We consider the composition 
\begin{equation}\label{composition2}
\mathcal{C}(D) = \mathcal{C} \oplus \mathcal{C}_{\rm{contr}} \stackrel{\rho_{2}}{\to} \mathcal{C} \stackrel{\operatorname{isom_{2}}}{\to} \mathcal{C}(D'),
\end{equation} 
which consists of the projection $\rho_2 :$ 
$\mathcal{C} \oplus \mathcal{C}_{\rm{contr}} \to \mathcal{C}$ and  
\begin{equation}\label{isom2}
\begin{split}
\isome_2 :~&\mathcal{C} \to \mathcal{C}\left(\secondD\right); \\
&\twomp \otimes [xa] + \twopmbc \otimes [xb] \mapsto (-1)^i \ \secondD
\begin{picture}(0,30)
\put(-30,0){$p$}
\put(-5,0){$q$}
\put(10,0){$\otimes [x]$,}
\end{picture} 
\end{split}
\end{equation}
where $i$ denotes the degree of $\mathcal{C}^i$.   
Since $\rho_2$ is a projection, it is a chain map.   
\begin{lem}
The linear map $\isome_2$ is a bijection and a chain map.  
\end{lem}
\begin{proof}
Recall that, by the definition (\ref{rIIgen}), 
 the second term $
\sxb
        \begin{picture}(0,0)
\put(-25,20){$p:q$}
\put(-25,-15){$q:p$}
        \end{picture} \otimes [xb]$
 is fixed by the first term $\twomp \otimes [xa]$.    
Then, it implies that $\twomp \otimes [xa]$ fixes $(-1)^i \ \secondD
\begin{picture}(30,30)
\put(-30,0){$p$}
\put(-5,0){$q$}
\put(10,0){$\otimes [x]$}
\end{picture}$ for each paired signs $p$ and $q$ by $\isome_2$.  
Conversely, $(-1)^i \ \secondD
\begin{picture}(30,30)
\put(-30,0){$p$}
\put(-5,0){$q$}
\put(10,0){$\otimes [x]$}
\end{picture}$ also fixes $\twomp \otimes [xa]$, which implies that $\isome_2$ is a bijection.  
    
Next, letting $p_u$ and $q_u$ be signs as in Notation~\ref{noteU}, 
\begin{align*}
\delta_{s, t} \circ \isome_2 \left ( \riigen \right) 
&= (-1)^i \delta_{s, t} \left(  \ \secondD \begin{picture}(30,30)
\put(-30,0){$p$}
\put(-5,0){$q$}
\put(10,0){$\otimes [x]$}
\end{picture} \right)\\
&= (-1)^i \sum_u  \ \secondD
\begin{picture}(30,30)
\put(-32,0){$p_u$}
\put(-5,0){$q_u$}
\put(10,0){$\otimes [xu]$,}
\end{picture}   
\end{align*}
and 
\begin{align*}
&\isome_2 \circ \delta_{s, t} \left (~ \riigen \right)\\ 
&= \isome_2 \left( - \sum_u~~ \twompe \otimes [xua] + \twopmbde \otimes [xub] \right) \\
&= (-1)^{i} \sum_u  \ \secondD
\begin{picture}(30,30)
\put(-32,0){$p_u$}
\put(-5,0){$q_u$}
\put(10,0){$\otimes [xu]$.} 
\end{picture} 
\end{align*}
Note that we use $(-1)^{i+1} \cdot (-1)$ $=$ $(-1)^i$ to obtain the last equality.
\end{proof}
\begin{prop}\label{ret2}   
Let $D$, $\mathcal{C}(D)$, and $\mathcal{C}$ be as in Definition~\ref{dfn2} and let $\rho_2$ be the projection as in (\ref{composition2}).  
Then, 
the linear map $\rho_2 : \mathcal{C}(D) \to \mathcal{C}$ is given by \begin{align*}
\twomp \otimes [xa] \mapsto& \twomp \otimes [xa] + \twopmbc \otimes [xb], \\
\sx \otimes [xb] \mapsto& - \twompb \otimes [xa]  - \twopmbe \otimes [xb], \\
{\rm{otherwise}} \mapsto& 0
\end{align*}
\end{prop}
\begin{proof}
Recalling the formula (\ref{riieq}), we obtain the image of $\sx \otimes [xb]$ explicitly, which implies the statement.   
\end{proof}
\begin{thm}\label{thm2}
Let ``$\ic$" be an inclusion map and ``$\id$" be an identity map.   Let $\delta_{s, t}$ be as in Definition~\ref{def_st}.  Then, 
a homotopy $h_2$ connecting $\ic \circ \rho_2$ to the identity, i.e. a map $h_2 : \mathcal{C}\left(~\second~ \right)$ $\to$ $\mathcal{C}\left(~\second~ \right)$ such that $\delta_{s, t} \circ h_2$ $+$ $h_2 \circ \delta_{s, t}$ $=$ $\id$ $-$ $\ic \circ \rho_2$, is obtained by the formulas 
\begin{align*}
\twomm \otimes [xab] &\mapsto - \twopmb \otimes [xb], 
\sx \otimes [xb] \mapsto \twopp \otimes [x], \\
{\textrm{otherwise}} &\mapsto 0.
\end{align*}
\end{thm}
\begin{proof}
Based on this section as above, what to prove is that $\delta_{s, t}$ $\circ$ $h_{2}$ $+$  $h_{2}$ $\circ$ $\delta_{s, t}$ $=$ $\operatorname{id} - \operatorname{in} \circ \rho_{2}$ in the following (A)--(C).      

\noindent(A)~The following equalities follow from just changing markers.  

\begin{equation*}
\begin{split}
\left(h_{2} \circ \delta_{s, t} + \delta_{s, t} \circ h_{2} \right) &\left(\twomp \otimes [xa]\right) =h_{2}\left(\twommb \otimes [xab]\right)\\ &=- \twopmbc \otimes [xb] =(\operatorname{id} - \operatorname{in} \circ \rho_{2})\left(\twomp \otimes [xa]\right), 
\end{split}
\end{equation*}

\begin{equation*}\label{h1}
\begin{split}
&\left(h_{2} \circ \delta_{s, t} + \delta_{s, t} \circ h_{2} \right) \left(\twopp \otimes [x] \right) =h_{2}\left(\delta_{s, t} \left(\twopp \otimes [x] \right) \right)  \\ 
&= h_2 \left(\sx \otimes [xb] \right) = \twopp \otimes [x] =(\operatorname{id} - \operatorname{in} \circ \rho_{2})\left(\twopp \otimes [x]\right).  
\end{split}
\end{equation*}
Here, the second equality of the second formula follows from (\ref{eq_pp}) and the definitions of $h_2$.

\noindent(B)~The following equalities follow from (\ref{f2})--(\ref{f4}) of Fig.~\ref{frobenius}.  

\begin{equation*}
\begin{split}
\left(h_{2} \circ \delta_{s, t} + \delta_{s, t} \circ h_{2} \right) \left(\twomm \otimes [xab]\right) &= \twomm \otimes [xab]\\
&=(\operatorname{id} - \operatorname{in} \circ \rho_{2})\left(\twomm \otimes [xab]\right), 
\end{split}
\end{equation*}
\begin{equation*}
\begin{split}
\left(h_{2} \circ \delta_{s, t} + \delta_{s, t} \circ h_{2} \right)& \left(\twopmb \otimes [xb]\right) =h_{2}\left(-\twomm \otimes [xab]\right)\\
&= \twopmb \otimes [xb] =(\operatorname{id} - \operatorname{in} \circ \rho_{2})\left(\twopmb \otimes [xb]\right).  
\end{split}
\end{equation*}


\noindent(C) The following equation is given by all the relations in Fig.~\ref{frobenius}.  
We show the case after preparing Notation~\ref{noteTilde}.  
\begin{notation}\label{noteTilde}
Let $p$, $q$, and $r$ be signs of state circles $c_p$, $c_q$, and $c_r$, respectively.       
For a Frobenius calculus (Fig.~\ref{frobenius}) sending $p:q$ to $q:p$, we denote by $\tilde{r}$ the sign such that $\tilde{r}:=r$ if $c_r \neq c_p$ and  $c_r \neq c_q$ and $\tilde{r}$ is  $p:q$ or $q:p$ otherwise.     
\end{notation}
\begin{exa}
$\twommc$ is of a case $(p, +, q)$ corresponding to $(p, q, r)$ in  Notation~\ref{noteTilde}.   
\end{exa}

Let $p_u$ and $q_u$ be signs that depend on a marker on $u$.   
\begin{align*}\label{2nd_c2}
&\left(h_{2} \circ \delta_{s, t} + \delta_{s, t} \circ h_{2} \right) \left(\sx \otimes [xb] \right) \nonumber \\
&= h_2 \left(- \twommc \otimes [xab] - \sum_u \sxc \otimes [xub] \right) + \delta_{s, t} \left( \twopp \otimes [x] \right)  \\
\nonumber \\
&=  \twopmbf \otimes [xb] - \sum_u \twoppb \otimes [xu] + \twompb \otimes [xa] +  \sx \otimes [xb] \nonumber \\ 
&\quad + \sum_u \twoppb \otimes [xu] + \delta_{s, t}\left(\twopp \otimes [x] \right) \verline_{~+-, -} \nonumber \\ 
\nonumber \\
&=  \twompb \otimes [xa] +  \sx \otimes [xb] + \twopmbe \otimes [xb] \nonumber \\  
\end{align*}
\begin{align*}
&=(\operatorname{id} - \operatorname{in} \circ \rho_{2})\left(\sx \otimes [xb]\right).  \nonumber
\end{align*}
Here, the third equality follows from Lemma~\ref{+--lem}.  
\end{proof}
\begin{lem}\label{+--lem}
\begin{equation}\label{+--eq}
\begin{split}
\\
\twopmbf \otimes [xb] + \delta_{s, t}\left(\twopp \otimes [x] \right) \verline_{~+-, -} = \twopmbe \otimes [xb]. 
\end{split} 
\end{equation}  
\end{lem}
\begin{proof}
For the second term of LHS of (\ref{+--eq}), recall that Notation~\ref{notation4} and (\ref{eq_pp}) implies that  
\begin{align*}
\delta_{s, t}
\left(\twopp \otimes [x] \right) \verline_{~+-, -} =  \twopmbd \otimes [xb], 
\end{align*}
where $\begin{picture}(5,5) \put(3,3){\circle{10}} \put(0,1){$\nu$} \end{picture}~$ $=$ $t\,\, \bigcirclem~$ or $\bigcirclep\, - s \,\,\bigcirclem~$ and  if $\begin{picture}(5,5) \put(3,3){\circle{10}} \put(0,1){$\mu$} \end{picture}~$ $\neq$ $\begin{picture}(5,5) \put(3,3){\circle{10}} \put(0,1){$\nu$} \end{picture}~$, 
$\left( \begin{matrix} \begin{picture}(5,5) \put(3,3){\circle{10}} \put(0,1){$\mu$} \end{picture}~ \\ \begin{picture}(5,5) \put(3,3){\circle{10}} \put(0,1){$\nu$} \end{picture}~ \end{matrix} \right)$ $=$ $\left( \begin{matrix} \bigcirclepp \\ \nucircle \end{matrix} \right)$.
Here, note that $\begin{picture}(5,5) \put(3,3){\circle{10}} \put(0,1){$\nu$} \end{picture}~$ $=$ $t\,\, \bigcirclem~$ if $q=+$ and $\begin{picture}(5,5) \put(3,3){\circle{10}} \put(0,1){$\nu$} \end{picture}~$ $=$ $\bigcirclep\,\, - s \,\,\bigcirclem~$ if $q=-$.  

First, we will check four cases: 
 $\left( \begin{matrix} \begin{picture}(5,5) \put(3,3){\circle{10}} \put(0,1){$p$} \end{picture}~ \\ \begin{picture}(5,5) \put(3,3){\circle{10}} \put(0,1){$q$} \end{picture}~ \end{matrix} \right)$ $=$ $\left( \begin{matrix} \begin{picture}(5,5) \put(3,3){\circle{10}} \put(0,1){$-$} \end{picture}~ \\ \begin{picture}(5,5) \put(3,3){\circle{10}} \put(0,1){$-$} \end{picture}~ \end{matrix} \right)$, $\left( \begin{matrix} \begin{picture}(5,5) \put(3,3){\circle{10}} \put(0,1){$+$} \end{picture}~ \\ \begin{picture}(5,5) \put(3,3){\circle{10}} \put(0,1){$-$} \end{picture}~ \end{matrix} \right)$, $\left( \begin{matrix} \begin{picture}(5,5) \put(3,3){\circle{10}} \put(0,1){$-$} \end{picture}~ \\ \begin{picture}(5,5) \put(3,3){\circle{10}} \put(0,1){$+$} \end{picture}~ \end{matrix} \right)$, or $\left( \begin{matrix} \begin{picture}(5,5) \put(3,3){\circle{10}} \put(0,1){$+$} \end{picture}~ \\ \begin{picture}(5,5) \put(3,3){\circle{10}} \put(0,1){$+$} \end{picture}~ \end{matrix} \right)$.

\noindent$\bullet$ Case $\left( \begin{matrix} \begin{picture}(5,5) \put(3,3){\circle{10}} \put(0,1){$p$} \end{picture}~ \\ \begin{picture}(5,5) \put(3,3){\circle{10}} \put(0,1){$q$} \end{picture}~ \end{matrix} \right)$ $=$ $\left( \begin{matrix} \begin{picture}(5,5) \put(3,3){\circle{10}} \put(0,1){$-$} \end{picture}~ \\ \begin{picture}(5,5) \put(3,3){\circle{10}} \put(0,1){$-$} \end{picture}~ \end{matrix} \right)$.  
\begin{align*}
\text{LHS} &= \twopmbga \otimes [xb] + \delta_{s, t}\left(\twoppe \otimes [x]\right)\verline_{~+-, -} \nonumber \\
    &= \twopmbm \otimes [xb] + \twopmbl \otimes [xb] - s \twopmbk \otimes [xb]  \nonumber \\
    &= \text{RHS}. \nonumber \\
\end{align*}
\noindent$\bullet$ Case $\left( \begin{matrix} \begin{picture}(5,5) \put(3,3){\circle{10}} \put(0,1){$p$} \end{picture}~ \\ \begin{picture}(5,5) \put(3,3){\circle{10}} \put(0,1){$q$} \end{picture}~ \end{matrix} \right)$ $=$ $\left( \begin{matrix} \begin{picture}(5,5) \put(3,3){\circle{10}} \put(0,1){$+$} \end{picture}~ \\ \begin{picture}(5,5) \put(3,3){\circle{10}} \put(0,1){$-$} \end{picture}~ \end{matrix} \right)$.
\begin{align*}
\text{LHS} &= \twopmbgb \otimes [xb] + \delta_{s, t}\left(\twoppf \otimes [x]\right)\verline_{~+-, -} \nonumber \\
    &= \left( s \twopmbm \otimes [xb] + t \twopmbk \otimes [xb] \right) + \left( \twopmbn \otimes [xb] - s \twopmbm \otimes [xb] \right) \nonumber \\
    &= \text{RHS}. \nonumber \\
\end{align*}
\noindent$\bullet$ Case $\left( \begin{matrix} \begin{picture}(5,5) \put(3,3){\circle{10}} \put(0,1){$p$} \end{picture}~ \\ \begin{picture}(5,5) \put(3,3){\circle{10}} \put(0,1){$q$} \end{picture}~ \end{matrix} \right)$ $=$ $\left( \begin{matrix} \begin{picture}(5,5) \put(3,3){\circle{10}} \put(0,1){$-$} \end{picture}~ \\ \begin{picture}(5,5) \put(3,3){\circle{10}} \put(0,1){$+$} \end{picture}~ \end{matrix} \right)$.  
\begin{align*}
\text{LHS} &= \twopmbgc \otimes [xb] + \delta_{s, t}\left(\twoppg \otimes [x]\right)\verline_{~+-, -} \nonumber \\
    &= \twopmbn \otimes [xb] + t \twopmbk \otimes [xb]   \nonumber \\
    &= \text{RHS}. \nonumber\\
\end{align*}
\noindent$\bullet$ Case $\left( \begin{matrix} \begin{picture}(5,5) \put(3,3){\circle{10}} \put(0,1){$p$} \end{picture}~ \\ \begin{picture}(5,5) \put(3,3){\circle{10}} \put(0,1){$q$} \end{picture}~ \end{matrix} \right)$ $=$ $\left( \begin{matrix} \begin{picture}(5,5) \put(3,3){\circle{10}} \put(0,1){$+$} \end{picture}~ \\ \begin{picture}(5,5) \put(3,3){\circle{10}} \put(0,1){$+$} \end{picture}~ \end{matrix} \right)$. 
\begin{align*}
\text{LHS} &= \twopmbgd \otimes [xb] + \delta_{s, t}\left(\twopph \otimes [x]\right)\verline_{~+-, -} \nonumber \\
    &= \left( s \twopmbn \otimes [xb] + t \twopmbl \otimes [xb] \right) + t \twopmbm \otimes [xb] \nonumber \\
    &= \text{RHS}. \nonumber
\end{align*}

Second, we will check two cases satisfying that 
$~\begin{picture}(5,5) \put(3,3){\circle{10}} \put(0,1){$p$} \end{picture}~$ $=$ $\begin{picture}(5,5) \put(3,3){\circle{10}} \put(0,1){$q$} \end{picture}~$.

\noindent$\bullet$ Case $~\begin{picture}(5,5) \put(3,3){\circle{10}} \put(0,1){$p$} \end{picture}~$ $=$ $\begin{picture}(5,5) \put(4,3){\circle{10}} \put(0,1){$-$} \end{picture}~$ ($=$ $\begin{picture}(5,5) \put(3,3){\circle{10}} \put(0,1){$q$} \end{picture}~$). 
\begin{align*}
\text{LHS} &= \twopmbh \otimes [xb] + \delta_{s, t}\left(\twoppd \otimes [x] \right)\verline_{~+-, -} \nonumber \\
    &= \twopmbi \otimes [xb] + \left( \twopmbi \otimes [xb] - s \twopmbj \otimes [xb] \right) \nonumber \\
    &= \text{RHS}. \nonumber \\
\end{align*}
\noindent$\bullet$ Case $~\begin{picture}(5,5) \put(3,3){\circle{10}} \put(0,1){$p$} \end{picture}~$ $=$ $\begin{picture}(5,5) \put(4,3){\circle{10}} \put(0,1){$+$} \end{picture}~$ ($=$ $\begin{picture}(5,5) \put(3,3){\circle{10}} \put(0,1){$q$} \end{picture}~$).
\begin{align*}
\text{LHS} &= \twopmbg \otimes [xb] + \delta_{s, t}\left(\twoppc \otimes [x] \right)\verline_{~+-, -} \nonumber \\
    &= \left( s \twopmbi \otimes [xb] + t \twopmbj \otimes [xb] \right) + t \twopmbi \otimes [xb]  \nonumber \\
    &= \text{RHS}. \nonumber \\
\end{align*}

\end{proof}

\begin{rem}
One may wish a little bit more information of two chain homotopies $h_2$ and $\widetilde{h}_2$ we mention here.
\begin{enumerate}
\item  The composition $\mathcal{C}(D) \stackrel{\rho_{2}}{\to} \mathcal{C}  \stackrel{\isome_2}{\to} \mathcal{C}(D')  \stackrel{\isome^{-1}_2 }{\to} \mathcal{C} \stackrel{\inc}{\to} \mathcal{C}(D)$ corresponds to  $\delta_{s, t}$ $\circ$ ${h}_{2}$ $+$  ${h}_{2}$ $\circ$ $\delta_{s, t}$ $=$ $\operatorname{id} - \operatorname{in} \circ \rho_{2}$, which we have checked as above.    
\item The composition $
\mathcal{C}(D')  \stackrel{\isome^{-1}_2 }{\to} \mathcal{C} \stackrel{\inc}{\to} \mathcal{C}(D) \stackrel{\rho_{2}}{\to} \mathcal{C} \stackrel{\isome_2}{\to} \mathcal{C}(D')$ corresponds to  $\delta_{s, t}$ $\circ$ $\widetilde{h}_{2}$ $+$  $\widetilde{h}_{2}$ $\circ$ $\delta_{s, t}$ $=$ $\operatorname{id} - \isome_2 \circ \rho_2  \circ \inc \circ \isome^{-1}_2$.  However, the composition $\isome_2 \circ \rho_2 \circ \inc \circ \isome^{-1}_2$ is equal to $\operatorname{id}_{\mathcal{C}(D')}$ (it implies that we may set $\widetilde{h}_2=0$).  
\end{enumerate}
\end{rem}
\begin{rem}
The explicit formula of the homotopy map $h_{2}$ for the special case ($s$ $=$ $t$ $=$ $0$) obtained from the original Khovanov homology is given by the author  \cite{ito3}.  
\end{rem}

\subsection{The invariance under the third Reidemeister move}\label{3rd}
\begin{defin}\label{dfn3}
Let $D$ and $D'$ be link diagrams related by a third Reidemeister move.    Let $a$, $b$, and $c$ be crossings, and $D$ and $D'$ are represented by $
       \begin{picture}(0,0)
            \put(4,10){$c$}
            \put(41,1){$b$}
            \put(3,-18){$a$}
        \end{picture}\dIII$ and $
\begin{picture}(0,0)
\put(39,-22){$c$}
\put(5,-13){$a$}
\put(24,5){$b$}
\end{picture}
\DIII$, respectively.  
Let $x$, $y$ be sequences of crossings with negative markers, and let $p$, $q$, and $r$ be signs.  
Below, in order to describe formulas simply, we often omit  markers or signs when no confusion is likely to arise; any signs of state circles are allowed if they are not specified.  
Generators of $\mathcal{C}(D)$ (Definition~\ref{def_st}) are selected as follows.  
\begin{equation}\label{six}
\begin{split}
\pppu \otimes [x], \Smpp \otimes [xa],  \Spmp \otimes [xb], \\ \SpmpM \otimes [xb],   \mmpu \otimes [xab],~{\text{and}}~\eem \otimes [yc].   
\end{split}
\end{equation}
These six types are labeled by markers and signs as follows: 
\begin{align*}
``+++"~{\textrm{for}}~\pppu \otimes [x], ``-++"~{\textrm{for}}~\Smpp \otimes [xa], \\
\end{align*}
\begin{align*}
``+-+, +"~{\textrm{for}}~\Spmp \otimes [xb], 
``+-+, -"~{\textrm{for}}~\SpmpM \otimes [xb], \\
``--+"~{\textrm{for}}~\mmpu \otimes [xab],~{\textrm{and}}~ ``**-"~{\textrm{for}}~\eem \otimes [yc].
\end{align*}
For convenience, 
a linear map $\cdot |_{+-+, -}$ is defined by
\begin{align*}
\pmpb \otimes [xb] &\mapsto \pmpb \otimes [xb],
~{\textrm{otherwise}}~\mapsto 0.  
\end{align*}
Let $\mathcal{C}$ be a $\mathbb{Z}$-submodule of $\mathcal{C}(D)$.  Generators are of three types.  One of them is     
\begin{equation*}\label{rIIIgen}
\mpp  \otimes [xa] + \pmpbdSnd   \otimes [xb] \quad (\forall p, q, r \in \{ +, - \}), 
\end{equation*}
where let $\tilde{r}$ be as in  
Notation~\ref{noteTilde}.  
\noindent The other two types are     
$\mmpu \otimes [xab]$ and $\eem \otimes [yc]$.  
Let $\mathcal{C}_{\rm{contr}}$ be a $\mathbb{Z}$-submodule generated by 
$\pppu \otimes [x]$ and $\SpmpM \otimes [xb]$.    
\end{defin}
\begin{rem}\label{rem1}
One thinks that the letters $a$, $b$, and $c$, representing the crossings, might be wired.  However, this is typical for the definition of an isomorphism, which is used to obtain a correspondence between $\eem$ and $\zeem$, e.g.  
\begin{equation*}
\begin{picture}(0,0)
\put(4,10){$c$}
\put(10,20){$r$}
\put(41,1){$b$}
\put(50,-15){$p$}
\put(0,-20){$a$}
\put(10,-30){$q$}
\put(14,7){\line(0,1){10}}
\put(14,7){\cb}
\put(14,17){\cb}
\put(10,-18){\line(1,0){10}}
\put(10,-18){\cb}
\put(20,-18){\cb}
\put(34,-3){\line(0,1){10}}
\put(34,-3){\cb}
\put(34,7){\cb}
\end{picture}
\dIII
  \mapsto 
  \begin{picture}(0,0)
  \put(10,-30){$q$}
  \put(50,-15){$p$}
  \put(10,20){$r$}
  \put(34,3){\line(0,1){10}}
  \put(34,3){\cb}
  \put(34,13){\cb}
  \put(34,-28){\line(0,1){10}}
  \put(34,-28){\cb}
  \put(34,-18){\cb}
  \put(10,-12){\line(1,0){10}}
  \put(10,-12){\cb}
  \put(20,-12){\cb}
  \put(43,-25){$c$}
  \put(2,-15){$a$}
  \put(24,5){$b$}
  \end{picture}
  \DIII \qquad (\forall p, q, r \in \{+, - \}).
\end{equation*}
\end{rem}
\begin{lem}
\begin{equation*}\label{decomp3}
\mathcal{C}(D) = \mathcal{C} \oplus \mathcal{C}_{\rm{contr}},   
\end{equation*}
where each of $\mathcal{C}$ and $\mathcal{C}_{\rm{contr}}$ is a subcomplex of $\mathcal{C}(D)$.  
\end{lem}
\begin{proof}

First, recalling Notation~\ref{noteU}, 
\begin{equation}\label{eq_+-+}
\begin{split}
\delta_{s, t} \left( \pmpb \otimes [xb] \right) &= \mmp \otimes [xba]+  \pmm \otimes [xbc] \\& + \sum_u \pmpbU \otimes [xbu].  
\end{split}
\end{equation}
Since the left-hand side and the last  term of the right-hand side of (\ref{eq_+-+}) are in $\mathcal{C}_{\rm{contr}}$, 
\begin{equation}\label{Riiicontreq}
 \mmp \otimes [xba] +  \pmm \otimes [xbc] = 0 \quad {\textrm{on}}~\mathcal{C}.  
\end{equation}
Thus, we may choose $\eem \!\!\!\!\!\! \otimes [yc]$ and $\mpp  \!\!\!\!\!\! \otimes [xa]$ $+$ $\pmpbdSnd    \otimes [xb]$ as generators of $\mathcal{C}$.    
Then, 
\begin{equation}\label{r3}
\begin{split}
\delta_{s, t} \left( \mpp  \otimes [xa] + \pmpbdSnd   \otimes [xb] \right) \\ 
= \quad \mpmbSnd \quad \otimes [xac] \quad + 
\pmmbSnd \quad \otimes [xbc]
 \\
+ \sum_u \left( \Smpp \!\!\!\!\!\! \otimes [xau] + \SpmpM \!\!\!\!\!\!\otimes [xbu] \right) \in \mathcal{C},
\end{split}
\end{equation}
here we omit signs $p_u$, $q_u$, $(p:q)_u$, etc.  
as in Notation~\ref{noteU} since it is straightforward to prove that a tuple $((p:q)_u, (q:p)_u, (\tilde{r})_u)$ is replaced by the  tuple $(p_u : q_u, q_u : p_u, \tilde{r_u})$.   
Below, we omit to mention the similar remark when no confusion is likely to arise.  
(\ref{r3}) implies that $\mathcal{C}$ is a subcomplex of $\mathcal{C}(D)$. 

Second, 
\begin{equation}\label{eq_+++}
\begin{split}
&\delta_{s, t} \left( \ppp \otimes [x] \right) = \quad   \mppbSnd \otimes [xa] 
 + \pmp \otimes [xb] \\ &+ \pmpbc \otimes [xb] + \ppmbSnd \otimes [xc]+ \sum_u \pppUU \otimes [xu], 
\end{split}
\end{equation}
where $\mucircle$ $=$ $t\,\, \bigcirclem~$, $\bigcirclep~ - s \,\,\bigcirclem~$ using Notation~\ref{notation4}; if ~$\begin{picture}(5,5) \put(3,3){\circle{10}} \put(0,1){$\mu$} \end{picture}~$ $\neq$ $\begin{picture}(5,5) \put(3,3){\circle{10}} \put(0,1){$\nu$} \end{picture}~$, $\left( \begin{matrix} \begin{picture}(5,5) \put(3,3){\circle{10}} \put(0,1){$\mu$} \end{picture}~ \\ \begin{picture}(5,5) \put(3,3){\circle{10}} \put(0,1){$\nu$} \end{picture}~ \end{matrix} \right)$ $=$ $\left( \begin{matrix} \mucircle \\ \bigcircleqp \end{matrix} \right)$
; and if $\begin{picture}(5,5) \put(3,3){\circle{10}} \put(0,1){$\mu$} \end{picture}~$ $\neq$ $\begin{picture}(5,5) \put(3,3){\circle{10}} \put(0,1){$\lambda$} \end{picture}~$, $\left( \begin{matrix} \begin{picture}(5,5) \put(3,3){\circle{10}} \put(0,1){$\mu$} \end{picture}~ \\ \begin{picture}(5,5) \put(3,3){\circle{10}} \put(0,1){$\lambda$} \end{picture}~ \end{matrix} \right)$ $=$ $\left( \begin{matrix} \mucircle \\ \bigcirclerp \end{matrix} \right)$.  

Third,   
since the term of the left-hand side and the third and fifth terms of the right-hand side of (\ref{eq_+++}) are in $\mathcal{C}_{\rm{contr}}$,  
\begin{equation}\label{riiieq}
\mppbSnd \otimes [xa] +  \pmp \otimes [xb] + \quad \ppmbSnd \otimes [xc] = 0  ~~{\textrm{on}}~\mathcal{C}.  
\end{equation} 
Hence, 
$\pmp \otimes [xb] \in \mathcal{C}$.
This and the list 
 (\ref{six}) imply that $\mathcal{C}(D)$ is a direct sum of two $\mathbb{Z}$-modules $\mathcal{C} \oplus \mathcal{C}_{\rm{contr}}$.   
Modulo the subcomplex $\mathcal{C}$, 
$\delta_{s, t} \left( \mmp \!\!\!\!\!\!  \otimes [xab] \right)$ $=$ $\sum_u \mmpu  \!\!\!\!\!\! \otimes [xabu] \in \mathcal{C}_{\rm{contr}}$ and  
(\ref{eq_+++}) implies   
$\delta_{s, t} \left( \ppp \otimes [x] \right) 
 \in \mathcal{C}_{\rm{contr}}$.  Further, modulo the subcomplex $\mathcal{C}$,   
(\ref{eq_+-+}) and (\ref{Riiicontreq}) imply   
$\delta_{s, t} \left( \pmpb \!\!\!\!\otimes [xb] \right)$  
$=$ $ \sum_u \SpmpM \!\!\!\!\!\! \otimes [xbu] 
\in \mathcal{C}_{\rm{contr}}$.      
Thus, $\mathcal{C}_{\rm{contr}}$ is also a subcomplex of $\mathcal{C}(D)$.   

In conclusion, the decomposition of $\mathcal{C}(D)$ implies the direct sum of two subcomplexes $ \mathcal{C} \oplus \mathcal{C}_{\rm{contr}}$ ($= \mathcal{C}(D)$). 
\end{proof}
Recall that $D'$ is represented by  $
\begin{picture}(0,0)
\put(39,-22){$c$}
\put(5,-13){$a$}
\put(24,5){$b$}
\end{picture}
\DIII$.  Let $\mathcal{C}(D')$ be a $\mathbb{Z}$-module generated by the enhanced states of $D'$.   Let $\mathcal{C}'$ be a $\mathbb{Z}$-submodule of $\mathcal{C}$ generated by three types: 
\begin{equation*}\label{rIIIgenb}
\zpmp \otimes [xb] + \zmppdSnd \otimes [xa] \quad (\forall p, q, r \in \{+, - \}), 
\end{equation*}
$
\zmmpu \otimes [xab]$, and $\zeem \otimes [yc]$.    
Let $\mathcal{C}'_{\rm{contr}}$ be a $\mathbb{Z}$-module generated by enhanced states of types
$
\zpppu \otimes [x]$ and $\zmppu \otimes [xa] 
$.
\begin{lem}
\begin{equation}\label{decomp3b}
\mathcal{C}(D') = \mathcal{C}' \oplus \mathcal{C}'_{\rm{contr}},   
\end{equation}
where each of $\mathcal{C}'$ and $\mathcal{C}'_{\operatorname{contr}}$ is a subcomplex of $\mathcal{C}(D')$.   
\end{lem}
\begin{proof}
Recall that $D$ and $D'$ are $
       \begin{picture}(0,0)
            \put(4,10){$c$}
            \put(41,1){$b$}
            \put(3,-18){$a$}
        \end{picture}\dIII$ and $
\begin{picture}(0,0)
\put(39,-22){$c$}
\put(5,-13){$a$}
\put(24,5){$b$}
\end{picture}
\DIII$, respectively.   Thus, by turning $D$ upside down, we have the diagram $D'$ except for labels $a$ and $b$.   Hence, 
this proof is given by just turning the diagram $D$ of the above case together with the exchange of the label  $a$ with $b$.  
\end{proof}
We consider the composition 
\begin{equation}\label{composition3}
\mathcal{C}(D) = \mathcal{C} \oplus \mathcal{C}_{\rm{contr}} \stackrel{\rho_{3}}{\to} \mathcal{C} \stackrel{{\rm isom_{3}}}{\to} \mathcal{C}' , 
\end{equation} 
which consists of the projection $\rho_{3} : \mathcal{C} \oplus \mathcal{C}_{\rm{contr}} \to \mathcal{C}$ and 
$\isome_3 : $ 
\begin{equation}
\label{isom3}
\begin{split}
\mpp \!\!\!\!\!\!\otimes [xa] + \pmpbdSnd \otimes [xb] &\mapsto \zpmp \otimes [xb] + \zmppdSnd \otimes [xa], \\
\eem \otimes [yc] &\mapsto \zeem \otimes [yc].  
\end{split}
\end{equation}   
\begin{lem}
The linear map $\isome_3$ is a bijection and a chain map.  
\end{lem}
\begin{proof}
For $\mpp \!\!\!\!\!\!\otimes [xa]$ $+$ $\pmpbdSnd \otimes [xb]$, $\mpp  \otimes [xa]$ fixes $\pmpbdSnd   \otimes [xb]$.      
For $\zpmp \otimes [xb]$ $+$ $\zmppdSnd \otimes [xa]$, $\zpmp \otimes [xb]$ fixes $\zmppdSnd \otimes [xa]$.   
Then, we can say that $\mpp  \otimes [xa]$ fixes $\zpmp \otimes [xb]$ for each tuple $(p, q, r)$ by $\isome_3$.  
Conversely, $\zpmp \otimes [xb]$ also fixes $\mpp  \otimes [xa]$.  
Note also that $\eem$ and $\zeem$ give a natural correspondence between enhanced states (Remark~\ref{rem1}), which implies that 
 $\isome_3$ is a bijection.  

Next, letting $\tilde{p}$, $\tilde{q}$, and $\tilde{r}$ be signs obtained by $\delta_{s, t}$, 
\begin{align*}
\delta_{s, t} \circ \isome_3 &\left( \mpp \!\!\!\!\!\!\otimes [xa] + \pmpbdSnd \otimes [xb] \right)\\ &= \delta_{s, t} \left( \zpmp \otimes [xb] + \zmppdSnd \otimes [xa] \right)\\
&=  \sum_u \left( \zpmpe \otimes [xbu] + \zmppde \otimes [xau] \right) \\
&+ \zpmmSnd \quad \otimes [xbc] \quad + \quad \zmpmSnd \otimes [xac]\\
\end{align*} 
\begin{align*}
&=  \sum_u \left( \zpmpe \otimes [xbu] + \zmppdeSnd \otimes [xau] \right) \\
&+ \zpmmSnd \quad \otimes [xbc] \quad + \quad \zmpmSnd \otimes [xac],~{\rm{and}}~\\
\isome_3 \circ \delta_{s, t} &\left(\mpp \!\!\!\!\!\!\otimes [xa] + \pmpbdSnd \otimes [xb] \right) \\
&=
\isome_3 \left( \sum_u \left(  \mppe \!\!\!\!\!\!\otimes [xau] + \pmpbde \otimes [xbu] \right) \right) \\ 
&+ \isome_3 \left( \pmmbSnd \quad \otimes [xbc] \quad + \quad \mpmbSnd \otimes [xac] \right)
\\
&=
\sum_u \left( \isome_3 \left( \mppe \!\!\!\!\!\!\otimes [xau] + \pmpbdeSnd \otimes [xbu] \right) \right) \\ 
&+ \isome_3 \left( \pmmbSnd \quad \otimes [xbc] \quad + \quad \mpmbSnd \otimes [xac] \right)
\\
&=  \sum_u \left( \zpmpe \otimes [xbu] + \zmppdeSnd \otimes [xau] \right)\\
\end{align*} 
\begin{align*}
&+ \zpmmSnd \quad \otimes [xbc] \quad + \quad \zmpmSnd \otimes [xac]
\end{align*}
\end{proof}   
\begin{prop}\label{ret3}   
Let $D$, $\mathcal{C}(D)$, and $\mathcal{C}$ be as in Definition~\ref{dfn3}.  Let $\rho_3$ be as in (\ref{composition3}) and $\widetilde{r}$ of a sign $r$ as in Notation~\ref{noteTilde}.    
Then, $\rho_3 : \mathcal{C}(D) \to \mathcal{C}$ is given by 
\begin{align}
\mpp  \otimes [xa] &\mapsto \mpp \otimes [xa] + \pmpbdSnd \otimes [xb],  \label{rho3F}\\
 \eem \otimes [x] &\mapsto \eem \otimes [x], \label{rho3S}\\
\pmp \otimes [xb] &\mapsto - \mppbSnd \otimes [xa]  - \pmpbfSnd \otimes [xb] \label{rho3T}\\
&\quad\  - \ppmbSnd \otimes [xc],  \nonumber \\
\mmp \otimes [xab] &\mapsto \pmm \otimes [xbc], \label{rho3Fo}\\
 \nonumber \\
{\text{otherwise}} &\mapsto 0. \nonumber
\end{align}
\end{prop}
\begin{proof}
In order to extend (\ref{rho3F}) and (\ref{rho3S}) to a map on $\mathcal{C}(D)$, what to check is targets of  
$\mmp \!\!\!\!\!\! \otimes [xab]$ and $\pmp \!\!\!\!\!\! \otimes [xb]$, which are given by (\ref{Riiicontreq}) and (\ref{riiieq}).    
\end{proof}     
\begin{thm}\label{thm3}
Let ``$\ic$" be an inclusion map and ``$\id$" be an identity map.   Let $\delta_{s, t}$ be as in Definition~\ref{def_st}.  Then, 
a homotopy $h_3$ connecting $\ic \circ \rho_3$ to the identity, i.e. $h_3 : \mathcal{C}\left(\dIII \right)$ $\to$ $\mathcal{C}\left(\dIII \right)$ such that $\delta_{s, t} \circ h_3$ $+$ $h_3 \circ \delta_{s, t}$ $=$ $\id-\ic \circ \rho_3$, is obtained by the formulas 
\begin{align*}
\mmp \!\!\!\!\!\!\otimes [xab] &\mapsto - \pmpb \!\!\!\!\!\!\otimes [xb], 
\pmp \!\!\!\!\!\!\otimes [xb] \mapsto \ppp \!\!\!\!\!\!\otimes [x], \\
{\rm{otherwise}} &\mapsto 0. 
\end{align*}
\end{thm}
\begin{proof}
Based on this section as above, what to prove is that 
$\delta_{s, t}$ $\circ$ $h_{3}$ $+$  $h_{3}$ $\circ$ $\delta_{s, t}$ $=$ $\operatorname{id} - \operatorname{in} \circ \rho_{3}$ in the following (A)--(C).  

\noindent(A) The following equalities follow from changing markers.  

\begin{align*}
&\left(h_{3} \circ \delta_{s, t} + \delta_{s, t} \circ h_{3} \right) \left(\eem \!\!\!\!\!\! \otimes [x] \right) = 0 =(\operatorname{id} - \operatorname{in} \circ \rho_{3})\left(\eem \!\!\!\!\!\! \otimes [x]\right), \\ 
&\left(h_{3} \circ \delta_{s, t} + \delta_{s, t} \circ h_{3} \right) \left(\mpp \otimes [xa] \right) = h_{3}\left(\mmpbSnd \otimes [xab]\right)\\
&= - \pmpbdSnd \otimes [xb] = (\operatorname{id} - \operatorname{in} \circ \rho_{3})\left(\mpp \otimes [xa]\right).  
\end{align*}
\begin{align*}\label{3rd1to1}
&(h_{3} \circ \delta_{s, t} + \delta_{s, t} \circ h_{3})\left(\ppp \otimes [x]\right)
= h_{3}\left(\delta_{s, t} \left( \ppp \otimes [x] \right) \right)\\
&= h_3 \left( \pmp \otimes [xb] \right) 
= (\operatorname{id} - \operatorname{in} \circ \rho_{3})\left(\ppp \otimes [x]\right).  
\end{align*}

\noindent(B) The following equalities follow from (\ref{f2})--(\ref{f4}) of Fig.~\ref{frobenius}.  

\begin{align*}
&\left(h_{3} \circ \delta_{s, t} + \delta_{s, t} \circ h_{3} \right) \left(\mmp \otimes [xab] \right) = \mmp \otimes [xab]\\
& - \pmm \otimes [xbc] = (\operatorname{id} - \operatorname{in} \circ \rho_{3})\left(\mmp \otimes [xab]\right), \\
&\left(h_{3} \circ \delta_{s, t} + \delta_{s, t} \circ h_{3} \right) \left(\pmpb \otimes [xb]\right) = {h_{3}}\left(\mmp \otimes [xba]\right)\\
&= \pmpb \otimes [xb] = (\operatorname{id} - \operatorname{in} \circ \rho_{3})\left(\pmpb \otimes [xb]\right).  
\end{align*}   

\noindent(C) The following equation is given by all the relations in Fig.~\ref{frobenius}.   Let $\tilde{r}$, $\widetilde{\tilde{r}}$, etc. be as in Notation~\ref{noteTilde}.  
\begin{align*}
(&h_{3} \circ \delta_{s, t} + \delta_{s, t} \circ h_{3})\left( \pmp \otimes [xb] \right)   
\end{align*}
\begin{align*}
&= \mppbSnd \otimes [xa] + \pmp \otimes [xb]  + \ppmbSnd \otimes [xc] \\ 
&+\pmpbe \otimes [xb] + \delta_{s, t}\left(\ppp \otimes [x] \right) \verline_{~+-+, -} \\
&= \mppbSnd \otimes [xa] + \pmp \otimes [xb]  + \quad \ppmbSnd \otimes [xc]\\
&+ \pmpbfSnd \otimes [xb] \qquad\qquad (\because {\textrm{Lemma}~\ref{+-+,-eq}})
\\ 
&= (\operatorname{id} - \operatorname{in} \circ \rho_{3})\left(\pmp \otimes [xb]\right).    
\end{align*}
\end{proof} 
\begin{lem}\label{+-+,-eq}
\begin{equation}\label{eqtwothree}
\!\!\!\!\!\!\!\!\!\!
\pmpbe \!\!\!\!\!\! \otimes [xb] + \delta_{s, t}\left(\ppp \otimes [x] \right) \verline_{~+-+, -} 
= \pmpbfSnd \otimes [xb]
\end{equation}
\end{lem}
\begin{proof}
We may suppose that the positive marker on the crossing $c$ of $ 
       \begin{picture}(0,0)
            \put(2,8){$c$}
            \put(41,1){$b$}
            \put(3,-18){$a$}
            \put(8,13){\cb}
\put(18,13){\cb}
\put(8,13){\line(1,0){10}}
        \end{picture}\dIII$ is fixed.  Then, forcusing on the crossings $a$ and $b$, we may rewrite (\ref{eqtwothree}) as
      
\begin{equation}\label{eq1}
\threepmbf \otimes [xb] + \delta_{s, t}\left(\threepp \otimes [x] \right) \verline_{~+-+, -} = \threepmbe \otimes [xb],
\end{equation}
where the rightmost marker of each term in (\ref{eq1}) represents the fixed marker on the crossing $c$.   
        Therefore, a proof of (\ref{eq1}) is the same as that of (\ref{+--eq}) except for exchanging $p$ with $q$.  
\end{proof}
\begin{rem}
Reader may notice that the above  arguments (A)--(C) correspond to (A)--(C) of Sec.~\ref{2nd}.  
\end{rem}
\begin{rem}
One may wish a little bit more information of two chain homotopies $h_3$ and $\tilde{h}_3$ we mention here.
\begin{enumerate}
\item  The composition $\mathcal{C}(D) \stackrel{\rho_{3}}{\to} \mathcal{C}  \stackrel{\isome_3}{\to} \mathcal{C}'  \stackrel{\isome^{-1}_3 }{\to} \mathcal{C} \stackrel{\inc}{\to} \mathcal{C}(D)$ corresponds to  $\delta_{s, t}$ $\circ$ ${h}_{3}$ $+$  ${h}_{3}$ $\circ$ $\delta_{s, t}$ $=$ $\operatorname{id} - \operatorname{in} \circ \isome_3 \circ \isome^{-1}_3 \circ \rho_{3}$, which we have checked as above.    
\item Let $D$ and $D'$ be $
       \begin{picture}(0,0)
            \put(4,10){$c$}
            \put(41,1){$b$}
            \put(3,-18){$a$}
        \end{picture}\dIII$ and $
\begin{picture}(0,0)
\put(39,-22){$c$}
\put(5,-13){$a$}
\put(24,5){$b$}
\end{picture}
\DIII$, respectively.   Then, by turning $D$ upside down, we have the diagram $D'$.   
We define a chain map $\tilde{\rho}_3$ by turning the diagram $D$ of the above case, which is the map $\mathcal{C}(D)$ $\stackrel{\rho_{3}}{\to} \mathcal{C}$  (i.e., $\tilde{\rho_{3}}$ is actually the same as $\rho_{3}$).            
Then, we have the inverse case of the above case,  that is,   
the composition $
\mathcal{C}(D') \stackrel{\tilde{\rho}_{3}}{\to} \mathcal{C}' \stackrel{\isome^{-1}_3 }{\to} \mathcal{C} \stackrel{\isome_3}{\to} \mathcal{C}' \stackrel{\inc}{\to} \mathcal{C}(D')$ corresponds to $\delta_{s, t}$ $\circ$ $\tilde{h}_{3}$ $+$  $\tilde{h}_{3}$ $\circ$ $\delta_{s, t}$ $=$ $\operatorname{id} - \inc \circ \isome^{-1}_3 \circ \isome_3 \circ \tilde{\rho}_3$.  

\end{enumerate}
\end{rem}
\begin{rem}
The explicit formula of the homotopy map $h_{3}$ for the special case ($s$ $=$ $t$ $=$ $0$) obtained from the original Khovanov homology is given by the author  \cite{ito3}.  
\end{rem}

\subsection*{Acknowledgements}
The work was partially supported by Sumitomo Foundation (Grant for Basic Science Research Projects, Project number: 160556).  
The author also thanks Mr.~Gregory Mezera for his fruitful comments on the presentation of this paper.  
The author would like to express his gratitude to Professor Jozef H. Przytycki for his kind guidance.  

This is a refined version of arXiv: 0907.2104.  The author was a Research Fellow of the Japan Society for the Promotion of Science (20$\cdot$935).  This work was partially supported by IRTG 1529, Waseda University Grants for Special Projects (2010A-863, 2015K-342), and JSPS KAKENHI Grant Number 20$\cdot$935, 23740062.  
\end{document}